\documentclass[a4paper,reqno]{amsart}

\usepackage{etoolbox}

\newtoggle{arxiv}\togglefalse{arxiv}
\newtoggle{siam}\togglefalse{siam}
\newtoggle{ams}\togglefalse{ams}
\newtoggle{wiley}\togglefalse{wiley}
\newtoggle{birk}\togglefalse{birk}

\toggletrue{arxiv}

\usepackage[utf8]{inputenc}
\usepackage[english]{babel}
\usepackage{microtype} % better automatic linebreaks
\usepackage[T1]{fontenc}

\usepackage{amsmath,amssymb,amsfonts,amsthm}
\usepackage{mathtools}
\usepackage{thmtools}

\usepackage{enumitem}
\setlist[enumerate,1]{leftmargin=*,label=\textup{(\roman{*})}}
\setlist[itemize,1]{leftmargin=15pt}

\newlist{steps}{enumerate}{1}
\setlist[steps]{label=\textup{\textbf{\arabic{*}.~Step:}}, ref=\textup{\arabic{*}.~Step}, align=left, leftmargin=0pt, itemindent=*,labelindent=0pt, labelsep=3pt, itemsep=3pt, parsep=2pt,topsep=2pt}

\usepackage{calc}

\usepackage{bm}
\usepackage{wrapfig}
\usepackage{booktabs}
\usepackage{multirow}

\usepackage{hyperref} % for links
\ifboolexpr{togl{arxiv} or togl{birk}}{
  \hypersetup{%
    colorlinks=true,
    linkcolor=blue,
    citecolor=green!60!black,
  }
}{}
\usepackage{cleveref}

\usepackage{orcidlink} % for orcid logo and link

\usepackage{upgreek}
\usepackage{pifont}

\usepackage{subcaption}

\usepackage{tikz}
\usetikzlibrary{cd}
\usetikzlibrary{calc}
\usetikzlibrary{arrows}

\usepackage{ifthen}

\graphicspath{{./}{./figs/}}

\makeatletter
\def\moverlay{\mathpalette\mov@rlay}
\def\mov@rlay#1#2{\leavevmode\vtop{%
   \baselineskip\z@skip \lineskiplimit-\maxdimen
   \ialign{\hfil$\m@th#1##$\hfil\cr#2\crcr}}}
\newcommand{\charfusion}[3][\mathord]{
    #1{\ifx#1\mathop\vphantom{#2}\fi
        \mathpalette\mov@rlay{#2\cr#3}
      }
    \ifx#1\mathop\expandafter\displaylimits\fi}
\makeatother

\newcommand{\cupdot}{\charfusion[\mathbin]{\cup}{\cdot}}
\newcommand{\bigcupdot}{\charfusion[\mathop]{\bigcup}{\cdot}}

\newcounter{i} % this counter is used for macros
\newtoks\striche % this toks is used for macros

\newcommand{\dd}{\mathrm{d}} % differential d
\newcommand{\dx}[1][x]{\mathop{\dd#1}}
\newcommand{\cl}[2][]{\overline{#2}\ifthenelse{ \equal{#1}{} }{}{^{#1}}} % topological closure
\newcommand{\conj}[1]{\overline{#1}} % complex conjugate

\newcommand{\argdot}{\boldsymbol{\cdot}}

\DeclareMathOperator{\ran}{ran}

\DeclareMathOperator{\id}{id}
\DeclareMathOperator{\supp}{supp}

\renewcommand{\div}{\operatorname{div}}
\newcommand{\grad}{\nabla}
\DeclareMathOperator{\rot}{rot}

\DeclarePairedDelimiter{\sset}{\lbrace}{\rbrace}
\DeclarePairedDelimiter{\norm}{\lVert}{\rVert}
\DeclarePairedDelimiter{\abs}{\vert}{\vert}

\DeclarePairedDelimiterX{\dset}[2]{\lbrace}{\rbrace}{#1\,\delimsize\vert\,\mathopen{} #2}
\DeclarePairedDelimiterX{\scprod}[2]{\langle}{\rangle}{#1,#2}
\DeclarePairedDelimiterX{\dualprod}[2]{\langle}{\rangle}{#1,#2}

\renewcommand{\Re}{\operatorname{Re}}

\newcommand{\adjunsymb}{\ast} % if you want to change the symbol of the adjoint just change this command

\newcommand{\adjun}[1][1]{%
  \setcounter{i}{1}%
  \striche={\adjunsymb}%
  \loop%
  \ifnum\value{i}<#1%
  \striche=\expandafter{\the\expandafter\striche\adjunsymb}%
  \stepcounter{i}%
  \repeat%
  ^{\the\striche}%
}

\newcommand{\mapping}[4]{%
  \left\{%
    \begin{array}{rcl}%
      #1 &\to & #2, \\
      #3 &\mapsto & #4
    \end{array}%
  \right.%
}

\newcommand{\hamiltonian}{\mathcal{H}}

\newcommand{\Lb}{\mathcal{L}_{\mathrm{b}}} % linear bounded operators
\newcommand{\Lp}[1]{\mathrm{L}^{#1}} % Lp spaces
\newcommand{\conC}{\mathrm{C}} % C for space of continuous functions
\newcommand{\Cc}[1][\infty]{\mathring{\mathrm{C}}^{#1}} % continuous functions with compact support

\newcommand{\Hspace}{\mathrm{H}} % sobolev spaces in L2
\newcommand{\cH}{\mathring{\Hspace}} % sobolev spaces with homogeneous boundary trace

\newcommand{\Vtau}{\mathcal{V}_{\tau}} % boundary space for tangential trace
\newcommand{\Vtaud}{\mathcal{V}_{\tau}^{\times}} % boundary trace for twisted tangential trace

\newcommand{\XH}{\mathcal{X}_{\hamiltonian}}
\newcommand{\X}{\mathcal{X}}
\newcommand{\Y}{\mathcal{Y}}
\newcommand{\U}{\mathcal{U}}

\newcommand{\boundtr}[1][]{\gamma_{0}\ifthenelse{\equal{#1}{}}{}{\big\vert_{#1}}}
\newcommand{\normaltr}[1][]{\gamma_{\nu}\ifthenelse{\equal{#1}{}}{}{\big\vert_{#1}}}
\newcommand{\tantr}[1][]{\pi_{\tau}\ifthenelse{\equal{#1}{}}{}{\big\vert_{#1}}}
\newcommand{\tanxtr}[1][]{\gamma_{\tau}\ifthenelse{\equal{#1}{}}{}{\big\vert_{#1}}}

\newcommand{\portOp}{P}

\newcommand{\N}{\mathbb{N}}
\newcommand{\R}{\mathbb{R}}
\newcommand{\C}{\mathbb{C}}

\newcommand{\K}{\mathbb{K}}

\newcommand{\fA}{\mathfrak{A}}
\newcommand{\fB}{\mathfrak{B}}

\DeclareMathOperator{\dom}{dom}
\newcommand{\ddt}{\frac{\dd}{\dd t}}

\DeclareMathOperator{\ext}{ext}

\newcommand{\sbvek}[2]{\left[\begin{smallmatrix}#1\\#2\end{smallmatrix}\right]}
\newcommand{\spvek}[2]{\left(\begin{smallmatrix}#1\\#2\end{smallmatrix}\right)}


\ifboolexpr{togl{arxiv} or togl{ams} or togl{birk}}{

\theoremstyle{plain}% default
\newtheorem{theorem}{Theorem}[section]
\newtheorem{lemma}[theorem]{Lemma}
\newtheorem{proposition}[theorem]{Proposition}
\newtheorem{corollary}[theorem]{Corollary}

\theoremstyle{definition}
\newtheorem{definition}[theorem]{Definition}

\newtheorem{assumption}[theorem]{Assumption}

\theoremstyle{remark}
\newtheorem{remark}[theorem]{Remark}

}{}

\begin{document}

\title{Analysis of coupled Maxwell-cable problems}


% MSC classification 2020
%35Q61 Maxwell equation
%37L05 General theory of infinite-dimensional dissipative dynamical systems, nonlinear semigroups, evolution equations
%78A25 electromagnetic theory (general)
%93B28  Operator-theoretic methods
%93C25 Control/observation systems in abstract spaces

\begin{abstract}
  Building on the recently published work \cite{ClemensGuReSkrepek25}, which introduces a model for the interaction between electromagnetic fields and radiating (possibly curved) cables, we analyze the qualitative properties of the resulting dynamical system.
  The model features inputs and outputs given by the currents and voltages at the cable ends, while the state comprises the corresponding distributions along the cables and the electromagnetic fields in the surrounding domain.
  We show that the autonomous dynamics (i.e., with zero input) generate a strongly continuous semigroup and establish sufficient conditions for well-posedness, meaning continuous dependence of the state and output trajectories on the inputs and initial conditions.
\end{abstract}

\author[T.~Reis]{Timo Reis\,\orcidlink{0000-0003-0721-8494}}
\address{Institut f\"ur Mathematik, Technische Universit\"at Ilmenau, Weimarer Str.\ 25, 98693, Ilmenau, Thuringia, Germany}
\email{timo.reis@tu-ilmenau.de}

\author[N.~Skrepek]{Nathanael Skrepek\,\orcidlink{0000-0002-3096-4818}}
\address{Department of Applied Mathematics, University of Twente, P.O.\ Box 217, 7500 AE, Enschede, Overijssel, The Netherlands}
\email{n.skrepek@utwente.nl}

\date{October 23, 2025}
\maketitle

\section{Introduction}

The present work builds on the mathematical model introduced in \cite{ClemensGuReSkrepek25} for describing the interaction between electromagnetic fields and radiating cables.
In this framework, the electromagnetic field is governed by Maxwell’s equations, whereas the cables (referred to synonymously as \emph{transmission lines} throughout this work) are described by the telegrapher’s equations.
The coupling between these subsystems is realized through interface conditions that link the cable currents to the magnetic field strength and the cable voltages to the electric field strength.

A particular challenge of this setting lies in its mixed-dimensional character.
The currents and voltages along the cables are functions defined on one-dimensional intervals, while the coupling interface consists of the cable surfaces, that is, two-dimensional submanifolds of $\R^3$.
The electromagnetic field dynamics, in turn, evolve in the three-dimensional exterior domain surrounding the cables.
Altogether, the resulting model comprises coupled telegrapher’s and Maxwell’s equations.

A similar mixed-dimensional coupling was used in \cite{JaSkEh23} for heat exchange, albeit without examining well-posedness.

There exists a vast body of literature in electrical engineering on this type of coupled problems (see, e.g., \cite{RuRaPaRe02,LiWaYaKaAlChFa17,Ra12,LaNuTe88,PaAb81,Agrawal1980}).
However, the simultaneous consideration of curved cables and a rigorous mathematical analysis, including solvability and well-posedness, appears to be missing.

In this article, we address these aspects from the viewpoint of infinite-dimensional linear systems theory and the theory of port-Hamiltonian systems.
We consider $k$ cables with circular cross-sections and possibly curved geometries, which interact with the surrounding electromagnetic field.
The voltage and current distributions along the cables are extended to their lateral surfaces, forming a two-dimensional interface where these quantities serve as tangential boundary data for the electric and magnetic field intensities governed by Maxwell’s equations in the exterior domain.
The inputs and outputs of the overall system are given by the boundary values of the telegrapher’s equations, i.e., by the voltages and currents at the cable ends.

Since the state of the overall system consists of functions of spatial variables, the state space is infinite-dimensional, whereas the input and output spaces are finite-dimensional, determined by linear combinations of the boundary voltages and currents.
We prove that the free (autonomous) dynamics generate a strongly continuous semigroup.
By employing the theory of system nodes \cite{St05}, we derive conditions on the input and the initial state that ensure existence of solutions.
Finally, we provide conditions on the input configuration that guarantee well-posedness, i.e., continuous dependence of the state and output on the input and the initial state.

This article is organized as follows.
After introducing the notation in the remainder of this section, we present the mathematical model in \Cref{sec:model}.
In \Cref{sec:analysis-of-port-operator}, we derive several key properties of the coupling conditions that are essential for the subsequent analysis.
\Cref{sec:coupling-operator-theory} reformulates the coupled problem in an operator-theoretic framework and characterizes all boundary conditions that render the system dissipative.
In \Cref{sec:sysnodes}, we focus on the input–output behavior of the coupled field–cable system.
After introducing a system-node formulation, we establish criteria for well-posedness and demonstrate that the resulting model fits naturally into the port-Hamiltonian framework.

\subsection*{Notation and convention}

Let $\mathcal{X}$ and $\mathcal{Y}$ be complex Hilbert spaces. The Cartesian product of $\mathcal{X}$ and $\mathcal{Y}$ is commonly represented as $\begin{bsmallmatrix}
\mathcal{X}\\ \mathcal{Y}
\end{bsmallmatrix}_\times$, and the extension to more than two spaces is straightforward.
We denote the norm in $\mathcal{X}$ as $\norm{\argdot}_{\mathcal{X}}$ and the identity mapping in $\mathcal{X}$ as $\id_\mathcal{X}$. Similarly, we use $\id_n$ for the identity mapping on $\C^{n}$.

We omit the subscripts indicating the space when the context is clear.
%The symbol  $\mathcal{X}^*$ stands for   the {\em anti-dual} of $\mathcal{X}$, that is, the space of all bounded, additive and conjugate homogeneous functionals on $\mathcal{X}$.
% Hence, the canonical duality product $\langle\argdot,\argdot\rangle_{\mathcal{X}^*,\mathcal{X}}$ (as well as the inner product $\langle\argdot,\argdot\rangle_{\mathcal{X}}$ in $\mathcal{X}$) is linear in the first argument, and antilinear in the second argument. %Further, note that
%the Riesz map $J_\mathcal{X}$, sending $x\in \mathcal{X}$ to the functional $z\mapsto\langle x,z\rangle_{\mathcal{X}}$ is a linear isometric isomorphism from $\mathcal{X}$ to $\mathcal{X}^*$.
%If the spaces are clear from context, we may skip the subindices.
Further, if not stated else, a~Hilbert space is canonically identified with its anti-dual. %Note that, in this case, $J_\mathcal{X}=\id_\mathcal{X}$.

The space of bounded linear operators from $\mathcal{X}$ to $\mathcal{Y}$ is denoted by $\Lb(\mathcal{X},\mathcal{Y})$. As usual, we abbreviate $\Lb(\mathcal{X})\coloneq \Lb(\mathcal{X},\mathcal{X})$.
The domain $\dom (A)$ of a possibly unbounded linear operator $A\colon \dom (A) \subseteq \mathcal{X}\to \mathcal{Y}$ is typically equipped with the graph norm
$\norm{x}_{\dom (A)}\coloneq \big(\norm{x}_{\mathcal{X}}^2 + \norm{Ax}_{\mathcal{Y}}^2\big)^{1/2}$. By writing $A\subset B$ for two operators, we mean that $A$ is a~restriction of $B$, and $\cl{A}$ stands for the closure of a~closable linear operator $A$.
The adjoint of a~densely defined linear operator $A\colon \dom (A) \subseteq \mathcal{X}\to \mathcal{Y}$ is $A\adjun \colon  \dom (A\adjun) \subseteq \mathcal{Y}\to \mathcal{X}$ with
\[
 \dom (A\adjun) = \dset{y\in \mathcal{Y}}
 {\exists\, z\in \mathcal{X}\text{ s.t.\ }\forall\,x\in\dom (A):\scprod{y}{Ax}_{\mathcal{Y}} = \scprod{z}{x}_{\mathcal{X}}}.
\]
The vector $z\in \mathcal{X}$ in the above set is uniquely determined by $y\in\dom (A\adjun)$, and we set $A\adjun y=z$. Note that we identify $\C^{n\times m}\cong \Lb(\C^m,\C^n)$. Together with the fact that $\C^n$ and $\C^m$ are equipped with the Euclidean inner product, this means that $A\adjun \in\C^{n\times m}$ is the conjugate transpose of $A\in\C^{m\times n}$. Likewise, $x^*$ is the conjugate transpose of $x\in\C^n\cong \C^{n\times 1}$, such that the inner product in $\C^n$ reads
\[
  \scprod{x}{y}_{\C^{n}} = y\adjun x.
\]
For $P\in\C^{n\times n}$, we write $P>0$ ($P\geq0$), if $P=P\adjun$ is positive (semi-)definite. Likewise,
$P<0$ ($P\leq0$) means that $P=P\adjun$ is negative (semi-)definite. Further, $A^\dagger\in \C^{n\times m}$ denotes the Moore-Penrose inverse of $A\in\C^{m\times n}$.

% The resolvent set of $A\colon \mathcal{X}\supset\dom (A)\to \mathcal{X}$ is denoted by $\rho(A)$, and
 %, i.e., \[\rho(A) = \dset{\lambda\in\C}{(\lambda \id - A)^{-1}\in \Lb(\mathcal{X})}.\]
%by $\C_+ = \dset{\lambda\in\C}{\Re\lambda > 0}$ we mean the open right complex half-plane.

We use the notation of the widely used textbook \cite{AdamFour03} by \textsc{Adams et al.} for Lebesgue and Sobolev spaces.
For function spaces with values in a~Hilbert space $\mathcal{X}$, we indicate the additional mark ``$;\mathcal{X}$'' after writing the domain. For instance, the Lebesgue space of $p$-integrable $\mathcal{X}$-valued functions on the domain $\Omega$ is $\Lp{p}(\Omega;\mathcal{X})$.
Note that, throughout this article, integration of $\mathcal{X}$-valued functions always has to be understood in the Bochner sense \cite{Dies77}.

\section{The model}\label{sec:model}
We now present the details of our model.
To this end, we describe the assumptions on the spatial geometry of the problem, the modeling of the cables including the assumptions on the input–output configuration, the modeling of the electromagnetic field, and the coupling between the cable and the field.

Throughout the entire article, $k$ denotes the number of cables.

\subsection{The geometry}\label{sec:Omega}

We assume that the electromagnetic field evolves within a~domain $\Omega\subseteq\R^{3}$ structured as
\begin{equation*}
  \Omega = \Omega_{0} \setminus \bigcup_{i=1}^{k} \cl{\Omega_{i}},
\end{equation*}
where $\Omega_{0} \subseteq \R^{3}$ Lipschitz domain, and the sets
$\Omega_{1},\ldots,\Omega_{k} \subseteq \R^{3}$  fulfill
\begin{align*}
  \cl{\Omega_{i}} \subseteq \Omega_{0},\quad &i=1,\ldots,k,\\
  \cl{\Omega_{i}} \cap \cl{\Omega_{j}} = \emptyset,\quad &i,j=1,\ldots,k \ \text{with}\  i\neq j
\end{align*}
%The meaning of the above domains can be summarized as follows.
All physical processes under consideration take place within $\Omega_{0}$, which therefore represents the region enclosing the entire field–cable system.
This region will be referred to as the computational domain.
No boundedness of $\Omega_{0}$ is required, and in particular, choosing $\Omega_{0} = \R^{3}$ is admissible.
The subdomains $\Omega_{1}, \ldots, \Omega_{k}$ denote the spatial regions occupied by the individual cables, as illustrated in \Cref{fig:domain}.
Each of them is assumed to have a tubular geometry, as will be specified below.
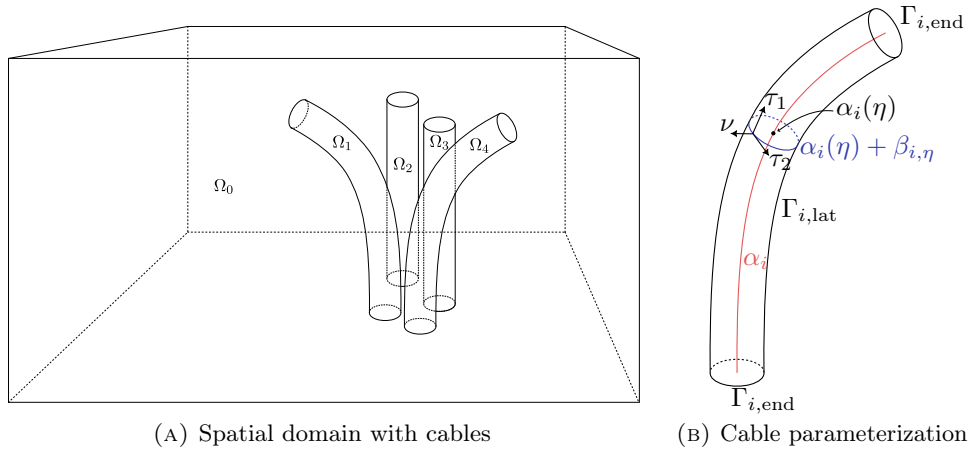
\begin{figure}[htb]
  \centering
  \begin{subfigure}[b]{0.66\textwidth}
  \centering
  \def\svgwidth{0.7\textwidth}
  
  \def\svgwidth{1\textwidth}
  %% Creator: Inkscape 1.2.1 (9c6d41e, 2022-07-14), www.inkscape.org
%% PDF/EPS/PS + LaTeX output extension by Johan Engelen, 2010
%% Accompanies image file 'domain_with_cables_trimmed.eps' (pdf, eps, ps)
%%
%% To include the image in your LaTeX document, write
%%   \input{<filename>.pdf_tex}
%%  instead of
%%   \includegraphics{<filename>.pdf}
%% To scale the image, write
%%   \def\svgwidth{<desired width>}
%%   \input{<filename>.eps_tex}
%%  instead of
%%   \includegraphics[width=<desired width>]{<filename>.pdf}
%%
%% Images with a different path to the parent latex file can
%% be accessed with the `import' package (which may need to be
%% installed) using
%%   \usepackage{import}
%% in the preamble, and then including the image with
%%   \import{<path to file>}{<filename>.pdf_tex}
%% Alternatively, one can specify
%%   \graphicspath{{<path to file>/}}
%% 
%% For more information, please see info/svg-inkscape on CTAN:
%%   http://tug.ctan.org/tex-archive/info/svg-inkscape
%%
\begingroup%
  \makeatletter%
  \providecommand\color[2][]{%
    \errmessage{(Inkscape) Color is used for the text in Inkscape, but the package 'color.sty' is not loaded}%
    \renewcommand\color[2][]{}%
  }%
  \providecommand\transparent[1]{%
    \errmessage{(Inkscape) Transparency is used (non-zero) for the text in Inkscape, but the package 'transparent.sty' is not loaded}%
    \renewcommand\transparent[1]{}%
  }%
  \providecommand\rotatebox[2]{#2}%
  \newcommand*\fsize{\dimexpr\f@size pt\relax}%
  \newcommand*\lineheight[1]{\fontsize{\fsize}{#1\fsize}\selectfont}%
  \ifx\svgwidth\undefined%
    \setlength{\unitlength}{418.63182729bp}%
    \ifx\svgscale\undefined%
      \relax%
    \else%
      \setlength{\unitlength}{\unitlength * \real{\svgscale}}%
    \fi%
  \else%
    \setlength{\unitlength}{\svgwidth}%
  \fi%
  \global\let\svgwidth\undefined%
  \global\let\svgscale\undefined%
  \makeatother%
  \begin{picture}(1,0.61858422)%
    \lineheight{1}%
    \setlength\tabcolsep{0pt}%
    \put(0,0){\includegraphics[width=\unitlength]{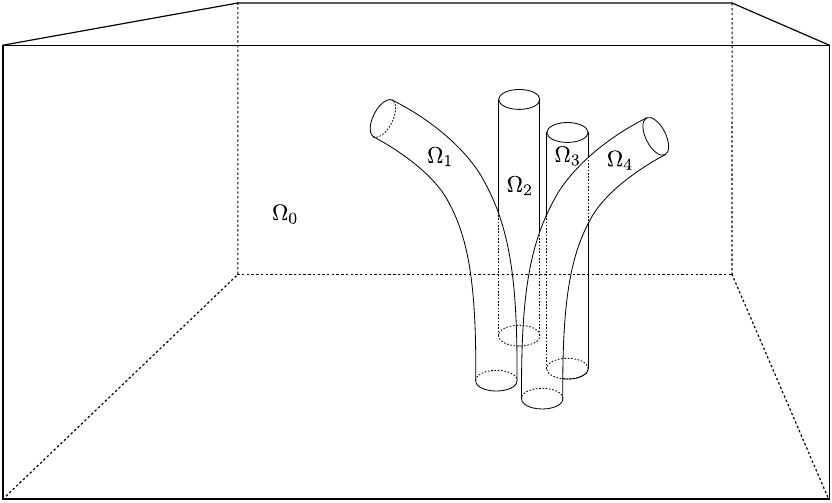}}%
  \end{picture}%
\endgroup%

  \caption{\label{fig:domain}Spatial domain with cables}
  \end{subfigure}
  \hfill
  \begin{subfigure}[b]{0.3\textwidth}
    \centering
  \centering
  \def\svgwidth{0.85\textwidth}
  %% Creator: Inkscape 1.2.2 (b0a84865, 2022-12-01), www.inkscape.org
%% PDF/EPS/PS + LaTeX output extension by Johan Engelen, 2010
%% Accompanies image file 'cable_version_6.pdf' (pdf, eps, ps)
%%
%% To include the image in your LaTeX document, write
%%   \input{<filename>.pdf_tex}
%%  instead of
%%   \includegraphics{<filename>.pdf}
%% To scale the image, write
%%   \def\svgwidth{<desired width>}
%%   \input{<filename>.pdf_tex}
%%  instead of
%%   \includegraphics[width=<desired width>]{<filename>.pdf}
%%
%% Images with a different path to the parent latex file can
%% be accessed with the `import' package (which may need to be
%% installed) using
%%   \usepackage{import}
%% in the preamble, and then including the image with
%%   \import{<path to file>}{<filename>.pdf_tex}
%% Alternatively, one can specify
%%   \graphicspath{{<path to file>/}}
%% 
%% For more information, please see info/svg-inkscape on CTAN:
%%   http://tug.ctan.org/tex-archive/info/svg-inkscape
%%
\begingroup%
  \makeatletter%
  \providecommand\color[2][]{%
    \errmessage{(Inkscape) Color is used for the text in Inkscape, but the package 'color.sty' is not loaded}%
    \renewcommand\color[2][]{}%
  }%
  \providecommand\transparent[1]{%
    \errmessage{(Inkscape) Transparency is used (non-zero) for the text in Inkscape, but the package 'transparent.sty' is not loaded}%
    \renewcommand\transparent[1]{}%
  }%
  \providecommand\rotatebox[2]{#2}%
  \newcommand*\fsize{\dimexpr\f@size pt\relax}%
  \newcommand*\lineheight[1]{\fontsize{\fsize}{#1\fsize}\selectfont}%
  \ifx\svgwidth\undefined%
    \setlength{\unitlength}{129.25983819bp}%
    \ifx\svgscale\undefined%
      \relax%
    \else%
      \setlength{\unitlength}{\unitlength * \real{\svgscale}}%
    \fi%
  \else%
    \setlength{\unitlength}{\svgwidth}%
  \fi%
  \global\let\svgwidth\undefined%
  \global\let\svgscale\undefined%
  \makeatother%
  \begin{picture}(1,1.61184216)%
    \lineheight{1}%
    \setlength\tabcolsep{0pt}%
    \put(0,0){\includegraphics[width=\unitlength,page=1]{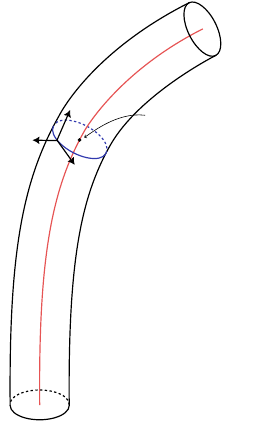}}%
    \put(0.40310242,1.00661566){\color[rgb]{0.19215686,0.22352941,0.68235294}\makebox(0,0)[lt]{\lineheight{1.25}\smash{\begin{tabular}[t]{l}$\alpha_{i}(\eta) + \beta_{i,\eta}$\end{tabular}}}}%
    \put(0.55182944,1.16867705){\color[rgb]{0,0,0}\makebox(0,0)[lt]{\lineheight{1.25}\smash{\begin{tabular}[t]{l}$\alpha_{i}(\eta)$\end{tabular}}}}%
    \put(0.32822577,0.74813787){\color[rgb]{0,0,0}\makebox(0,0)[lt]{\lineheight{1.25}\smash{\begin{tabular}[t]{l}$\Gamma_{i,\textup{lat}}$\end{tabular}}}}%
    \put(0.81257204,1.53978006){\color[rgb]{0,0,0}\makebox(0,0)[lt]{\lineheight{1.25}\smash{\begin{tabular}[t]{l}$\Gamma_{i,\textup{end}}$\end{tabular}}}}%
    \put(0.12366028,-0.03016618){\color[rgb]{0,0,0}\makebox(0,0)[lt]{\lineheight{1.25}\smash{\begin{tabular}[t]{l}$\Gamma_{i,\textup{end}}$\end{tabular}}}}%
    \put(0.2584777,1.207141){\color[rgb]{0,0,0}\makebox(0,0)[lt]{\lineheight{1.25}\smash{\begin{tabular}[t]{l}$\tau_{1}$\end{tabular}}}}%
    \put(0.27802836,0.95899527){\color[rgb]{0,0,0}\makebox(0,0)[lt]{\lineheight{1.25}\smash{\begin{tabular}[t]{l}$\tau_{2}$\end{tabular}}}}%
    \put(0.07698996,1.10208122){\color[rgb]{0,0,0}\makebox(0,0)[lt]{\lineheight{1.25}\smash{\begin{tabular}[t]{l}$\nu$\end{tabular}}}}%
    \put(0.17231858,0.55120236){\color[rgb]{0.90196078,0.32941176,0.32941176}\makebox(0,0)[lt]{\lineheight{1.25}\smash{\begin{tabular}[t]{l}$\alpha_{i}$\end{tabular}}}}%
  \end{picture}%
\endgroup%

  \smallskip
  \caption{\label{fig:cable}Cable parameterization}
  \end{subfigure}
  \caption{Cable geometry}
  \label{fig:geom}
\end{figure}

The cables are allowed to be bent, and we assume that each cable has a~circle-shaped cross-sectional area of constant radius. Denote the radius of this cross-sectional area of the $i$th cable by $r_i$, and let $l_{i}$ be its length. The center curve (i.e., the curve whose trace is consisting of the centers of the cross-sectional circles) is denoted by $\alpha_{i}\colon [0,1]\to\R^3$, see \Cref{fig:cable}. We assume the center curve to be twice continuously differentiable with constant infinitesimal arc length $l_{i}$,
and curvature is bounded by $\frac{1}{r_{i}}$, i.e.,
\begin{equation*}
  \norm{\alpha_{i}^{\prime}(\eta)} = l_{i}
  \quad\text{and}\quad
  \norm{\alpha''(\eta)} < \frac{l_{i}^{2}}{r_{i}} \quad\text{for all}\quad \eta \in [0,1].
\end{equation*}

%\begin{wrapfigure}{r}{3cm}
%  %\scalebox{0.7}{\input{./figs/cable_version_6.pdf_tex}}
%  \def\svgwidth{0.3\textwidth}
%  \input{./figs/cable_version_6.pdf_tex}
%  %\medskip
%  \caption{\label{fig:cable}Parameterization of the cable}
%\end{wrapfigure}
\noindent It has been shown in \cite[Lem.~C.1]{ClemensGuReSkrepek25} that there exist $\kappa_{i1},\kappa_{i2}\in \conC^{1}([0,1], \R^{3})$, such that, for all $\eta\in[0,1]$,
$\big(\tfrac1{l_i}\alpha'(\eta),\kappa_{i1}(\eta),\kappa_{i2}(\eta)\big)$
is a~positively oriented orthonormal basis of $\R^{3}$ (that is, it forms a~rotation matrix for all $\eta\in[0,1]$). 

The shape of the $i$th cable shape is now expressed by
\[
\begin{split}
 & \Omega_{i} = \dset[\big]{\alpha_i(\eta) + \delta\beta_{i\eta}(\theta)}
  {(\delta,\eta,\theta)\in[0,1]^2 \times \lparen-\uppi,\uppi\rbrack},
    \\[0.5ex]
  \mathllap{\text{with}\quad}
  &\beta_{i\eta}(\theta)
    = r_i\big(\kappa_{i1}(\eta)\sin(\theta) + \kappa_{i2}(\eta)\cos(\theta)\big).
\end{split}
\]
and its boundary is consisting of the disjoint union
$\partial\Omega_{i} = \Gamma_{i,\textup{lat}}\cupdot\Gamma_{i,\textup{end}}$,
where $\Gamma_{i,\textup{end}}$ is the union of cross-sectional areas at the two ends of the $i$th cable and $\Gamma_{i,\textup{lat}}$ is the lateral surface of the $i$th cable. That is, for $\beta_{i\eta}(\theta)$ as above, 
\begin{align*}
  \Gamma_{i,\textup{end}}
  &= \dset[\big]{\alpha_i(\eta) + \delta\beta_{i\eta}(\theta)}
     {(\delta,\eta,\theta) \in \sset{0,1} \times[0,1) \times (-\uppi,\uppi]}, \\
  \Gamma_{i,\textup{lat}}
  &= \dset[\big]{\alpha_{i}(\eta) + \beta_{i\eta}(\theta)}{(\eta,\theta) \in (0,1) \times \lparen -\uppi,\uppi\rbrack}.
\end{align*}

\noindent
The lateral boundary $\Gamma_{i,\textup{lat}}$ of the $i$th cable is now parameterised by
\begin{equation}\label{eq:Gammalat}
\begin{split}
&\Phi_i\colon\mapping{[0,1]\times\lparen -\uppi,\uppi \rbrack}{\R^3}
    {(\eta,\theta)}{\alpha_i(\eta)+\beta_{i\eta}(\theta)}
    \\[1.5ex]
  \mathllap{\text{with}\quad}
  &\beta_{i\eta}(\theta)
    = r_i\big(\kappa_{i1}(\eta)\sin(\theta) + \kappa_{i2}(\eta)\cos(\theta)\big).
    \end{split}
\end{equation}
The requirement that the curvature of the cable's profile curve is strictly limited by $\frac{1}{r_{i}}$ (i.e., $\norm{\alpha''(\eta)} \leq \frac{l_{i}^{2}}{r_{i}}$) ensures that the parametrization \eqref{eq:Gammalat} of the lateral surface is essentially injective.

% For
% $\xi\in\Gamma_{\textup{lat}}$, we below introduce the vectors $\tau_\textup{long}(\xi),\tau_\textup{perp}(\xi)\in\R^3$, referred to as unit vectors in longitudinal and perpendicular direction, resp., see \Cref{fig:cable}. We further consider the outward normal unit vector by $\nu(\xi)\in\R^3$. These three vectors are defined by
% %particular, the longitudinal tangential vector fulfills
% \begin{align}
% &\hspace*{-5cm}\forall\;i=1,\ldots,k,\,\eta\in[0,1],\,\theta\in \lparen -\pi,\pi \rbrack:\nonumber\\
% \tau_{\textup{long}}(\Phi_i(\eta,\theta))&=\frac{\alpha_i^\prime(\eta)}{l_i},\label{eq:longvec}\\
% \tau_{\textup{perp}}(\Phi_i(\eta,\theta))&=\frac{\beta_{i\eta}^\prime(\theta)}{r_i},\label{eq:perpvec}\\
% \nu(\Phi_i(\eta,\theta))&=\frac{\beta_{i\eta}(\theta)}{r_i},\label{eq:normvec}
% \end{align}
% where $\beta_{i\eta}^\prime$ stands for the derivative with respect with respect to $\theta$.
% Expectably, $\tau_\textup{long}(\xi)$ and $\tau_\textup{perp}(\xi)$ form an orthonormal basis of the tangent space of $\Gamma_{\textup{lat}}$ at $\xi$, and that $\nu(\xi)$ spans the normal space, see \cite[Lemma~3.2]{ClemensGuReSkrepek25}. 

% \section{Transmission lines and the electromagnetic field}\label{sec:singsys}
% Before analyzing the coupled system, we first consider the subsystems individually. Although we begin with the transmission lines, it is necessary to first say a few words about the electromagnetic field. 

\subsection{The transmission lines}

%\TODO{achtung spaeter sind es $m$ viele kabeln nicht $k$ viele}
\noindent The $k$ transmission lines are described by the \emph{telegrapher’s equations}, 
extended by external current inputs and electric field outputs. 
The internal variables of the model are the $\C^{k}$-valued functions
\smallskip
\begin{center}
 \begin{tabular}{rl}
 $\bm{\psi}$: & \text{magnetic flux}, \\
 $\bm{q}$: &  \text{electric charge}, 
 \end{tabular}
\end{center}
where each component corresponds to the flux (respectively, the charge) 
associated with one of the transmission lines. 
Both quantities depend on time~$t$ and on the spatial coordinate 
$\eta \in [0,1]$. %By writing $\bm{\psi}(t)$, $\bm{q}(t)$, we mean the spatial distributions of the magnetic flux and electric charge at time $t$. That is, both $\bm{\psi}(t)$ and $\bm{q}(t)$ are $\C^k$-valued functions defined on $[0,1]$.
The system is driven by an external current input~$\bm{I}_{\ext}:\R_{\ge0}\times[0,1]\to\C^k$, 
and the corresponding external electric field intensity~$\bm{E}_{\ext}:\R_{\ge0}\times[0,1]\to\C^k$ 
is taken as the output. 
These quantities, $\bm{I}_{\ext}$ and $\bm{E}_{\ext}$, 
will later serve as coupling variables between the transmission lines 
and the electromagnetic field.
The transmission line model is
\begin{equation}\label{eq:teleq}
  \begin{aligned}
    \ddt
    \begin{pmatrix}
      \bm{\psi}(t,\eta) \\[1mm] \bm{q}(t,\eta)
    \end{pmatrix}
    &=
    \begin{bmatrix}
      -\bm{R}(\eta) & -\tfrac{\partial}{\partial \eta}\\
      -\tfrac{\partial}{\partial \eta} & -\bm{G}(\eta)
    \end{bmatrix}
    \begin{pmatrix}
      \bm{L}(\eta)^{-1}\bm{\psi}(t,\eta) \\[1mm]
      \bm{C}(\eta)^{-1}\bm{q}(t,\eta)
    \end{pmatrix}+
    \begin{bmatrix}
      \;\;0\\[1mm]-\tfrac{\partial}{\partial \eta}
    \end{bmatrix} \bm{I}_{\ext}(t,\eta),\\
    \bm{E}_{\ext}(t,\eta)
    &=
    \mspace{3mu}
    \begin{bmatrix}
      \mathrlap{\mspace{33mu}0}\hphantom{\quad\quad\;\tfrac{\partial}{\partial \eta}} &{\phantom{-}\tfrac{\partial}{\partial \eta}\;\;}%\mathllap{0\mspace{10mu}}
    \end{bmatrix}
    \begin{pmatrix}
      \bm{L}(\eta)^{-1}\bm{\psi}(t,\eta) \\[1mm]
      \bm{C}(\eta)^{-1}\bm{q}(t,\eta)
    \end{pmatrix},
  \end{aligned}
\end{equation}
%where $\bm{\psi}(t),\bm{q}(t),\bm{V}_{\ext}(t),\bm{I}_{\ext}(t)$  are $\C^k$-valued functions depending on the spatial variable $\eta$.\\
%\begin{equation}
%  \begin{aligned}
%    \tfrac{\partial}{\partial t} \bm{\psi}(t,\eta)
%    &=
%    -\bm{R}(\eta)\Lb(\eta)^{-1}\bm{\psi}%(t,\eta)-\tfrac{\partial}{\partial \eta}\big(\bm{C}(\eta)^{-1}\bm{q}(t,\eta)\big) + \bm{V}_{\ext}(t,\eta),
    %\\
    %\tfrac{\partial}{\partial t} \bm{q}(t,\eta)
    %&= -\tfrac{\partial}{\partial \eta}\big(\Lb(\zeta)^{-1}\bm{\psi}(t,\eta)\big)-\bm{G}(\eta)\bm{C}(\eta)^{-1}\bm{q}(t,\eta),
    %\\
    %\bm{I}_{\ext}(t,\eta)
    %&=
    %\quad\quad\quad \Lb(\zeta)^{-1}\bm{\psi}(t,\eta),
    %\qquad\eta\in[0,1],\;t\geq0.
  %\end{aligned}
%\end{equation}
where the mappings 
$\bm{C}, \bm{G}, \bm{L}, \bm{R} \colon [0,1] \to \C^{k\times k}$ 
denote, respectively, the transverse capacitance and conductance matrices, 
and the longitudinal inductance and resistance matrices. For any $t>0$, the voltage and total current  
along the transmission lines are represented by the functions 
$\bm{V}, \bm{I}_{\textup{tot}} \colon \R_{\ge0}\times [0,1] \to \C^{k}$, defined as
\begin{equation}\label{eq:input1}
  \bm{V}(t,\eta) \coloneq \bm{C}(\eta)^{-1}\bm{q}(t,\eta), \qquad
  \bm{I}_{\textup{tot}}(t,\eta) \coloneq \bm{L}(\eta)^{-1}\bm{\psi}(t,\eta) + \bm{I}_{\ext}(t,\eta).
\end{equation}% The interpretation of a model for a single transmission line in terms of an infinitesimal ladder network is illustrated in \Cref{Fig:TML}.

\noindent
Our assumptions concerning the parameters align with those presented in \cite{JaZw12}.

\begin{assumption}[Transmission lines - parameters]\label{ass:tl}
  $k\in\N$, $\bm{C}$, $\bm{L}$, $\bm{R}$, $\bm{G}\in \Lp{\infty}([0,1];\C^{k\times k})$ with
  $\bm{C}^{-1},\bm{L}^{-1}\in \Lp{\infty}([0,1];\C^{k\times k})$ and
  \begin{equation*}  
    \bm{C}(\eta)>0,\;
    \bm{L}(\eta)>0,\;
    \bm{R}(\eta)+\bm{R}(\eta)\adjun \geq 0,\;
    \bm{G}(\eta)+\bm{G}(\eta)\adjun \geq 0
    \quad\text{for a.e.}\quad \eta\in[0,1].
  \end{equation*}
  %Further, the functions $\bm{C},\bm{L},\bm{C}^{-1},\bm{L}^{-1},\bm{R},\bm{G}$ are essentially bounded.
\end{assumption}
\noindent The system is provided with an initial condition $\bm{q}(0,\eta)=\bm{q}_0(\eta)$, $\bm{\psi}(0,\eta)=\bm{\psi}_0(\eta)$ for some given $\bm{q}_0,\bm{\psi}_0:[0,1]\to\C^k$.
We further impose boundary conditions for the voltage and current along the transmission line, as given in \eqref{eq:input1}. For some $m \leq 2k$, with $W_{B,{\rm inp}} \in \C^{m\times 4k}$ and $W_{B,0} \in \C^{(2k-m)\times 4k}$, these are given by
\begin{equation}
u(t) = W_{B,{\rm inp}}
  \begin{pmatrix}
    \phantom{-}\bm{V}(t,0)\\
    \phantom{-}\bm{I}(t,0)  \\
    \phantom{-}\bm{V}(t,1) \\
    -\bm{I}(t,1)
  \end{pmatrix},
\quad
  0 = W_{B,0}
  \begin{pmatrix}
    \phantom{-}\bm{V}(t,0)\\
    \phantom{-}\bm{I}(t,0)  \\
    \phantom{-}\bm{V}(t,1) \\
    -\bm{I}(t,1)
  \end{pmatrix},\label{eq:input2}
\end{equation}
where $u\colon \R_{\geq 0} \to \C^{m}$ represents the input of the system.
The negative sign in the current at $\eta = 1$ reflects that, unlike at $\eta = 0$, the directions of current and voltage are opposite.

\begin{assumption}[Transmission line - input configuration]\label{ass:bndcont}
  The matrix $W_B\coloneq [W_{B,{\rm inp}}^*,\,W_{B,0}^*]^*\in\C^{2k\times 4k}$ has full row rank, and
  \begin{equation}
    W_B
    \begin{bmatrix}
      0 & \id_{{2k}} \\
      \id_{{2k}} & 0
    \end{bmatrix}
    W_B^*\ge0.\label{eq:WBdef}
  \end{equation}
\end{assumption}
\noindent Physical interpretations of the above assumption are discussed in \cite{ClemensGuReSkrepek25}, for example in the context of parallel or serial interconnections with ohmic resistors.
%We impose some further conditions on $W_B$ and $W_C$.
%\begin{assumptions}\label{ass:bndcont}
%The matrices $W_B,W_C\in\C^{2m\times 4m}$ fulfill
%\[
%\begin{bmatrix}
%  W_B \\ W_C
%\end{bmatrix}^*
%\begin{bmatrix}
%  0 & 0 & \id_m & 0 \\
%  0 & 0 & 0 & \id_m \\
%  \id_m & 0 & 0 & 0 \\
%  0 & \id_m & 0 & 0
%\end{bmatrix}
%\begin{bmatrix}
%  W_B \\ W_C
%\end{bmatrix}
%=
%\begin{bmatrix}
%  0 & 0 & -\id_m & 0 \\
%  0 & 0 & 0 & \id_m \\
%  -\id_m & 0 & 0 & 0 \\
%  0 & \id_m & 0 & 0
%\end{bmatrix}
    %\]
%\end{assumptions}

\noindent Our system is moreover equipped with an $\C^{p}$-valued output $y$ of the form
\begin{equation}\label{eq:output}
  y(t) = W_{C,{\rm out}}
  \begin{pmatrix}
    \phantom{-}\bm{V}(t,0) \\
    \phantom{-}\bm{I}(t,0) \\
    \phantom{-}\bm{V}(t,1) \\
    -\bm{I}(t,1)
  \end{pmatrix},
  %\begin{pmatrix}-(\bm{L}^{-1}\bm{\psi})(t,0)\\\phantom{-}(\bm{L}^{-1}\bm{\psi})(t,1)\\\phantom{-}(\bm{C}^{-1}\bm{q})(t,0)\\\phantom{-}(\bm{C}^{-1}\bm{q})(t,1)\end{pmatrix},
\end{equation}
for some $W_{C,{\rm out}}\in\C^{p\times 4m}$. A special role is played by so-called \emph{co-located outputs}.
\begin{definition}
  Assume that $W_{B,{\rm inp}}\in\C^{m\times 4k}$, $W_{B,0}\in\C^{(2k-m)\times 4k}$, $m\leq 2k$, $W_B \coloneq [W_{B,{\rm inp}}^*,\,W_{B,0}^*]^*\in\C^{2m\times 4m}$ fulfill \Cref{ass:bndcont}.
  Then an output \eqref{eq:output} is called \em{co-located} to $u$ as in \eqref{eq:input1}, \eqref{eq:input2}, if $W_{C,{\rm out}}\in\C^{m\times 4k}$ has the form
  \begin{equation}\label{eq:col1}
    W_{C,{\rm out}}=[\id_m,\,0_{m\times (2k-m)}]\,W_C
  \end{equation}
  for some $W_C\in\C^{2k\times 4k}$ with the property that $[W_B^*,\,W_C^*]\in\C^{4k\times 4k}$ %is invertible
  with
  \begin{equation}
    \begin{bsmallmatrix}0&\id_{2k}\\\id_{2k}&0\end{bsmallmatrix}-\begin{bsmallmatrix}W_{B}\\W_C\end{bsmallmatrix}\adjun\begin{bsmallmatrix}0&\id_{2k}\\\id_{2k}&0\end{bsmallmatrix}\begin{bsmallmatrix}W_{B}\\W_C\end{bsmallmatrix}\leq0.\label{eq:col2}
    %    \left(\begin{bmatrix}W_B\\{W}_C\end{bmatrix}\begin{bmatrix}0&\id_{2k}\\\id_{2k}&0\end{bmatrix}\begin{bmatrix} W_B^*&{W}_C^*\end{bmatrix}\right)^{-1}-\begin{bmatrix}0&\id_{2k}\\\id_{2k}&0\end{bmatrix}\leq0.
  \end{equation}
\end{definition}
The existence of co-located outputs is established in \cite[Prop.~2.7]{ClemensGuReSkrepek25},
where a subsequent discussion also provides their physical interpretation. It is further shown there that such outputs give rise to a system satisfying an energy balance, where the inner product of the input and output corresponds to the power supplied to the system.
Our transmission line model can now be viewed from a systems-theoretic perspective as a system featuring two classes of inputs and outputs:
the external input together with the distributed current along the cable, and the external output accompanied by the corresponding distributed voltages.

\subsection{The electromagnetic field}\label{sec:max}

For the domain $\Omega \subset \R^3$ introduced in \Cref{sec:Omega},
the evolution of the electromagnetic field in the linear case is governed by Maxwell’s equations
\begin{equation}\label{eq:Maxwell}
   \frac{\partial}{\partial t}
   \begin{pmatrix} \bm{B}(t,\xi)\\ \bm{D}(t,\xi)\end{pmatrix}
   =
   \begin{bmatrix}0 & -\rot \\ \rot &-\bm{\sigma}(\xi)\end{bmatrix}
   \begin{pmatrix}\bm{\mu}(\xi)^{-1}\bm{B}(t,\xi) \\ \bm{\epsilon}(\xi)^{-1}\bm{D}(t,\xi)\end{pmatrix}
 \end{equation}
for the $\C^3$-valued physical quantities
\smallskip
\begin{center}
 \begin{tabular}{rl}
 $\bm{B}$: & magnetic flux density, \\
 $\bm{D}$: & electric flux density. \\
 \end{tabular}
\end{center}
Here, $\bm{\mu}\colon \Omega \to \C^{3\times 3}$ denotes the magnetic permeability, 
$\bm{\epsilon}\colon \Omega \to \C^{3\times 3}$ the electric permittivity, 
and $\bm{\sigma}\colon \Omega \to \C^{3\times 3}$ the electric conductivity.  
The quantities
\[
  \bm{H}(t,\xi) \coloneq \bm{\mu}(\xi)^{-1}\bm{B}(t,\xi),
  \quad
  \bm{E}(t,\xi) \coloneq \bm{\epsilon}(\xi)^{-1}\bm{D}(t,\xi)
\]
are referred to as the magnetic and electric field intensities, respectively.  
Typically, Maxwell’s equations are complemented by the constraints
\[
  \div \bm{B}(t,\xi) = 0,
  \qquad
  \div \bm{D}(t,\xi) = \bm{\rho}(t,\xi),
\]
where $\bm{\rho}(t,\argdot)\colon \Omega \to \C$ is a scalar field representing the charge density at time~$t$.  
Since $\div \rot = 0$, these divergence relations are preserved in time and can therefore be imposed through the initial conditions; consequently, they will be omitted.

Our assumptions on the corresponding parameters are as follows.

\begin{assumption}[Maxwell's equations - parameters]\label{ass:Maxwell}
  We assume that $\bm{\epsilon},\bm{\mu},\bm{\sigma}\in \Lp{\infty}(\Omega;\C^{3\times 3})$ with
  $\bm{\epsilon}^{-1},\bm{\mu}^{-1}\in \Lp{\infty}(\Omega;\C^{3\times 3})$ and
  \[
    \bm{\epsilon}(\xi)>0,\;\bm{\mu}(\xi)>0,\;\bm{\sigma}(\xi)+\bm{\sigma}(\xi)\adjun\geq 0
    \quad\text{for almost all}\quad \xi\in \Omega.
  \]
\end{assumption}

%\begin{remark}\label{rem:max}\
%The boundedness and positivity assumption on $\bm{\epsilon}$ and $\bm{\mu}$ means that there exist $c_1,c_2>0$ such that
%\[
%c_1\id_3\leq \bm{\epsilon}(\zeta)\leq c_2\id_3,\;c_1\id_3\leq \bm{\mu}(\zeta)\leq c_2\id_3\quad\text{for almost all }\\xi\in\Omega.
%\]
%In particular,
%\[\left(\begin{smallmatrix}\bf{B}\\\bf{D}\end{smallmatrix}\right)\mapsto \left(\langle \bm{B},\bm{\mu}^{-1}\bm{B}\rangle_{\Lp{2}(\Omega;\C^3)}+\langle \bm{D},\bm{\epsilon}^{-1}\bm{D}\rangle_{\Lp{2}(\Omega;\C^3)} \right)^{1/2}\]
%is equivalent to the standard one in $\Lp{2}(\Omega;\C^3)^2$, cf. \Cref{rem:tl}\,\ref{rem:tla}.
%\end{remark}
\noindent
The system is provided with an initial condition $\bm{B}(0)=\bm{B}_0$, $\bm{D}(0)=\bm{D}_0$ for some given $\bm{B}_0,\bm{D}_0\colon \Omega\to\C^3$.  To ensure a physically complete description, it is necessary to define appropriate boundary conditions.

The geometry of the domain $\Omega$ implies that its boundary $\partial\Omega$ 
is the disjoint union of the sub-boundaries 
$\partial\Omega_{0}, \partial\Omega_{1}, \ldots, \partial\Omega_{k}$.  
Each cable domain $\Omega_i$ has a boundary composed of its lateral surface 
$\Gamma_{i,\textup{lat}}$ and its two end surfaces 
$\Gamma_{i,\textup{end}}$, $i=1,\ldots,k$.  
Consequently,
\begin{align*}
  \partial\Omega
    = \partial\Omega_0 \cupdot \Gamma_{\textup{end}} \cupdot \Gamma_{\textup{lat}},
    \qquad
    \Gamma_{\textup{end}} \coloneq \bigcupdot_{i=1}^{k}\Gamma_{i,\textup{end}},
    \quad
    \Gamma_{\textup{lat}} \coloneq \bigcupdot_{i=1}^{k}\Gamma_{i,\textup{lat}}.
\end{align*}
\noindent
We next specify boundary conditions on $\partial\Omega_0$, 
$\Gamma_{\textup{end}}$, and $\Gamma_{\textup{lat}}$.  
To this end, let $\nu(\xi)\in\R^{3}$ denote the outward unit normal, 
which exists for almost every $\xi\in\partial\Omega$ 
(since $\Omega$ is a Lipschitz domain).  
Hence, $\nu \in \Lp{\infty}(\partial\Omega;\C^{3})$.  
For almost every $\xi\in\partial\Omega$, we define the orthogonal projection on the tangent space, which reads
\begin{align*}
  \pi_\tau(\xi)\colon \C^{3}\to\C^{3}, \qquad
  w \mapsto (\nu(\xi)\times w)\times \nu(\xi).
\end{align*}
%which satisfies
%\begin{equation}\label{eq:tangproj}
%  \pi_\tau(\xi) w
%    = (\nu(\xi)\times w)\times \nu(\xi)
%    = w - (\nu(\xi)^\top w)\,\nu(\xi),
%    \quad \forall\, w\in\C^{3}.
%\end{equation}
%Thus, $\pi_\tau \in \Lp{\infty}(\partial\Omega; \C^{3\times3})$, 
%and for almost every $\xi\in\partial\Omega$, 
%$\pi_\tau(\xi)$ acts as the orthogonal projector 
%onto the tangent space of $\partial\Omega$ at~$\xi$.  
%It follows immediately that, for almost all $\xi\in\partial\Omega$ and any $w\in\C^3$,
%$\pi_\tau(\xi) w = 0$ if and only if $\nu(\xi)\times w = 0$.

\subsubsection{Boundary conditions at the computational domain}
In the following analysis, we impose the boundary condition
\begin{equation}
  \pi_\tau(\argdot) \big(\bm{\epsilon}(\argdot)^{-1}\bm{D}(t,\argdot)\big)\big\vert_{\partial\Omega_0} = 0, \qquad t \ge 0,
\label{eq:bndcondcomp}\end{equation}
which corresponds to perfect electrical insulation outside $\Omega_0$.  
Alternatively, superconductivity in the exterior region can be modeled by
\begin{equation*}
  \nu(\argdot) \times\big(\bm{\mu}(\argdot)^{-1}\bm{B}(t,\argdot)\big)\big\vert_{\partial\Omega_0} = 0, \qquad t \ge 0,
\end{equation*}
or by employing the \emph{Leontovich boundary conditions}~\cite{Leo44}, 
which relate the tangential components of the electric and magnetic fields.  
An analogous analysis could be carried out for these cases, 
but this is not pursued here.

\subsubsection{Boundary conditions at the cover surfaces of the cable}
We assume that the electric field is polarized normally to the cable end surfaces, which leads to
\begin{equation}
  \pi_\tau(\argdot)(\bm{\epsilon}(\argdot)^{-1}\bm{D}(t,\argdot))\big\vert_{\Gamma_{\textup{end}}} = 0,\qquad t\geq0.
\label{eq:bndcondend}\end{equation}

\subsection{Coupling - transmission line and lateral cable surfaces}\label{sec:coupling-idea}

The lateral cable surfaces serve as the coupling interface between the external electromagnetic field, governed by Maxwell’s equations~\eqref{eq:Maxwell}, and the transmission line, described by the telegrapher’s equations~\eqref{eq:teleq}. 
%We assume that the tangential electric field is aligned with the cable axis and determined by $\bm{E}_{\ext}$, while the magnetic field is polarized transversely to the cable and depends on the current distribution within the line.
% In order to couple the transmission line \eqref{eq:teleq} and Maxwell's equations \eqref{eq:Maxwell-system} we need to overcome a miss match of dimensions of the ports. The ports of the transmission line are functions in space, where the spatial dependency is one-dimensional, whereas the ports of Maxwell's equations are function on the lateral surface of the cable, which is two-dimensional.

% Hence, we will lift the dynamics of the transmission line on the entire volume of the cable by distributing the quantities according to curvature on every cross section that is perpendicular to the direction of the cable, e.g., no curvature gives a uniform distribution.
% %On the other hand we can bundle the quantities of Maxwell's equation on the lateral surface of the cable, i.e., integrating over every cross section.
\noindent The coupling between the transmission-line current and the magnetic field intensity follows from Ampère’s law, relating the line integral of the magnetic field along a closed contour encircling the conductor to the enclosed current. For $i=1,\ldots,k$,
\[
  \forall\, \eta \in [0,1] \; : \;
  \oint_{\alpha_i(\eta)+\beta_{i\eta}}
  \underbrace{\bm{\mu}(\xi)^{-1}\bm{B}(t,\xi)}_{=\mathrlap{\bm{H}(t,\xi)}} \mathclose{}
  \cdot \dx[\bm{s}(\xi)]
  = -\bm{I}_{i,\ext}(t,\eta),
  \]
where $\bm{I}_{i,\ext}$ is the $i$th component of $\bm{I}_{\ext}(t,\eta)$. Now using that the above integral only incorporated the tangential component of the magnetic field, we see that Amp{\`e}re's law is equivalent to, for $i=1,\ldots,k$, 
\begin{equation}
  \forall \, \eta\in[0,1] \; : \; 
  \oint_{\alpha_{i}(\eta) + \beta_{i\eta}} \mspace{-8mu} \big(\nu(\xi)\times \big(\bm{\mu}(\xi)^{-1}\bm{B}(t,\xi)\big)\big)\times \nu(\xi) \cdot \dx[\bm{s}(\xi)]
  = -\bm{I}_{i,\ext}(t,\eta).
  \label{eq:magcouple}
\end{equation}

\noindent The coupling between the external voltage variable $\bm{E}_{\ext}$ of the transmission line
and the electric field on the cable surface relies on the polarization assumptions
that the electric field is tangentially aligned with the cable axis so that
its line integral along a longitudinal path equals the voltage drop along the line,
and that it vanishes in the perpendicular direction.  
According to the results in \cite[Sec.~4]{ClemensGuReSkrepek25}, 
this leads, for $i = 1,\ldots,k$, to
\begin{equation}
  \pi_\tau(\Phi_i(\eta,\theta))\,
  \underbrace{\bm{\epsilon}(\Phi_i(\eta,\theta))^{-1}\bm{D}(t,\Phi_i(\eta,\theta))}_{=\mathrlap{\bm{E}(t,\Phi_{i}(\eta,\theta))}}
  = \nabla\Phi_i(\eta,\theta)^{\dagger}
  \begin{pmatrix}
    \bm{E}_{i,\ext}(t,\eta)\\ 0
  \end{pmatrix}.
  \label{eq:elcouple}
\end{equation}

\section{Analysis of the coupling conditions}\label{sec:analysis-of-port-operator}

\noindent 
The coupling relations \eqref{eq:magcouple} and \eqref{eq:elcouple} between the tangential traces of the electromagnetic fields and the external quantities $\bm{I}_{\ext}$ and $\bm{E}_{\ext}$ can be expressed by the operators $\portOp_\textup{mag}$ and $\portOp_\textup{el}$ 
via
 \begin{equation}
 \bm{I}_{\ext}(t,\argdot) = \portOp_\textup{mag} \big(\nu(\argdot) \times \bm{\mu}(\argdot)^{-1} \bm{B}(t,\argdot)\big), \quad \bm{E}_{\ext}(t,\argdot) = \portOp_\textup{el} \big(\pi_\tau(\argdot) \bm{\epsilon}(\argdot)^{-1} \bm{D}(t,\argdot)\big).\label{eq:couplcondop}
 \end{equation}
In this section, we provide the functional-analytic framework required for a rigorous treatment of these coupling conditions.  
We first introduce the boundary trace spaces for the electric and magnetic fields on the cable surfaces and then analyze the operators $\portOp_\textup{mag}$ and $\portOp_\textup{el}$, clarifying their mutual relation.

\subsection{Tangential boundary trace spaces}
%\TODO{ist die frage wie tief man da reingehen will.}

As mentioned in \Cref{sec:Omega}, the spatial domain where the electromagnetic field evolves is represented as $\Omega = \Omega_{0} \setminus \bigcup_{i=1}^{k} \cl{\Omega_{i}}$, and the boundary is divided into
$\partial \Omega =   %\Gamma_{\textup{lat}} \cup \Gamma_{\textup{end}} \cupdot \Gamma_{\textup{ext}} =
  \Gamma_{\textup{lat}} \cupdot \Gamma_{\textup{r}}$,
where $\Gamma_{\textup{r}} = \Gamma_{\textup{end}} \cupdot \Gamma_{\textup{ext}}$ with
\begin{align*}
  \Gamma_{\textup{lat}} = \bigcup_{i=1}^{k} \Gamma_{i,\textup{lat}},\quad
  \Gamma_{\textup{end}} = \bigcup_{i=1}^{k} \Gamma_{i,\textup{end}}, \quad
  \Gamma_{\textup{ext}} = \partial\Omega_{0},
\end{align*}
i.e., $\Gamma_{\textup{lat}}$ is the union of the lateral surfaces of the cables, $\Gamma_{\textup{end}}$ is the union of the end faces of the cables, $\Gamma_{\textup{ext}}$ is the exterior boundary of the computational domain and $\Gamma_{\textup{r}}$ is the rest of boundary from the perspective of $\Gamma_{\textup{lat}}$, see \Cref{fig:domain}.
Motivated by the zero tangential boundary conditions for the electric field on $\Gamma_{\textup{r}}$, we introduce the space of smooth functions with compact support not intersecting with $\Gamma_r$, i.e.,
\begin{align*}
  \Cc_{\Gamma_{\textup{r}}}(\Omega) \coloneq \dset*{f \in \conC^{\infty}(\cl{\Omega})}{\supp f \subseteq \cl{\Omega}\text{ compact with } \cl{\Gamma_{\textup{r}}} \cap \supp f =\emptyset}.
\end{align*}
Hence, $f \in \Cc_{\Gamma_{\textup{r}}}(\Omega)$ is always zero on $\Gamma_{\textup{r}}$, but can be non-zero on $\partial\Omega \setminus \Gamma_{\textup{r}}$, and let $\Cc(\Omega)\coloneq \Cc_{\partial\Omega}(\Omega)$. As for the other function spaces in this work, we append ``$; \C^m$'' to denote the space of $\C^m$-valued functions being component-wise within these spaces.
Also recall the Sobolev space for the distributional $\rot$-operator, namely
\begin{align*}
  \Hspace(\rot,\Omega) = \dset{\bm{E} \in \Lp{2}(\Omega;\C^3)}{\rot\bm{E} \in \Lp{2}(\Omega;\C^3)},
\end{align*}
%where $\rot\bm{E}$ has to be understood in a distributional sense. This space 
which is endowed with the inner product
\begin{align*}
  \scprod{\bm{E}}{\bm{H}}_{\Hspace(\rot,\Omega)} \coloneq \scprod{\bm{E}}{\bm{H}}_{\Lp{2}(\Omega;\C^3)} + \scprod{\rot \bm{E}}{\rot \bm{H}}_{\Lp{2}(\Omega;\C^3)}.
\end{align*}
We define a subspace of $\Hspace(\rot,\Omega)$ that has homogeneous boundary conditions on $\Gamma_{\textup{r}}$ by the closure of $\Cc_{\Gamma_{\textup{r}}}(\Omega)$ with respect to the norm in $\Hspace(\rot,\Omega)$, i.e.,
\begin{align*}
  \cH_{\Gamma_{\textup{r}}}(\rot,\Omega) \coloneq \cl[\Hspace(\rot,\Omega)]{\Cc_{\Gamma_{\textup{r}}}(\Omega; \C^3)}
\end{align*}
Motivated by the boundary conditions discussed in \Cref{sec:max} for Maxwell's equations, we consider two specific types of boundary trace operators, namely
\begin{align*}  
  \tantr\colon\!\!
  \mapping{\conC^{\infty}(\cl{\Omega};\C^3)}{\Lp{2}(\partial\Omega;\C^3)}{\bm{E} }{\big(\nu \times \big(\bm{E}\big\vert_{\partial\Omega}\big)\big) \times \nu,}
  \;
  \tanxtr\colon\!\!
  \mapping{\conC^{\infty}(\cl{\Omega};\C^3)}{\Lp{2}(\partial\Omega;\C^3)}{\bm{H}}{\nu \times \big(\bm{H}\big\vert_{\partial\Omega}\big),}
\end{align*}
where $\nu\in \Lp{\infty}(\partial\Omega;\C^3)$ pointwisely refers to the outward normal vector on $\partial \Omega$.
%These are called the tangential trace and the twisted tangential trace.
% The operators
% \[
%   \tantr \boundtr\colon \Hspace^{1}(\Omega) \to \Lp{2}(\partial\Omega)
%   \quad\text{and}\quad
%   \tanxtr \boundtr\colon \Hspace^{1}(\Omega) \to \Lp{2}(\partial\Omega)
% \]
% can be extended to $\Hspace(\rot,\Omega)$.
%$For convenience, we will denote
%\[\forall\,\bm{E} \in \Hspace^{1}(\Omega;\C^3):\quad \tantr \bm{E} \coloneq \tantr \big(\bm{E} \big\vert_{\partial\Omega}\big),\;\tanxtr \bm{E} \coloneq \tanxtr\big( \bm{E} \big\vert_{\partial\Omega}).\]
%One challenge with Maxwell's equations involving boundary conditions is related to the boundary space. 
To properly define the boundary spaces, one would need to introduce several technical concepts that could make the presentation less accessible.  
To keep the exposition concise, we therefore adopt a simplified approach.  
The downside is that certain structural aspects of the coupling are not immediately apparent in its operator-theoretic representation, since its domain cannot be characterized explicitly without introducing these additional spaces.
The integration by parts formula for the $\rot$ operator
\cite[Thm.~2.11]{GiraRavi86} gives
%Recall that for smooth functions $E, H \in \conC^{\infty}(\R^{3})$ we have the following
%\begin{align*}
  %\scprod{\rot E}{H}_{\Lp{2}(\Omega)} - %\scprod{E}{\rot H}_{\Lp{2}(\Omega)} = %\scprod{\tanxtr E}{\tantr H}_{\Lp{2}%(\partial\Omega)}.
%\end{align*}
%Clearly, if we regard , then we can replace the boundary space $\Lp{2}(\partial\Omega)$ with $\Lp{2}(\Gamma_{\textup{lat}})$, i.e.,
for all $\bm{E} \in \Cc_{\Gamma_{\textup{r}}}(\Omega;\C^3)$, $\bm{H} \in \conC^{\infty}(\cl{\Omega};\C^3)$,
\begin{align}\label{eq:int-by-parts-with-boundary-conditions}
  \scprod{\rot \bm{E}}{\bm{H}}_{\Lp{2}(\Omega;\C^3)} - \scprod{\bm{E}}{\rot \bm{H}}_{\Lp{2}(\Omega;\C^3)} &= \scprod{\tanxtr\bm{E}}{\tantr \bm{H}}_{\Lp{2}(\Gamma_{\textup{lat}};\C^3)}
\end{align}
Now it is possible to extend this integration by parts formula for $\bm{E}, \bm{H} \in \Hspace(\rot,\Omega)$ with suitable boundary spaces.
However, the left-hand-side of~\eqref{eq:int-by-parts-with-boundary-conditions} is well defined for $\bm{E} \in \cH_{\Gamma_{\textup{r}}}(\rot, \Omega)$ and $\bm{H} \in \Hspace(\rot,\Omega)$. Hence, we extend the right-hand-side $\scprod{\tanxtr\bm{E}}{\tantr \bm{H}}_{\Lp{2}(\Gamma_{\textup{lat}};\C^3)}$ just by the left-hand-side of~\eqref{eq:int-by-parts-with-boundary-conditions} and denote it as
\begin{align*}
  \dualprod{\tanxtr \bm{E}}{\tantr \bm{H}}_{\Vtaud,\Vtau} \coloneq
  \scprod{\rot \bm{E}}{\bm{H}}_{\Lp{2}(\Omega;\C^3)} - \scprod{\bm{E}}{\rot \bm{H}}_{\Lp{2}(\Omega;\C^3)}.
\end{align*}
We will regard the left-hand-side of the previous equation just as symbol, although it possible to define spaces $\Vtau$ and $\Vtaud$, and extensions of $\tantr$ and $\tanxtr$ to $\Hspace(\rot,\Omega)$ such that this is really a dual pairing, see e.g., \cite{BuCoSh02} or \cite{Sk21} for a more general approach.
We further define $\dualprod{\tantr \bm{H}}{\tanxtr \bm{E}}_{\Vtau,\Vtaud} = \conj{\dualprod{\tanxtr \bm{E}}{\tantr \bm{H}}_{\Vtaud,\Vtau}}$, which gives
\begin{equation*}
  \dualprod{\tantr \bm{H}}{\tanxtr \bm{E}}_{\Vtau,\Vtaud}
  = \scprod{\bm{H}}{\rot \bm{E}}_{\Lp{2}(\Omega;\C^3)} - \scprod{\rot \bm{H}}{\bm{E}}_{\Lp{2}(\Omega;\C^3)}.  
\end{equation*}

\begin{definition}
Suppose that the spatial domains are as in \Cref{sec:Omega}.
  Let $f \in \Lp{2}(\Gamma_{\textup{lat}};\C^3)$ and $\bm{E} \in \cH_{\Gamma_{\textup{r}}}(\rot,\Omega)$. We say $\tantr \bm{E} = f$ on $\Gamma_{\textup{lat}}$, if
  \begin{align*}
    \dualprod{\tantr \bm{E}}{\tanxtr \bm{H}}_{\Vtau,\Vtaud} = \scprod{f}{\tanxtr \bm{H}}_{\Lp{2}(\Gamma_{\textup{lat}})}
    \quad\text{for all}\quad
    \bm{H} \in \conC^{\infty}(\cl{\Omega};\C^{3}),
  \end{align*}
  i.e., $f$ is the weak $\Lp{2}$ representative of $\tantr \bm{E}$.
\end{definition}

\noindent A concise overview of these boundary spaces is provided in the appendix of \cite{SkWa24}.
Since their detailed structure is not required in the present work, it is not discussed here in detail.

% Let $\Upsilon \subseteq \partial\Omega$. Then we define
% \begin{align*}
%   \cVtau(\Upsilon) \coloneq \tantr \cH_{\partial\Omega \setminus \Upsilon}(\rot,\Omega)
%   \quad\text{and}\quad
%   \cVtaud(\Upsilon) \coloneq \tanxtr \cH_{\partial\Omega \setminus \Upsilon}(\rot,\Omega)
% \end{align*}
% In particular we are interested in $\cVtaud(\Gamma_{i,\textup{lat}})$.

% Note that $\portOp_{i}$ is defined on $\Lp{2}_{\tau}(\Gamma_{i,\textup{lat}})$. The boundary space $\cVtaud(\Gamma_{i,\textup{lat}})$ of Maxwell's equations has a dense intersection with $\Lp{2}_{\tau}(\Gamma_{i,\textup{lat}})$. Moreover, for $\Hspace^{1}(0,1)$ we know that $\portOp_{i}\adjun$ maps into $\Lp{2}_{\tau}(\Gamma_{i,\textup{lat}}) \cap \Vtau(\Gamma_{i,\textup{lat}})$. Hence, we can weakly define $\portOp_{i}$ on $\cVtaud(\Gamma_{i,\textup{lat}})$.

% % Let $D \in \cH_{\Gamma_{i,\textup{end}}}(\rot,\Omega)$

% \begin{definition}
%   % g = I
%   Let $\phi \in \cVtaud(\Gamma_{\textup{lat}})$. If there exists a $f \in \Lp{2}(0,1)^{m}$ such that
%   \begin{align*}
%     \scprod{f}{g}_{\Lp{2}(0,L)} = \dualprod{\phi}{\portOp\adjun g}_{\cVtaud(\Gamma_{\textup{lat}}),\Vtau(\Gamma_{\textup{lat}})}
%     \quad\text{for all}\quad g \in \Hspace^{1}(0,1)^{m},
%   \end{align*}
%   then we say $\portOp \phi = f$ and $\portOp \phi \in \Lp{2}(0,1)^{m}$.
% \end{definition}
% \TODO{eventuell $\phi \in \Vtaud(\Gamma_{\textup{lat}})$ und dafuer $g \in \cH^{1}(0,1)^{m}$}

% \TODO{Question: does the weak definition conincide with a strong definition (defined by limits of smooth functions)?}

\subsection{The port operators}\label{sec:portop}

Next we provide a rigorous functional-analytic treatment of the port operators $\portOp_\textup{mag}$ and $\portOp_\textup{el}$, which were previously introduced in \Cref{sec:coupling-idea} at a formal level. Namely, for the parameterization $\Phi_i$ of the
lateral boundary $\Gamma_{i,\textup{lat}}$ of the $i$th cable (see \eqref{eq:Gammalat}), we denote the inverse by $\Psi_i$. That is,
\begin{align*}
\Psi_{i}\colon
\mapping{\Gamma_{i,\textup{lat}}}{[0,1]\times \lparen -\uppi,\uppi \rbrack}{\Phi(\eta,\theta)}{(\eta,\theta),}
\qquad\Psi_{i1}\colon
\mapping{\Gamma_{i,\textup{lat}}}{[0,1]}{\Phi(\eta,\theta)}{\eta.}
\end{align*}
We
introduce the operators
\begin{align*}
  \portOp_{i,\textup{mag}} \colon&\;
  \mapping{\Lp{2}(\Gamma_{i,\textup{lat}};\C^{3})}{\Lp{2}((0,1);\C)}{g}{%
    \bigg(\eta \mapsto\displaystyle \vphantom{\oint}\smash{\oint\limits_{\mathclap{\alpha_{i}(\eta) + \beta_{i\eta}}}} \mspace{8mu} g \times \nu \cdot \dx[s]\bigg),}
    % fix spacing
    \vphantom{\mapping{\Lp{2}\C^{3}}{\Lp{2}}{}{\displaystyle\oint\limits_{\beta_{i\eta}}}}\\[3mm]
  \portOp_{i,\textup{el}} \colon&\;
  \mapping{\Lp{2}((0,1);\C)}{\Lp{2}(\Gamma_{i,\textup{lat}};\C^{3})}{f}{(\grad \Phi_{i}\circ\Psi_i)^{\dagger}
    \begin{pmatrix}
      f \circ \Psi_{i1}\\
      0
    \end{pmatrix}.
  }\end{align*}
%\Timo{Hier ist noch was unklar: $\Psi_i$ muss noch erklärt werden, und das Argument von $\nabla\Phi$ ist $(\eta,\theta)$}
The coupling conditions \eqref{eq:magcouple} and \eqref{eq:elcouple} can now be reformulated to \eqref{eq:couplcondop}
via% the {\em port operators} 
\begin{align}
  \portOp_{\textup{mag}} \colon&\;
  \mapping{\Lp{2}(\Gamma_{\textup{lat}};\C^{3})}{\Lp{2}((0,1);\C^k)}{g}{%
    \begin{pmatrix}
        \portOp_{1,\textup{mag}}\Big(\left.g\right|_{\Gamma_{1,\textup{lat}}}\Big)\\\vdots\\\portOp_{k,\textup{mag}}\Big(\left.g\right|_{\Gamma_{k,\textup{lat}}}\Big)
    \end{pmatrix},}\label{eq:magop}\\[4mm]
  \portOp_{\textup{el}} \colon&\;
  \mapping{\Lp{2}((0,1);\C^k)}{\Lp{2}(\Gamma_{\textup{lat}};\C^{3})}{\begin{pmatrix}f_1\\\vdots\\f_k\end{pmatrix}}{%
 g \quad \text{with} \quad g\big|_{\Gamma_{i,\textup{lat}}} = \portOp_{i,\textup{el}}f_i \quad\forall\, i=1,\ldots,k.}\label{eq:elop}
\end{align}
%Well-definition of $\portOp_{\textup{mag},i}$ follows by \cite[Lemma~A.1]{ClemensGuReSkrepek25}.
A key property for our subsequent analysis is that these two operators are adjoint to each other, a fact that was established in \cite{ClemensGuReSkrepek25} in the context of deriving an energy balance.
\begin{proposition}[{\cite[Prop.~4.2]{ClemensGuReSkrepek25}}]
The operators in \eqref{eq:magop} and \eqref{eq:elop} fulfill
\[\portOp_\textup{el}\adjun=\portOp_\textup{mag}.\]
\end{proposition}
% In the sequel, we show that $\portOp_{\textup{mag}}$ and $\portOp_{\textup{el}}$ are adjoint to each other,
% which means that the coupling scheme depicted in \Cref{fig:coupled-diagram} 
\noindent A schematic representation of the overall coupling structure is given in \Cref{fig:coupled-diagram_adj}.

\begin{figure}[htbp]
  \centering
    \begin{tikzpicture}[scale=0.8]
      \tikzstyle{mynode1} = [rectangle, minimum width=2.5cm, minimum height=0.8cm, text centered, draw=black]
      \tikzstyle{mynode2} = [rectangle, minimum width=0.8cm, minimum height=0.8cm, text centered, draw=black]

      \coordinate (vh) at (0,-2); % horizontal
      \coordinate (vv) at (2,0); % vertical

      \node[mynode1] (H3) at (0,0) {\parbox[c]{2cm}{\centering Transmission\\line}};
      \node[mynode1] (H1) at ($(H3) + 2*(vh)$)  {\parbox[c]{2cm}{\centering Maxwell's\\equations}};

      \node[mynode2] (A) at ($(H3) + (vh) + (vv)$) {$\phantom{-}\portOp_\textup{el}$};
      \node[mynode2] (B) at ($(H3) + (vh) - (vv)$) {$-\portOp_\textup{el}\adjun$};

      % internal coupling
      \draw[-Latex] (H1.180) -- (B |- H1.180) -- node[left] {$\nu\times\bm{H}$} (B);
      \draw[-Latex] (B) -- node[left] {$\bm{I}_{\ext}$}(B |- H3.185) -- (H3.185);  \draw[-Latex] (H3.-5) -- (A |- H3.-5) -- node[right] {$\bm{E}_{\ext}$} (A);
      \draw[-Latex] (A) -- node[right] {$\pi_\tau \bm{E}$} (A |- H1.0) -- (H1.0);
      \draw[-Latex] ($(H3|- H3.175) - (vv) + (-1,0)$) -- node[above] {$u$} (B|- H3.175) -- (H3.175);
      \draw[Latex-] ($(H3|- H3.5) + (vv) + (1,0)$) -- node[above] {$y$} (A|- H3.5) -- (H3.5);
    \end{tikzpicture}
    \caption{\label{fig:coupled-diagram_adj}Coupled ports}
\end{figure}
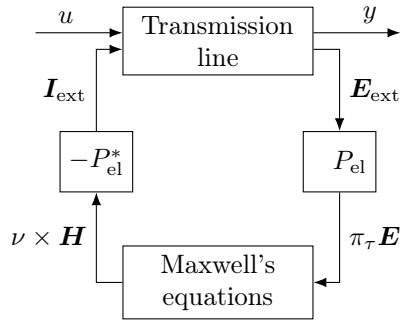

\subsection{Lifting of a \texorpdfstring{$\Lp{2}(0,1)$}{L2(0,1)} function to \texorpdfstring{$\Hspace(\rot,\Omega)$}{H(rot,Omega)}}

Note that the coupling operator $\portOp_{\textup{el}}$ maps into $\Lp{2}(\Gamma_{\textup{lat}})$, 
and we aim to impose the coupling condition $\portOp_{\textup{el}}\bm{E}_{\ext} = \tantr\bm{E}$.
However, the tangential trace operator $\tantr$ maps to a~space which is neither a~subset nor a~superset of $\Lp{2}(\Gamma_{\textup{lat}})$, see \cite{BuCoSh02}.
%\TODO{fix reference. and not surjective, because of closed graph theorem applied on $\iota_{+}^{-1}$. not completely trivial}
Hence, one must ensure that the coupling condition is well-defined, that is, 
for a given output $\bm{E}_{\ext}$ of the transmission line, there exists a field 
$\bm{E} \in \Hspace(\rot,\Omega)$ satisfying $\portOp_{\textup{el}}\bm{E}_{\ext} = \tantr\bm{E}$.

Note that the output $\bm{E}_{\ext} = \partial_{\eta} (\bm{C}^{-1}\bm{q}) = \partial_{\eta}\bm{V}$ employs the derivative of a~part of the state of the transmission lines.  
Furthermore, $\portOp_{\textup{el}}\bm{E}_{\ext}$ can be viewed as the result of applying the chain rule, 
representing the tangential derivative (or gradient) of $\bm{V} \circ \Psi$.  
This observation suggests that $\portOp_{\textup{el}}\bm{E}_{\ext}$ provides a natural candidate 
for extension from the boundary $\Gamma_{\textup{lat}}$ into the domain~$\Omega$.

Note that we can extend the paths $\alpha_{i}$ in both direction by an $\epsilon > 0$ to a path $\hat{\alpha}_{i} \in \conC^{2}((-\epsilon,1+\epsilon))$.
This allows us to extend $\Phi_{i}$ to a $\conC^{2}$ diffeomorphism (onto its image) by
\begin{equation}\label{eq:diffeomorphism-for-lifting}
  \hat{\Phi}_{i} \colon
  \mapping{(-\epsilon, 1 + \epsilon) \times (-\uppi,\uppi) \times (-\epsilon,\epsilon)}{\R^{3}}{(\eta,\theta,s)}{\hat{\alpha}_{i}(\eta) + (1+s) \beta_{i\eta}(\theta),}
\end{equation}
for $\epsilon > 0$ sufficiently small. We will denote the inverse of $\hat{\Phi}_{i}$ by
\begin{equation*}
  \hat{\Psi}_{i}\colon \ran \hat{\Phi}_{i} \to (-\epsilon,1 + \epsilon) \times (-\uppi,\uppi) \times (-\epsilon,\epsilon).
\end{equation*}
Clearly $\hat{\Psi}_{i} = \Psi_{i}$ on $\Gamma_{i,\textup{lat}}$. Hence,\footnote{In order to simplify notation we say $\bm{V}(\eta,\theta,s) \coloneq \bm{V}(\eta)$ to make compositions work.} (by the chain rule $(\grad \hat{\Phi}_{i})^{-1} \circ \hat{\Psi}_{i} = \grad \hat{\Psi}_{i}$)
\begin{align}\label{eq:lifting-by-gradient-field}
  \begin{split}
  \portOp_{i,\textup{el}} \bm{E}_{\ext}
  &= \portOp_{i,\textup{el}} \partial_{\eta}\bm{V}
  = (\grad \Phi_{i})^{\dagger}
  \begin{pmatrix} \partial_{\eta} \bm{V}  \\ 0 \end{pmatrix}
  \circ \Psi_{i}
  \\
  &= (\grad \hat{\Phi}_{i})^{-1}
  \begin{pmatrix} \partial_{\eta} \bm{V} \\ 0 \\ 0 \end{pmatrix} \circ \hat{\Psi}_{i}\big\vert_{\Gamma_{i,\textup{lat}}}
  = \grad (\bm{V} \circ \hat{\Psi}_{i})\big\vert_{\Gamma_{i,\textup{lat}}}.
  \end{split}
\end{align}
Note that $\grad (\bm{V} \circ \hat{\Psi}_{i})$ has vanishing $\rot$ as a gradient field. So in order to extend it from $\ran \hat{\Phi}_{i}$ we multiply it by a cut-off function $\chi_{i} \in \Cc(\ran \hat{\Phi}_{i})$ that is $1$ on $\Gamma_{i,\textup{lat}}$. Then we can extend $\chi_{i} \grad(\bm{V} \circ \hat{\Psi}_{i})$ by zero outside of $\ran \hat{\Phi}_{i}$ and obtain $\chi_{i} \grad(\bm{V} \circ \hat{\Psi}_{i}) \in \Hspace(\rot,\Omega)$ by the product rule for $\rot$.

\begin{lemma}\label{le:lifting-to-Hrot}
  The mapping
  \begin{equation*}
    \hat{\portOp}_{i,\textup{el}}\colon
    \mapping{\Hspace^{1}((0,1))}{\Hspace(\rot,\Omega)}{\bm{V}}{\chi_{i} \grad \hat{\Psi}_{i}
    \begin{pmatrix}\big(\partial_{\eta}\bm{V}\big) \circ \hat{\Psi}_{i} \\ 0 \\ 0\end{pmatrix}}
  \end{equation*}
  is bounded and satisfies $\hat{\portOp}_{i,\textup{el}} \bm{V} \big\vert_{\Gamma_{i,\textup{lat}}} = \portOp_{i,\textup{el}} \big(\partial_{\eta}\bm{V}\big)$. In particular, $\hat{\portOp}_{i,\textup{el}} \bm{V} \in \Hspace(\rot,\Omega)$ such that $\tantr \hat{\portOp}_{i,\textup{el}} \bm{V} = \portOp_{i,\textup{el}} \big(\partial_{\eta}\bm{V}\big)$.
\end{lemma}

\begin{proof}
  %Let $\bm{V}$ be any anti-derivative of $\bm{E}_{\ext}$. Then $\hat{\portOp}_{i,\textup{el}} \bm{E}_{\ext}$ can be written as $\chi_{i} \grad (\bm{V} \circ \hat{\Psi}_{i})$.
  The equality $\hat{\portOp}_{i,\textup{el}} \bm{V} \big\vert_{\Gamma_{i,\textup{lat}}} = \portOp_{i,\textup{el}} \partial_{\eta} \bm{V}$ follows from \eqref{eq:lifting-by-gradient-field} and the boundedness from
  \begin{align*}
    \norm{\chi_{i} \grad (\bm{V} \circ \hat{\Psi}_{i})}_{\Hspace(\rot,\Omega)}^{2}
    &= \norm{\chi_{i} \grad (\bm{V} \circ \hat{\Psi}_{i})}_{\Lp{2}(\Omega)}^{2} + \norm{\grad \chi_{i} \times \grad(\bm{V} \circ \hat{\Psi}_{i})}_{\Lp{2}(\Omega)}^{2}
    \\
    &\leq (\norm{\chi_{i}}_{\infty}^{2} + \norm{\grad \chi_{i}}_{\infty}^{2}) \norm{\grad (\bm{V}\circ\hat{\Psi}_{i})}_{\Lp{2}(\ran \hat{\Phi}_{i})}^{2}
    \\
    &\leq C_{i}(\norm{\chi_{i}}_{\infty}^{2} + \norm{\grad \chi_{i}}_{\infty}^{2}) \norm{\grad \hat{\Psi}_{i}}_{\infty}^{2}\norm{\partial_{\eta} \bm{V}}_{\Lp{2}((0,1))}^{2}
    \\
    &\leq C_{i}(\norm{\chi_{i}}_{\infty}^{2} + \norm{\grad \chi_{i}}_{\infty}^{2}) \norm{\grad \hat{\Psi}_{i}}_{\infty}^{2}\norm{\bm{V}}_{\Hspace^{1}((0,1))}^{2},
  \end{align*}
  where $C_{i}$ is a suitable constant that corresponds to the cylindrical coordinates imposed by $\hat{\Phi}_{i}$
\end{proof}

\section{Analysis of the coupled system}\label{sec:coupling-operator-theory}

\subsection{Formulation as an infinite-dimensional system}
For our coupled system we regard the {\em state} $x(t)$ and {\em effort} $e(t)$, which are given by
\begin{subequations}
\label{eq:gensetup}
\begin{align}\label{eq:stateeffortH}
  %x(t) =
  %\begin{pmatrix}
  %  \bm{\psi}(t,\cdot) \\ \bm{B}(t,\cdot) \\ \bm{q}(t,\cdot) \\ \bm{D}(t,\cdot)
  %\end{pmatrix},
  %\quad
  e(t) =
  \begin{pmatrix}
    \bm{I}(t,\argdot) \\ \bm{H}(t,\argdot) \\ \bm{V}(t,\argdot) \\ \bm{E}(t,\argdot)
  \end{pmatrix}
  =
  \underbrace{
  \begin{bmatrix}
  \bm{L}(\argdot)^{-1} & 0 & 0 & 0 \\
  0 & \bm{\mu}(\argdot)^{-1} & 0 & 0 \\
  0 & 0 & \bm{C}(\argdot)^{-1} & 0 \\
  0 & 0 & 0 & \bm{\epsilon}(\argdot)^{-1}
  \end{bmatrix}
  }_{=\mathrlap{\hamiltonian}}
  \underbrace{
  \begin{pmatrix}
    \bm{\psi}(t,\argdot) \\ \bm{B}(t,\argdot) \\ \bm{q}(t,\argdot) \\ \bm{D}(t,\argdot)
  \end{pmatrix}
  }_{=\mathrlap{x(t)}},
\end{align}
and evolve in the
\emph{state space}
\begin{equation}
  \X =
  \begin{bmatrix}
    \Lp{2}((0,1);\C^{k}) \\
    \Lp{2}(\Omega;\C^{3}) \\
    \Lp{2}((0,1);\C^{k}) \\
    \Lp{2}(\Omega;\C^{3})
  \end{bmatrix}_{\times}\label{eq:Xspace}
\end{equation}
endowed with the canonical inner product. Then, by using \Cref{ass:tl} and \Cref{ass:Maxwell}, $\hamiltonian$ defines a~bounded and bijective operator on $\X$ via pointwise multiplication, and a bounded {\em damping operator}
\begin{equation}\label{eq:frakR}
\mathfrak{R}=  \begin{bmatrix}
    \bm{R}(\argdot) & \phantom{-}0\phantom{-} &\phantom{-}0\phantom{-} & \phantom{-}0\phantom{-}\\
    \phantom{-}0\phantom{-} & 0 & 0 & 0 \\
    0 & 0 & \bm{G}(\argdot) & 0 \\
    0 & 0 & 0 & \bm{\sigma}(\argdot)
  \end{bmatrix}.
\end{equation}
\end{subequations}
%Additionally, these assumptions also indicate that $\hamiltonian$ is bijective.

\noindent In order to give a vague understanding of our goals we are a little bit imprecise in the following few lines.
Our coupled system is defined by the differential operator (specifically, its operator closure)
\begin{align*}
  \mathfrak{J} &=
  \begin{bmatrix}
    0 & 0 & -\tfrac{\dd}{\dd \eta} & 0 \\
    0 & 0 & 0 & -\rot \\
    -\tfrac{\dd}{\dd \eta} & \tfrac{\dd}{\dd \eta}\portOp_{\textup{el}}\adjun \tanxtr & 0 & 0 \\
    0 & \rot & 0 & 0
  \end{bmatrix},
  \\
  \dom \mathfrak{J} &=
  \dset*{
  \begin{pmatrix}
    \bm{I} \\ \bm{H} \\ \bm{V} \\ \bm{E}
  \end{pmatrix}
           \in
  \begin{bmatrix}
    \Hspace^{1}((0,1);\C^{k}) \\
    \conC^{\infty}(\cl{\Omega};\C^3) \\
    \Hspace^{1}((0,1);\C^{k}) \\
    \cH_{\Gamma_{\textup{r}}}(\rot,\Omega)
  \end{bmatrix}_{\times}
  }{
  \tantr \bm{E} = \portOp_{\textup{el}} \tfrac{\dd}{\dd \eta}\bm{V} \ \text{on}\ \Gamma_{\textup{lat}}
                      %\tanxtr\bm{E} \in \Lp{2}_{\tau}(\Gamma_{\textup{lat}})
                      }
\end{align*}
through the (misleadingly appearing as autonomous) differential equation
\begin{align*}
  \dot{x} = (\mathfrak{J}-\mathfrak{R})\hamiltonian x.
\end{align*}
Note that the boundary condition $\tantr \bm{E} \big\vert_{\Gamma_{\textup{r}}} = 0$ (which comprises \eqref{eq:bndcondcomp} and \eqref{eq:bndcondend})
is encoded in $\bm{E} \in \Cc_{\Gamma_{\textup{r}}}(\Omega;\C^3)$. The operator closure does not alter this tangential boundary condition, since the tangential trace is continuous w.r.t.\ $\norm{\argdot}_{\Hspace(\rot,\Omega)}$.

Let two efforts $e^i = \big(\bm{I}^i, \bm{H}^i, \bm{V}^i, \bm{E}^i\big)$, $i = 1,2$, be given 
(we use row-vector notation for layout reasons).  
Using \eqref{eq:input1}, the corresponding total currents are defined as 
$\bm{I}^i_{\textup{tot}} \coloneq \bm{I}^i - \portOp_{\textup{el}}^{*}\bm{H}^i$.  
Our next objective is to show that a boundary triple (see \Cref{def:boundary-triple}) can be associated with the operator~$\mathfrak{J}$.  
In particular, for $e^1, e^2 \in \dom \mathfrak{J}$, partitioned and denoted as in \eqref{eq:stateeffortH},  
we aim to establish the abstract Green identity
\begin{align*}
  \MoveEqLeft[3]
  \scprod{\mathfrak{J}e^{1}}{e^{2}}_\X + \scprod{e^{1}}{\mathfrak{J}e^{2}}_\X \\
  &= -\scprod{\bm{V}^{1}(1)}{\bm{I}^{2}_{\textup{tot}}(1)}_{\C^k} + \scprod{\bm{V}^{1}(0)}{\bm{I}^{2}_{\textup{tot}}(0)}_{\C^k}
  \\ &\qquad\quad
  - \scprod{\bm{I}^{1}_{\textup{tot}}(1)}{\bm{V}^{2}(1)}_{\C^k} + \scprod{\bm{I}^{1}_{\textup{tot}}(0)}{\bm{V}^{2}(0)}_{\C^k} \\
  &= \scprod*{\begin{pmatrix}\bm{V}^{1}(0) \\ -\bm{V}^{1}(1)\end{pmatrix}}{\begin{pmatrix}\bm{I}^{2}_{\textup{tot}}(0) \\ \bm{I}^{2}_{\textup{tot}}(1)\end{pmatrix}}_{\C^{2k}}
  + \scprod*{\begin{pmatrix}\bm{I}^{1}_{\textup{tot}}(0) \\ \bm{I}^{1}_{\textup{tot}}(1)\end{pmatrix}}{\begin{pmatrix}\bm{V}^{2}(0) \\ -\bm{V}^{2}(1)\end{pmatrix}}_{\C^{2k}}.
\end{align*}
%\Nathanael{Das kann ein bisschen verwirrend sein, weil wir hier die Argumente als raeumliche variable sehen, aber am anfang sind die funktionen $\bm{V}$ und $\bm{I}$ als zeitlich abhaengige funktionen eingefuehrt worden.}
%We can also formulate this in the usual boundary flow and boundary effort variables that are favored in \cite{JaZw12}, i.e., for
%$e = \begin{bsmallmatrix}I \\ H \\ V \\\bm{E}\end{bsmallmatrix}$ we define
%\begin{align*}
 % f_{\partial,e} = \frac{1}{\sqrt{2}} \begin{bmatrix}V(1) - V(0) \\ I(1) - I(0) \end{bmatrix}
  %\quad\text{and}\quad
  %e_{\partial,e} = \frac{1}{\sqrt{2}} \begin{bmatrix}I(1) + I(0) \\ V(1) + V(0)\end{bmatrix}.
%\end{align*}
%\TODO{Die Flows sollten besser über Spannungen am Rand, und Efforts besser über Ströme am Rand definiert werden. Läuft mathematisch auf das Gleiche hinaus, und macht physikalisch mehr Sinn.}

%This leads to
%\begin{align*}
%  \scprod{\mathfrak{J}e_{1}}{e_{2}} + \scprod{e_{1}}{\mathfrak{J}e_{2}} = \scprod{e_{\partial,e_{1}}}{f_{\partial,e_{2}}} + \scprod{f_{\partial,e_{1}}}{e_{\partial,e_{2}}}.
%\end{align*}

\subsection{Boundary triples and the semigroup property}
In order to properly define boundary conditions that admit existence and uniqueness of solutions we analyze the blocks of the operator $\mathfrak{J}$. We define
\begin{subequations}
    \label{eq:Jdef}\begin{equation}
\begin{aligned}
  \fA_{1}
  &=
  \begin{bmatrix}
    -\frac{\dd}{\dd \eta} & \frac{\dd}{\dd \eta} \portOp_{\textup{el}}\adjun \tanxtr \\
    0 & \rot
  \end{bmatrix},\;
  \fA_{2}
  =
  \begin{bmatrix}
    -\frac{\dd}{\dd \eta} & 0 \\
    0 & -\rot
  \end{bmatrix},\;
  \X_1 =
  \begin{bmatrix}
    \Lp{2}((0,1);\C^{k}) \\
    \Lp{2}(\Omega;\C^{3})
  \end{bmatrix}_{\times},
  \\
%\end{align*}
%with domains
%\begin{align*}
  \dom \fA_{1}
  &=
    \sset*{
    \begin{pmatrix}
     \bm{I} \\ \bm{H}
    \end{pmatrix}
    \in
    \begin{bmatrix}
      \Hspace^{1}((0,1);\C^k) \\
      \conC^{\infty}(\cl{\Omega};\C^3)
    \end{bmatrix}_{\times}
    } \subseteq \X_1
  \\
  \dom \fA_{2}
  &=
  \dset*{
  \begin{pmatrix}
    \bm{V} \\ \bm{E}
  \end{pmatrix}
  \in
  \begin{bmatrix}
    \Hspace^{1}((0,1);\C^k) \\
    \cH_{\Gamma_{\textup{r}}}(\rot,\Omega)
  \end{bmatrix}_{\times}
  }{\tantr \bm{E} = \portOp_{\textup{el}} \tfrac{\dd}{\dd \eta}\bm{V} \ \text{on}\ \Gamma_{\textup{lat}}} \subseteq \X_1,
  % \dset*{
  % \begin{bmatrix}
  %  \bm{V} \\\bm{E}
  % \end{bmatrix}
  % \in
  % \begin{bmatrix}
  %   \Hspace^{1}(0,1) \\
  %   \Cc_{\Gamma_{\textup{r}}}(\Omega;\C^3)
  % \end{bmatrix}_{\times}
  % }{\tanxtr\bm{E} \in \Lp{2}_{\tau}(\Gamma_{\textup{lat}})}
\end{aligned}\label{eq:U1U2}
\end{equation}
By invoking that $\X = \X_1 \times \X_1$, the differential operator $\mathfrak{J}$ can now be written as
\begin{align}
    \mathfrak{J} =
    \begin{bmatrix}
        0 & \fA_{2} \\
        \fA_{1} & 0
    \end{bmatrix}.
\end{align}
\end{subequations}
We obviously have that $\dom\mathfrak{J} = \dom \fA_{1} \times \dom \fA_{2}$.
Moreover, we introduce the following operators that will form our minimal operator
\begin{subequations}\label{eq:Jring}
\begin{align}
  \mathring{\fA}_{1}
  =
  \begin{bmatrix}
    -\frac{\dd}{\dd \eta} & \frac{\dd}{\dd \eta} \portOp_{\textup{el}}\adjun \tanxtr \\
    0 & \rot
  \end{bmatrix}
  \quad\text{and}\quad
  \mathring{\fA}_{2}
  =
  \begin{bmatrix}
    -\frac{\dd}{\dd \eta} & 0 \\
    0 & -\rot
  \end{bmatrix}
\end{align}
with domains
\begin{equation}
\begin{aligned}
  \dom \mathring{\fA}_{1}
  &=
  \dset*{%
    \begin{pmatrix}
     \bm{I} \\ \bm{H}
    \end{pmatrix}
    \in
    \begin{bmatrix}
      \Hspace^{1}((0,1);\C^k) \\
      \conC^{\infty}(\cl{\Omega};\C^{3})
    \end{bmatrix}_{\times}
    }{\bm{I} - \portOp_{\textup{el}}\adjun\tanxtr \bm{H} \in \cH^{1}((0,1);\C^{k})},
  \\
  \dom \mathring{\fA}_{2}
  &=
  \dset*{
  \begin{pmatrix}
    \bm{V} \\ \bm{E}
  \end{pmatrix}
  \in
  \begin{bmatrix}
    \cH^{1}((0,1);\C^k) \\
    \cH_{\Gamma_{\textup{r}}}(\rot,\Omega)
  \end{bmatrix}_{\times}
  }{\tantr \bm{E} = \portOp_{\textup{el}} \tfrac{\dd}{\dd \eta}\bm{V}}.
    % \dset*{
    % \begin{bmatrix}
    %  \bm{V} \\\bm{E}
    % \end{bmatrix}
    % \in
    % \begin{bmatrix}
    %   \cH^{1}((0,1);\C^k) \\
    %   \cH_{\Gamma_{\textup{r}}}(\rot,\Omega)
    % \end{bmatrix}_{\times}
    % }{\tanxtr\bm{E} \in \Lp{2}_{\tau}(\Gamma_{\textup{lat}})}
\end{aligned}\end{equation}
\end{subequations}
Note that by definition we have
\begin{equation}\label{eq:minimal-subset-maximal-operator}
  -\mathring{\fA}_{1} \subseteq \fA_{1}\quad\text{and}\quad \mathopen{}-\mathring{\fA}_{2} \subseteq \fA_{2}.
\end{equation}

\begin{lemma}
Suppose that the spatial domains are as in \Cref{sec:Omega}. Then the operators $\mathring{\fA}_{1}$, $\mathring{\fA}_{2}$, $\fA_{1}$ and $\fA_{2}$ as in \eqref{eq:U1U2} and \eqref{eq:Jring}
 are densely defined.
\end{lemma}

\begin{proof}
  It is enough to show that $\mathring{\fA}_{1}$ and $\mathring{\fA}_{2}$ are densely defined as their domains are subsets of $\dom(\fA_{1})$ and $\dom(\fA_{2})$, respectively.
  Moreover, $\dom \mathring{\fA}_{1}$ contains $\Cc((0,1);\C^k) \times \Cc(\Omega;\C^3)$ and is therefore dense in $\Lp{2}((0,1);\C^k) \times \Lp{2}(\Omega;\C^3)$. Hence, it is left to show that $\dom \mathring{\fA}_{2}$ is dense in $\Lp{2}((0,1);\C^k) \times \Lp{2}(\Omega;\C^3)$.

  Let $\begin{psmallmatrix} \bm{V} \\ \bm{E} \end{psmallmatrix} \in \Lp{2}((0,1);\C^{k}) \times \Lp{2}(\Omega;\C^{3})$. Then there exists a sequence $(\bm{V}_{n})_{n\in\N}$ in $\Cc((0,1))$ that converges to $\bm{V}$ w.r.t.\ $\norm{\argdot}_{\Lp{2}((0,1))}$. We use the mapping $\hat{\portOp}_{\textup{el}}$ from \Cref{le:lifting-to-Hrot} to lift $ \bm{V}_n$ on $\Hspace(\rot,\Omega)$. Multiplying this by a sequence $(\chi_{n})_{n\in\N}$ of $\Cc$ cut-off functions that are $1$ on $\Gamma_{\textup{lat}}$ and satisfy $\norm{\chi_{n}}_{\Lp{2}} \leq \frac{1}{n}\frac{1}{\norm{\partial_{\eta}\bm{V}_{n}}_{\infty}}$, we define
  \begin{equation*}
    \bm{E}_{n}^{1} \coloneq \chi_{n} \hat{\portOp}_{\textup{el}} \bm{V}_{n},
  \end{equation*}
  which satisfies the boundary condition $\tantr \bm{E}_{n}^{1} = \portOp_{\textup{el}} \bm{V}_{n}$ and
  \begin{equation*}
    \norm{\bm{E}^{1}_{n}}_{\Lp{2}(\Omega)} \leq \norm{\chi_{n}}_{\Lp{2}(\Omega)} \norm{\partial_{\eta} \bm{V}_{n}}_{\infty} \to 0.
  \end{equation*}
  For the given $\bm{E}$ there exists a sequence $(\bm{E}^{2}_{n})_{n\in\N}$ in $\Cc(\Omega)$ that converges to $\bm{E}$ w.r.t.\ $\norm{\argdot}_{\Lp{2}(\Omega)}$. Hence, we define
  \begin{equation*}
    \bm{E}_n \coloneq \bm{E}^{1}_{n} + \bm{E}^{2}_{n}.
  \end{equation*}
  This makes sure that
  $
  \begin{psmallmatrix}
    \bm{V}_{n} \\ \bm{E}_{n}
  \end{psmallmatrix}
  \in \dom \mathring{\fA}_{1}
  $ and
  $
  \begin{psmallmatrix}
    \bm{V}_{n} \\ \bm{E}_{n}
  \end{psmallmatrix}
  \to
  \begin{psmallmatrix}
    \bm{V} \\ \bm{E}
  \end{psmallmatrix}
  $.
\end{proof}

\begin{lemma}\label{le:A1-as-adjoint-of-A2}
Suppose that the spatial domains are as in \Cref{sec:Omega}. Then the operators as in \eqref{eq:U1U2} and \eqref{eq:Jring}
 % We have the following relation between $\mathring{\fA}_{2}$ and $\fA_{1}$, and, $\mathring{\fA}_{1}$ and $\fA_{2}$.
  fulfill
  \begin{align*}
    \mathring{\fA}_{1}\adjun = -\fA_{2}
    \quad\text{and}\quad
    \fA_{1}\adjun = -\mathring{\fA}_{2}.
  \end{align*}
\end{lemma}

\begin{proof}
  We first show $\mathring{\fA}_{1}\adjun = -\fA_{2}$.
  For
  $\begin{psmallmatrix} \bm{V} \\ \bm{E} \end{psmallmatrix} \in \dom \mathring{\fA}_{1}\adjun$
  we have (by definition of the adjoint)
  \begin{align}\label{eq:adjoint-formula}
    \scprod*{\mathring{\fA}_{1}\adjun\begin{psmallmatrix} \bm{V} \\ \bm{E} \end{psmallmatrix}}{\begin{psmallmatrix}\bm{I} \\ \bm{H} \end{psmallmatrix}}_{\X_1}
    = \scprod*{\begin{psmallmatrix} \bm{V} \\ \bm{E} \end{psmallmatrix}}{\mathring{\fA}_{1}\begin{psmallmatrix}\bm{I} \\ \bm{H} \end{psmallmatrix}}_{\X_1}
    \quad\text{for all}\quad
    \begin{psmallmatrix}
     \bm{I} \\ \bm{H}
    \end{psmallmatrix}
    \in \dom \mathring{\fA}_{1}.
  \end{align}
  Hence, we can calculate the action of $\mathring{\fA}_{1}\adjun$ by testing component-wise.
  \begin{itemize}[itemsep=4pt]
    \item For every $\bm{I} \in \Cc((0,1);\C^{k})$ we have
          $\begin{psmallmatrix}\bm{I} \\ \bm{0} \end{psmallmatrix} \in \dom \mathring{\fA}_{1}$ and
          \begin{align}\label{eq:test-first-component}
            \scprod*{\pi_{1} \mathring{\fA}_{1}\adjun\begin{psmallmatrix}\bm{V} \\ \bm{E}\end{psmallmatrix}}{\bm{I}}_{\X_1}
            = \scprod*{\mathring{\fA}_{1}\adjun\begin{psmallmatrix}\bm{V} \\ \bm{E}\end{psmallmatrix}}{\begin{psmallmatrix}\bm{I} \\ \bm{0}\end{psmallmatrix}}_{\X_1}
            \stackrel{\mathclap{\eqref{eq:adjoint-formula}}}{=}
            \scprod*{\bm{V}}{-\tfrac{\dd}{\dd \eta}\bm{I}}_{\Lp{2}((0,1);\C^{k})},
          \end{align}
          where $\pi_{1}$ is the projection on the first component (the $\Lp{2}((0,1);\C^k)$ component). Hence, $\bm{I} \in \Hspace^{1}((0,1);\C^{k})$ and
          $\pi_{1}\mathring{\fA}_{1}\adjun\begin{psmallmatrix}\bm{V} \\\bm{E}\end{psmallmatrix} = \frac{\dd}{\dd \eta} \bm{V}$.

    \item For every $\bm{H} \in \Cc(\Omega;\C^{3})$ we have
          $\begin{psmallmatrix} \bm{0} \\ \bm{H} \end{psmallmatrix} \in \dom \mathring{\fA}_{1}$ and
          \begin{align}\label{eq:test-second-component}
            \scprod*{\pi_{2} \mathring{\fA}_{1}\adjun \begin{psmallmatrix}\bm{V} \\ \bm{E}\end{psmallmatrix}}{\bm{H}}_{\X_1}
            = \scprod*{\mathring{\fA}_{1}\adjun\begin{psmallmatrix}\bm{V} \\ \bm{E}\end{psmallmatrix}}{\begin{psmallmatrix} \bm{0} \\ \bm{H}\end{psmallmatrix}}_{\X_1}
            \stackrel{\eqref{eq:adjoint-formula}}{=}
            \scprod*{\bm{E}}{\rot\bm{H}}_{\Lp{2}(\Omega;\C^{3})},
          \end{align}
          where $\pi_{2}$ is the projection on the second component (the $\Lp{2}(\Omega;\C^3)$ component). Hence, $\bm{E} \in \Hspace(\rot,\Omega)$ and
          $\pi_{2}\mathring{\fA}_{2}\adjun \begin{psmallmatrix}\bm{V} \\ \bm{E}\end{psmallmatrix} = \rot\bm{E}$.
  \end{itemize}
  We conclude
  \begin{align*}
    \mathring{\fA}_{1}\adjun \begin{pmatrix}\bm{V} \\ \bm{E}\end{pmatrix} = \begin{pmatrix} \frac{\dd}{\dd \eta} \bm{V} \\ \rot\bm{E}\end{pmatrix}
    =
    \begin{bmatrix}\frac{\dd}{\dd \eta} & 0 \\ 0 & \rot \end{bmatrix}\begin{pmatrix}\bm{V} \\ \bm{E}\end{pmatrix}
    \quad\text{for}\quad
    \begin{pmatrix}\bm{V} \\ \bm{E}\end{pmatrix} \in \dom \mathring{\fA}_{1}\adjun.
  \end{align*}

  Now use the notation from \eqref{eq:int-by-parts-with-boundary-conditions} for the trace operators. By the boundary conditions of $\mathring{\fA}_{1}$ we have for
  $
    \begin{psmallmatrix}\bm{I} \\ \bm{H}\end{psmallmatrix}
    \in \dom \mathring{\fA}_{1}
  $
  and
  $\begin{psmallmatrix}\bm{V} \\ \bm{E}\end{psmallmatrix} \in \dom \mathring{\fA}_{1}\adjun$
  \begin{align*}
  \MoveEqLeft[2]
  \scprod*{\begin{psmallmatrix}\bm{V} \\ \bm{E}\end{psmallmatrix}}{\mathring{\fA}_{1}\begin{psmallmatrix}\bm{I} \\ \bm{H}\end{psmallmatrix}}_{\X_1}
    = \scprod*{\begin{psmallmatrix}\bm{V} \\ \bm{E}\end{psmallmatrix}}{\begin{psmallmatrix}-\frac{\dd}{\dd \eta}\bm{I} + \frac{\dd}{\dd \eta}\portOp_{\textup{el}}\adjun \tanxtr\bm{H} \\ \rot\bm{H}\end{psmallmatrix}}_{\X_1}
    \\[1ex]
    &= \scprod[\big]{\tfrac{\dd}{\dd \eta} \bm{V}}{\bm{I}}_{\Lp{2}((0,1);\C^{k})}
      - \scprod{\tfrac{\dd}{\dd \eta}\bm{V}}{\portOp_{\textup{el}}\adjun \tanxtr\bm{H}}_{\Lp{2}((0,1);\C^{k})}
    \\
    & \hspace{4.5cm}
    + \scprod{\rot\bm{E}}{\bm{H}}_{\Lp{2}(\Omega;\C^{3})} + \dualprod{\tantr\bm{E}}{\tanxtr\bm{H}}_{\Vtau,\Vtaud} \\[1ex]
    &= \scprod*{\mathring{\fA}_{1}\adjun\begin{psmallmatrix}\bm{V} \\ \bm{E}\end{psmallmatrix}}{\begin{psmallmatrix}\bm{I} \\ \bm{H}\end{psmallmatrix}}_{\X_1}
      - \Big[\scprod{\portOp_{\textup{el}} \tfrac{\dd}{\dd \eta}\bm{V}}{\tanxtr\bm{H}}_{\Lp{2}(\Gamma_{\textup{lat}})} - \dualprod{\tantr\bm{E}}{\tanxtr\bm{H}}_{\Vtau,\Vtaud}\Big],
  \end{align*}
  which implies $\dualprod{\tantr\bm{E}}{\tanxtr\bm{H}}_{\Vtau,\Vtaud} = \scprod{\portOp_{\textup{el}}\tfrac{\dd}{\dd \eta} \bm{V}}{\tanxtr\bm{H}}_{\Lp{2}(\Gamma_{\textup{lat}})}$
  for all
  $\bm{H} \in \conC^{\infty}(\cl{\Omega};\C^3)$.
  Therefore, $\tantr\bm{E} = \portOp_{\textup{el}} \tfrac{\dd}{\dd \eta}\bm{V}$ and
  $\begin{psmallmatrix}\bm{V} \\ \bm{E}\end{psmallmatrix} \in \dom \fA_{2}$ and in turn $\mathring{\fA}_{1}\adjun \subseteq -\fA_{2}$. On the other hand it is straightforward to show $-\fA_{2} \subseteq \mathring{\fA}_{1}\adjun$, which proves $\mathring{\fA}_{1}\adjun = -\fA_{2}$.

  In order to show $\fA_{1}\adjun = -\mathring{\fA}_{2}$ we regard $\begin{psmallmatrix} \bm{V} \\ \bm{E} \end{psmallmatrix} \in \dom \fA_{1}\adjun$ and the equation
  \begin{align*}
    \scprod[\Big]{\fA_{1}\adjun\begin{psmallmatrix} \bm{V} \\ \bm{E} \end{psmallmatrix}}{\begin{psmallmatrix}\bm{I} \\ \bm{H} \end{psmallmatrix}}_{\X_1}
    = \scprod[\Big]{\begin{psmallmatrix} \bm{V} \\ \bm{E} \end{psmallmatrix}}{\fA_{1}\begin{psmallmatrix}\bm{I} \\ \bm{H} \end{psmallmatrix}}_{\X_1}
    \quad\text{for all}\quad
    \begin{psmallmatrix}
     \bm{I} \\ \bm{H}
    \end{psmallmatrix}
    \in \dom \fA_{1}.
  \end{align*}
  We can repeat the arguments \eqref{eq:test-first-component} and \eqref{eq:test-second-component} of the first part of the proof to conclude
  $\bm{V} \in \Hspace^{1}((0,1);\C^k)$ and $\bm{E} \in \Hspace(\rot,\Omega)$.
  Moreover, since we can choose $\bm{I} \in \Hspace^{1}((0,1);\C^k)$ in~\eqref{eq:test-first-component} we even obtain $\bm{V} \in \cH^{1}((0,1);\C^k)$. Similarly, choosing $\bm{H} \in \Cc_{\Gamma_{\textup{lat}}}(\Omega; \C^{3})$ implies $\bm{E} \in \cH_{\Gamma_{\textup{r}}}(\rot,\Omega)$. Again repeating the remaining steps of the first part of the proof yields $\fA_{1}\adjun = -\mathring{\fA}_{2}$.
\end{proof}

\begin{corollary}
Suppose that the spatial domains are as in \Cref{sec:Omega}. Then the operators as in \eqref{eq:U1U2} and \eqref{eq:Jring} fulfill
\begin{enumerate}
    \item $\mathring{\fA}_{2}$ and $\fA_{2}$ are closed,
    \item $\mathring{\fA}_{1}$ and $\fA_{1}$ are closable,
    \item the closures of $\mathring{\fA}_{1}$ and $\fA_{1}$ fulfill
  \begin{equation*}
    \cl{\mathring{\fA}_{1}} = -\fA_{2}\adjun \quad\text{and}\quad \cl{\fA_{1}} = -\mathring{\fA}_{2}\adjun.
  \end{equation*}
\end{enumerate}
\end{corollary}

\begin{proof}
  By \Cref{le:A1-as-adjoint-of-A2} $\fA_{2} = -\mathring{\fA}_{1}\adjun$ and $\mathring{\fA}_{2} = -\fA_{1}\adjun$, which immediately implies the closedness of these operators. Since $\mathring{\fA}_{2}$ and $\fA_{2}$ are densely defined, we conclude that $\mathring{\fA}_{2}\adjun$ and $\fA_{2}\adjun$ are well-defined and closed operators. Hence,
  \begin{align*}
    \cl{\mathring{\fA}_{1}} = \mathring{\fA}_{1}\adjun[2] = -\fA_{2}\adjun
    \quad\text{and}\quad
    \cl{\fA_{1}} = \fA_{1}\adjun[2] = -\mathring{\fA}_{2}\adjun
  \end{align*}
  implies that $\mathring{\fA}_{1}$ and $\fA_{1}$ are closable.
\end{proof}

\begin{lemma}\label{le:abstract-int-by-parts-on-dense-domain}
  Suppose that the spatial domains are as in \Cref{sec:Omega}, and let $\fA_1$, $\fA_2$ be as in \eqref{eq:U1U2}. Then, for all
  $\begin{psmallmatrix}\bm{V} \\ \bm{E} \end{psmallmatrix} \in \dom \fA_{2}$ and
  $\begin{psmallmatrix} \bm{I} \\ \bm{H} \end{psmallmatrix} \in \dom \fA_{1}$, it holds that
  \begingroup
  \thinmuskip=2mu%
  \medmuskip=2mu%
  \thickmuskip=3mu plus 2mu%
  \begin{equation*}
   \scprod*{\fA_{2}\begin{pmatrix}\bm{V} \\ \bm{E} \end{pmatrix}}{\begin{pmatrix} \bm{I} \\ \bm{H} \end{pmatrix}}_{\mkern-8mu\X_1}
    + \scprod*{\begin{pmatrix}\bm{V} \\ \bm{E} \end{pmatrix}}{\fA_{1}\begin{pmatrix} \bm{I} \\ \bm{H} \end{pmatrix}}_{\mkern-8mu\X_1}
    = \scprod*{\begin{pmatrix}\bm{V}(0) \\ -\bm{V}(1)\end{pmatrix}}{\begin{pmatrix} \bm{I}_{\textup{tot}}(0) \\ \bm{I}_{\textup{tot}}(1)\end{pmatrix}}_{\mkern-6mu\C^{2k}},
  \end{equation*}
  \endgroup
  where $\bm{I}_{\textup{tot}} = \bm{I} - \portOp_{\textup{el}}\adjun \tanxtr \bm{H}$.
\end{lemma}

\begin{proof}
  Let $\begin{psmallmatrix}\bm{V} \\\bm{E} \end{psmallmatrix} \in \dom \fA_{2}$ and
  $\begin{psmallmatrix} \bm{I} \\\bm{H} \end{psmallmatrix} \in \dom \fA_{1}$. Then
  \begin{align*}
  \MoveEqLeft
    \scprod*{
    \begin{bmatrix}
      -\tfrac{\dd}{\dd \eta} & 0 \\
      0 & -\rot
    \end{bmatrix}
    \begin{pmatrix}
     \bm{V} \\ \bm{E}
    \end{pmatrix}
    }
    {
    \begin{pmatrix}
      \bm{I} \\ \bm{H}
    \end{pmatrix}
    }_{\mkern-8mu\X_1}
    +
    \scprod*{
    \begin{pmatrix}
     \bm{V} \\ \bm{E}
    \end{pmatrix}
    }
    {
    \begin{bmatrix}
      -\tfrac{\dd}{\dd \eta} & \tfrac{\dd}{\dd \eta} \portOp_{\textup{el}}\adjun \tanxtr \\
      0 & \rot
    \end{bmatrix}
    \begin{pmatrix}
      \bm{I} \\ \bm{H}
    \end{pmatrix}
    }_{\mkern-8mu\X_1}
    \\[1ex]
    &=
      \scprod{-\tfrac{\dd}{\dd\eta}\bm{V}}{\bm{I}}_{\Lp{2}((0,1);\C^{k})} + \scprod{\bm{V}}{-\tfrac{\dd}{\dd \eta} (\bm{I}-\portOp_{\textup{el}}\adjun \tanxtr\bm{H})}_{\Lp{2}((0,1);\C^{k})}
      \\[0.5ex]
      &\quad + \underbrace{\scprod{-\rot\bm{E}}{\bm{H}}_{\Lp{2}(\Omega;\C^{3})} +\scprod{\bm{E}}{\rot\bm{H}}_{\Lp{2}(\Omega;\C^{3})}}_{=\mathrlap{\dualprod{\tantr \bm{E}}{\tanxtr \bm{H}}_{\Vtau,\Vtaud}}}
    \intertext{Note that $\dualprod{\tantr \bm{E}}{\tanxtr \bm{H}}_{\Vtau,\Vtaud} = \scprod{\tfrac{\dd}{\dd t} \bm{V}}{\portOp_{\textup{el}}\adjun \tanxtr \bm{H}}_{\Lp{2}((0,1);\C^{k})}$ by the boundary condition of $\fA_{2}$. Hence, we further have}
    &=
    \scprod[\big]{-\tfrac{\dd}{\dd\eta}\bm{V}}{\bm{I} - \portOp_{\textup{el}}\adjun \tanxtr \bm{H}}_{\Lp{2}((0,1);\C^{k})} + \scprod[\big]{\bm{V}}{-\tfrac{\dd}{\dd \eta} (\bm{I}-\portOp_{\textup{el}}\adjun \tanxtr\bm{H})}_{\Lp{2}((0,1);\C^{k})}
    \intertext{integration by parts formula for $\tfrac{\dd}{\dd \eta}$ finally yields}
    &= -\big(\scprod{\bm{V}(1)}{\bm{I}_{\textup{tot}}(1)}_{\C^k} - \scprod{\bm{V}(0)}{\bm{I}_{\textup{tot}}(0)}_{\C^k}\big)
    = \scprod*{\begin{psmallmatrix}\bm{V}(0) \\ -\bm{V}(1)\end{psmallmatrix}}{\begin{psmallmatrix}\bm{I}_{\textup{tot}}(0) \\ \bm{I}_{\textup{tot}}(1)\end{psmallmatrix}}_{\mkern-3mu\C^{2k}},
  \end{align*}
  which finishes the proof.
\end{proof}

\begin{definition}\label{def:bndmap}
Suppose that the spatial domains are as in \Cref{sec:Omega}, and let $\fA_1$, $\fA_2$ be defined as in \eqref{eq:U1U2}. The {\em boundary mappings} $\hat{\fB}_{1}\colon \dom \fA_{1} \to \C^{2k}$ and $\hat{\fB}_{2} \colon \dom \fA_{2} \to \C^{2k}$ are specified by
  \begin{align*}
    \hat{\fB}_{1}\begin{pmatrix}\bm{I} \\\bm{H}\end{pmatrix} = \begin{pmatrix} \bm{I}_{\textup{tot}}(0) \\ \bm{I}_{\textup{tot}}(1)\end{pmatrix}
    \quad\text{and}\quad
    \hat{\fB}_{2}\begin{pmatrix}\bm{V} \\ \bm{E}\end{pmatrix} = \begin{pmatrix} \bm{V}(0) \\ -\bm{V}(1)\end{pmatrix},
  \end{align*}
  where $\bm{I}_{\textup{tot}} \coloneq \bm{I} - \portOp_{\textup{el}}\adjun \tanxtr \bm{H}$.
\end{definition}

By \Cref{le:abstract-int-by-parts-on-dense-domain} we have the abstract integration by parts formula
\begin{align*}
  \scprod{\fA_{2}x_{2}}{x_{1}}_{\X} + \scprod{x_{2}}{\fA_{1}x_{1}}_{\X} = \scprod{\hat{\fB}_{2} x_{2}}{\hat{\fB}_{1} x_{1}}_{\C^{2k}},
\end{align*}
where $x_{1} = \begin{psmallmatrix}\bm{V} \\ \bm{E}\end{psmallmatrix} \in \dom \fA_{1}$ and $x_{2} = \begin{psmallmatrix} \bm{I} \\\bm{H}\end{psmallmatrix} \in \dom \fA_{2}$.

\begin{lemma}\label{th:boundary-mappings-surjective}
Under the preconditions in \Cref{def:bndmap},  $\hat{\fB}_{1}$ and $\hat{\fB}_{2}$ are surjective.
\end{lemma}

\begin{proof}
  For given $\begin{psmallmatrix}a \\ b\end{psmallmatrix} \in \C^{2k}$ we find a linear function $\bm{I} \colon [0,1] \to \C^{k}$ such that $\bm{I}(0) = a$ and $\bm{I}(1) = b$. Since $\bm{I}$ is a linear interpolation, we automatically have $\bm{I} \in \Hspace^{1}((0,1);\C^{k})$. We set $\bm{H} \coloneq \bm{0}$, which gives $\begin{psmallmatrix} \bm{I} \\ \bm{H} \end{psmallmatrix} \in \dom \fA_{1}$ and
  \begin{equation*}
    \hat{\fB}_{1} \begin{pmatrix} \bm{I} \\ \bm{H} \end{pmatrix}
    = \begin{pmatrix} \bm{I}(0) - \portOp_{\textup{el}}\adjun \tanxtr \bm{0} \\ \bm{I}(1) - \portOp_{\textup{el}}\adjun \tanxtr \bm{0} \end{pmatrix}
    = \begin{pmatrix}a \\ b\end{pmatrix}.
  \end{equation*}
Again by linear interpolation we can find a $\bm{V} \colon [0,1] \to \C^{k}$ such that $\bm{V}(0) = c$ and $-\bm{V}(1) = d$ for given $\begin{psmallmatrix}c \\ d\end{psmallmatrix} \in \C^{2k}$.
  Then we set $\bm{E} \coloneq \hat{\portOp}_{\textup{el}}\bm{V}$ so that $\begin{psmallmatrix}\bm{V} \\ \bm{E} \end{psmallmatrix} \in \dom \fA_{2}$ and
  \begin{equation*}
    \hat{\fB}_{2} \begin{pmatrix} \bm{V} \\ \bm{E} \end{pmatrix}
    = \begin{pmatrix} \bm{V}(0) \\ -\bm{V}(1) \end{pmatrix}
    = \begin{pmatrix}c \\ d\end{pmatrix},
  \end{equation*}
  which proves the assertion.
\end{proof}

\begin{theorem}
Under the preconditions in \Cref{def:bndmap},  $\norm*{\argdot}_{\ran \hat{\fB}_{1}}$ with
  \begin{align*}
    \norm*{\begin{pmatrix} a \\ b\end{pmatrix}}_{\ran \hat{\fB}_{1}} \coloneq
    \inf \dset*{\norm*{\begin{pmatrix}\bm{I} \\ \bm{H}\end{pmatrix}}_{\dom(\fA_{1})}}{\begin{pmatrix} a \\ b\end{pmatrix} = \begin{pmatrix}\bm{I}_{\textup{tot}}(0) \\ \bm{I}_{\textup{tot}}(1)\end{pmatrix}, \begin{pmatrix}\bm{I} \\ \bm{H}\end{pmatrix} \in \dom \fA_{1}}
  \end{align*}
  is a norm on $\C^{2k}$. In particular $\hat{\fB}_{1}\colon \dom \fA_{1} \to \C^{2k}$ is continuous, where $\dom \fA_{1}$ is equipped with the graph norm of $\fA_{1}$.\footnote{$\C^{2k}$ can be equipped with any norm as all of them are equivalent.}
\end{theorem}

Note that also $\fB_{2}\colon \dom \fA_{2} \to \C^{2k}$ is continuous, but we do not use this fact as $\fA_{2}$ is already closed.

\begin{proof}
  The property $\norm{\lambda \begin{psmallmatrix}a \\ b\end{psmallmatrix}}_{\ran \hat{\fB}_{1}} = \abs{\lambda} \norm{\begin{psmallmatrix}a \\ b\end{psmallmatrix}}_{\ran \hat{\fB}_{1}}$ is obvious and also the triangle inequality can be shown by applying the infimum on both sides of the triangle inequality for $\norm{\argdot}_{\fA_{1}}$. These are well-known techniques, e.g., for factor spaces. The only tricky step is showing that $\norm{\begin{psmallmatrix}a \\ b\end{psmallmatrix}}_{\ran \hat{\fB}_{1}} = 0$ implies $\begin{psmallmatrix}a \\ b\end{psmallmatrix} = 0$.

  Hence, let $\begin{psmallmatrix}a \\ b\end{psmallmatrix} \in \C^{2k}$ be such that $\norm{\begin{psmallmatrix}a \\ b\end{psmallmatrix}}_{\ran \hat{\fB}_{1}} = 0$. Then there exists a sequence
  $\left(\begin{psmallmatrix}\bm{I}_{n} \\ \bm{H}_{n}\end{psmallmatrix}\right)_{n\in\N}$
  in $\dom \fA_{1}$ such that
  $\begin{psmallmatrix}\bm{I}_{n} \\ \bm{H}_{n}\end{psmallmatrix} \to 0$ w.r.t.\ $\norm{\argdot}_{\fA_{1}}$ and $\begin{psmallmatrix}\bm{I}_{\textup{tot},n}(0) \\ \bm{I}_{\textup{tot},n}(1)\end{psmallmatrix} = \begin{psmallmatrix}a \\ b\end{psmallmatrix}$ for all $n\in\N$.
  Since also $\hat{\fB}_{2}$ is surjective, there exists $\begin{psmallmatrix}\bm{V} \\ \bm{E}\end{psmallmatrix} \in \dom \fA_{2}$ such that $\begin{psmallmatrix}a \\ b\end{psmallmatrix} = \begin{psmallmatrix}\bm{V}(0) \\ -\bm{V}(1)\end{psmallmatrix}$. Therefore,
  \begin{align*}
    \norm*{\begin{pmatrix}a \\ b\end{pmatrix}}_{\C^{2m}}
    &= \scprod*{\begin{pmatrix}a \\ b\end{pmatrix}}{\begin{pmatrix}a \\ b\end{pmatrix}}_{\C^{2k}}
    = \scprod*{\begin{pmatrix}\bm{V}(0) \\ -\bm{V}(1)\end{pmatrix}}{\begin{pmatrix}\bm{I}_{\textup{tot},n}(0) \\ \bm{I}_{\textup{tot},n}(1)\end{pmatrix}}_{\C^{2k}} \\
    &= \scprod*{\fA_{2}\begin{pmatrix}\bm{V} \\ \bm{E}_{n}\end{pmatrix}}{\begin{pmatrix}\bm{I}_{n} \\ \bm{H}_{n}\end{pmatrix}}_{\X_1}
    + \scprod*{\begin{pmatrix}\bm{V} \\ \bm{E} \end{pmatrix}}{\fA_{1} \begin{pmatrix}\bm{I}_{n} \\ \bm{H}_{n}\end{pmatrix}}_{\X_1}
    \to 0,
  \end{align*}
  which implies $\begin{psmallmatrix}a \\ b\end{psmallmatrix} = 0$.

  Since we have shown that $\norm{\argdot}_{\ran \hat{\fB}_{1}}$ is a norm on $\C^{2k}$ we know that it is equivalent to the standard norm on $\C^{2k}$. By construction $\hat{\fB}_{1}$ is continuous w.r.t.\ $\norm{\argdot}_{\ran \hat{\fB}_{2}}$ and therefore also w.r.t.\ the standard norm on $\C^{2k}$.
\end{proof}

Since $\hat{\fB}_{1}\colon \dom \fA_{1} \to \C^{2k}$ is continuous we can continuously extend $\hat{\fB}_{1}$ to $\dom \cl{\fA_{1}}$, where $\cl{\fA_{1}}$ is the operator closure of $\fA_{1}$. We will denote the extension of $\hat{\fB}_{1}$ still by $\hat{\fB}_{1}$.
We immediately get the following corollary from \Cref{le:abstract-int-by-parts-on-dense-domain}.

\begin{corollary}\label{th:int-by-parts-extended}
  Suppose that the spatial domains are as in \Cref{sec:Omega}.
  Further, let $\fA_1$, $\fA_2$ be defined as in \eqref{eq:U1U2}, and let $\hat{\fB}_{1}$, $\hat{\fB}_{2}$ be as in \Cref{def:bndmap}.
  Then, for
  $\begin{psmallmatrix} \bm{I} \\ \bm{H} \end{psmallmatrix} \in \dom \cl{\fA_{1}}$ and
  $\begin{psmallmatrix}\bm{V} \\ \bm{E} \end{psmallmatrix} \in \dom \fA_{2}$
   we have
  \begin{align*}
    \scprod*{\fA_{2}\begin{psmallmatrix}\bm{V} \\ \bm{E} \end{psmallmatrix}}{\begin{psmallmatrix} \bm{I} \\ \bm{H} \end{psmallmatrix}}_{\X_1}
    + \scprod*{\begin{psmallmatrix}\bm{V} \\ \bm{E} \end{psmallmatrix}}{\cl{\fA_{1}}\begin{psmallmatrix} \bm{I} \\ \bm{H} \end{psmallmatrix}}_{\X_1}
    &= \scprod*{\hat{\fB}_{2}\begin{psmallmatrix}\bm{V} \\\bm{E}\end{psmallmatrix}}{\hat{\fB}_{1}\begin{psmallmatrix} \bm{I} \\\bm{H}\end{psmallmatrix}}_{\C^{2k}}
    \\*
    &= \scprod*{\begin{psmallmatrix}\bm{V}(0) \\ -\bm{V}(1)\end{psmallmatrix}}{\begin{psmallmatrix} \bm{I}_{\textup{tot}}(0) \\ \bm{I}_{\textup{tot}}(1)\end{psmallmatrix}}_{\C^{2k}}.
  \end{align*}
\end{corollary}

Note that this abstract integration by parts formula gives rise to an abstract Green identity for the corresponding block operator
$\cl{\mathfrak{J}} =
\begin{bsmallmatrix}
  0 & \fA_{2} \\
  \cl{\fA_{1}} & 0
\end{bsmallmatrix}$.
In particular, we obtain a boundary triple for $\cl{\mathfrak{J}}$.

\begin{theorem}\label{th:abstract-green-identity}
 Suppose that the spatial domains are as in \Cref{sec:Omega}.
Further, let $\X$ and $\mathfrak{J}$ be defined as in \eqref{eq:Xspace} and \eqref{eq:Jdef}. Further, let $\hat{\fB}_1$, $\hat{\fB}_2$ be as in \Cref{def:bndmap}.

  For
  % \(
  % e_{1} =
  % \begin{bsmallmatrix}
  %   \bm{I}_{1} \\\bm{H}_{1} \\\bm{V}_{1} \\\bm{E}_{1}
  % \end{bsmallmatrix},
  % e_{2} =
  % \begin{bsmallmatrix}
  %   \bm{I}_{2} \\\bm{H}_{2} \\\bm{V}_{2} \\\bm{E}_{2}
  % \end{bsmallmatrix}
  % \in
  % \begin{bsmallmatrix}
  %   \Hspace^{1}((0,1);\C^k) \\
  %   \Hspace(\rot,\Omega) \\
  %   \Hspace^{1}((0,1);\C^k) \\
  %   \cH_{\Gamma_{\textup{r}}}(\rot,\Omega)
  % \end{bsmallmatrix}_{\times}
  % \)
  $e^{i} = (\bm{I}^{i},\bm{H}^{i},\bm{V}^{i},\bm{E}^{i})\in \dom \cl{\mathfrak{J}}$, $i=1,2$,
  we have
  \begin{align*}
    \scprod{\cl{\mathfrak{J}}e^{1}}{e^{2}}_{\mathcal{X}} + \scprod{e^{1}}{\cl{\mathfrak{J}}e^{2}}_{\mathcal{X}}
    &= \scprod*{\hat{\fB}_{1}\begin{psmallmatrix} \bm{I}^{1} \\\bm{H}^{1} \end{psmallmatrix}}{\hat{\fB}_{2}\begin{psmallmatrix}\bm{V}^{2} \\ \bm{E}^{2}\end{psmallmatrix}}_{\C^{2k}}
    +
    \scprod*{\hat{\fB}_{2}\begin{psmallmatrix}\bm{V}^{1} \\\bm{E}^{1}\end{psmallmatrix}}{\hat{\fB}_{1}\begin{psmallmatrix} \bm{I}^{2} \\ \bm{H}^{2}\end{psmallmatrix}}_{\C^{2k}} \\
    &= \scprod*{\begin{psmallmatrix} \bm{I}_{\textup{tot}}^{1}(0) \\ \bm{I}_{\textup{tot}}^{1}(1)\end{psmallmatrix}}{\begin{psmallmatrix}\bm{V}^{2}(0) \\ -\bm{V}^{2}(1)\end{psmallmatrix}}_{\C^{2k}}
    +
    \scprod*{\begin{psmallmatrix}\bm{V}^{1}(0) \\ -\bm{V}^{1}(1)\end{psmallmatrix}}{\begin{psmallmatrix} \bm{I}_{\textup{tot}}^{2}(0) \\ \bm{I}_{\textup{tot}}^{2}(1)\end{psmallmatrix}}_{\C^{2k}}.
  \end{align*}
\end{theorem}

\begin{proof}
\newcommand{\tmpIntByPartsI}{%
\scprod*{\cl{\fA_{1}} \begin{psmallmatrix} \bm{I}^{1} \\\bm{E}^{1}\end{psmallmatrix}}{\begin{psmallmatrix}\bm{V}^{2} \\\bm{H}^{2}\end{psmallmatrix}}_{\X_1}
+ \scprod*{\begin{psmallmatrix} \bm{I}^{1} \\\bm{E}^{1}\end{psmallmatrix}}{\fA_{2} \begin{psmallmatrix}\bm{V}^{2} \\\bm{H}^{2}\end{psmallmatrix}}_{\X_1}
}
\newcommand{\tmpIntByPartsII}{
\scprod*{\fA_{2} \begin{psmallmatrix}\bm{V}^{1} \\\bm{H}^{1}\end{psmallmatrix}}{\begin{psmallmatrix}\bm{I}^{2} \\\bm{E}^{2}\end{psmallmatrix}}_{\X_1}
+ \scprod*{\begin{psmallmatrix}\bm{V}^{1} \\\bm{H}^{1}\end{psmallmatrix}}{\cl{\fA_{1}}\begin{psmallmatrix}\bm{I}^{2} \\\bm{E}^{2}\end{psmallmatrix}}_{\X_1}
}
\newcommand{\tmpAdjustedWidthI}[1]{\makebox[\widthof{$\displaystyle\tmpIntByPartsI$}][c]{$#1$}}
\newcommand{\tmpAdjustedWidthII}[1]{\makebox[\widthof{$\displaystyle\tmpIntByPartsII$}][c]{$#1$}}
By applying \Cref{th:int-by-parts-extended} twice we obtain
\begingroup\addtolength{\jot}{0.2em}
  \begin{align*}
    %\MoveEqLeft
    \scprod{\cl{\mathfrak{J}}e^{1}}{e^{2}}_\X + \scprod{e^{1}}{\cl{\mathfrak{J}}e^{2}}_\X
    &=
      \scprod*{\begin{psmallmatrix} 0 & \fA_{2} \\ \cl{\fA_{1}} & 0\end{psmallmatrix}e^{1}}{e^{2}}_\X
      +
      \scprod*{e^{1}}{\begin{psmallmatrix} 0 & \fA_{2} \\ \cl{\fA_{1}} & 0\end{psmallmatrix}e^{2}}_\X \\
    &=
      \tmpIntByPartsI \\
      &\hspace{1.3cm}+ \tmpIntByPartsII
    \\
    &=
      {\scprod*{\hat{\fB}_{1}\begin{psmallmatrix}\bm{I}^{1} \\\bm{H}^{1}\end{psmallmatrix}}{\hat{\fB}_{2}\begin{psmallmatrix}\bm{V}^{2} \\ \bm{E}^{2}\end{psmallmatrix}}_{\C^{2k}}}
      +
      {\scprod*{\hat{\fB}_{2}\begin{psmallmatrix}\bm{V}^{1} \\\bm{E}^{2}\end{psmallmatrix}}{\hat{\fB}_{1} \begin{psmallmatrix} \bm{I}^{2} \\ \bm{H}^{2}\end{psmallmatrix}}_{\C^{2k}},}
  \end{align*}
\endgroup
  which proves the assertion.
\end{proof}

Corresponding to $\hat{\fB}_{1}$ and $\hat{\fB}_{2}$ we define the operators on $\dom \cl{\mathfrak{J}}$ by projections. In particular for
$e = (\bm{I},\bm{H},\bm{V},\bm{E})\in \dom \cl{\mathfrak{J}}$ we define
\begin{equation}\label{eq:Bdef}
\begin{aligned}
  \fB_{1} e &=
  \fB_{1} \begin{psmallmatrix} \bm{I} \\ \bm{H} \\ \bm{V} \\ \bm{E} \end{psmallmatrix}
  \coloneq \hat{\fB}_{1} \begin{psmallmatrix} \bm{I} \\ \bm{H} \end{psmallmatrix}
  = \begin{psmallmatrix} \bm{I}_{\textup{tot}}(0) \\ \bm{I}_{\textup{tot}}(1) \end{psmallmatrix},
  \\
  \fB_{2} e &=
  \fB_{2} \begin{psmallmatrix} \bm{I} \\ \bm{H} \\ \bm{V} \\ \bm{E} \end{psmallmatrix}
  \coloneq \hat{\fB}_{2} \begin{psmallmatrix}\bm{V} \\ \bm{E} \end{psmallmatrix}
  = \begin{psmallmatrix}\bm{V}(0) \\ -\bm{V}(1) \end{psmallmatrix}.
\end{aligned}
\end{equation}
This allows us to write the abstract Green identity from \Cref{th:abstract-green-identity} just as
\begin{align}\label{eq:abstract-green-identity}
  \scprod{\cl{\mathfrak{J}}e_{1}}{e_{2}}_{\X} + \scprod{e_{1}}{\cl{\mathfrak{J}}e_{2}}_{\X}
  = \scprod{\fB_{1} e_{1}}{\fB_{2}e_{2}}_{\C^{2k}} + \scprod{\fB_{2}e_{1}}{\fB_{1}e_{2}}_{\C^{2k}}
\end{align}

\begin{corollary}\label{cor:btriple}
 Suppose that the spatial domains are as in \Cref{sec:Omega}.
Further, let $\X$ and $\mathfrak{J}$ be defined as in \eqref{eq:Xspace} and \eqref{eq:Jdef}, and let $\fB_1$, $\fB_2$ be as in \eqref{eq:Bdef}.
Then
  $(\C^{2k},\fB_{1},\fB_{2})$ is a boundary triple (see \Cref{def:boundary-triple}) for $\cl{\mathfrak{J}}$, i.e., the following properties are fulfilled.
  \begin{enumerate}
      \item\label{cor:btriplea} $-\cl{\mathfrak{J}}\adjun \subseteq \cl{\mathfrak{J}}$,
      \item\label{cor:btripleb} The abstract Green identity \eqref{eq:abstract-green-identity} is satisfied for all $e_{1}, e_{2} \in \dom \cl{\mathfrak{J}}$ and
      \item\label{cor:btriplec} $\begin{bsmallmatrix}\fB_{1} \\ \fB_{2}\end{bsmallmatrix} \colon \dom \cl{\mathfrak{J}} \to \begin{bsmallmatrix}\C^{2k} \\ \C^{2k}\end{bsmallmatrix}_{\times}$ is surjective and bounded.
  \end{enumerate}
\end{corollary}

\begin{proof}
  We prove these three statements in the given order.
 \begin{enumerate}
   \item Note that
   $\mathfrak{J} = \begin{bsmallmatrix} 0 & \fA_{2} \\ \fA_{1} & 0\end{bsmallmatrix}$ and therefore
   $\cl{\mathfrak{J}} = \begin{bsmallmatrix} 0 & \fA_{2} \\ \cl{\fA_{1}} & 0\end{bsmallmatrix}$. Hence, by \eqref{eq:minimal-subset-maximal-operator} and \Cref{le:A1-as-adjoint-of-A2} follows
   \begin{align*}
     \cl{\mathfrak{J}}\adjun = \mathfrak{J}\adjun
     = \begin{bsmallmatrix} 0 & \fA_{2} \\ \fA_{1} & 0\end{bsmallmatrix}\adjun
     = \begin{bsmallmatrix} 0 & \fA_{1}\adjun \\ \fA_{2}\adjun & 0\end{bsmallmatrix}
     = \begin{bsmallmatrix} 0 & -\mathring{\fA}_{2} \\ -\cl{\mathring{\fA}_{1}} & 0\end{bsmallmatrix}
     \subseteq -\begin{bsmallmatrix} 0 & \fA_{2} \\
     \cl{\fA_{1}} & 0\end{bsmallmatrix} = -\cl{\mathfrak{J}}.
   \end{align*}

   \item The abstract Green identity follows from the definition of $\fB_{1}$ and $\fB_{2}$, and \Cref{th:abstract-green-identity}

   \item Surjectivity follows from \Cref{th:boundary-mappings-surjective}. Boundedness is a~consequence of \cite[Lem.~2.4.7]{Sk-Phd}.\qedhere

 \end{enumerate}
\end{proof}

In the following theorem we will parameterize boundary conditions for our coupled systems that are more or less the boundary conditions for the telegraphers equations in port-Hamiltonian modeling (in the sense of \cite{JaZw12}) alone. This means that the allegedly more difficult Maxwell system does not have a big impact for the parametrization of the boundary conditions, which is surprising.

\begin{theorem}\label{th:restrictions-of-J-max-dissipative}
Suppose that \Cref{ass:bndcont} is satisfied, and the spatial domains are as in \Cref{sec:Omega}.
Further, let $\X$ and $\mathfrak{J}$ be defined as in \eqref{eq:Xspace} and \eqref{eq:Jdef}, and let $\fB_1$, $\fB_2$ be as in \eqref{eq:Bdef}. Further, assume that $W_B\in\C^{2k\times 4k}$ has the property as in \Cref{ass:bndcont}.
%
%  Let $W_{1}, W_{2} \in \C^{2k \times 2k}$.
Then the operator defined by
\begin{subequations}\label{eq:fAdef}
      \begin{align}
    \fA e = \cl{\mathfrak{J}}e =
    \begin{bmatrix}
      0 & \cl{\fA_{2}} \\
      \fA_{1} & 0
    \end{bmatrix}
    e
  \end{align}
  with domain
  \begin{align}
    \dom \fA = \dset*{e \in \dom \cl{\mathfrak{J}}}
    {
    W_B%\begin{bmatrix}
      %W_{1} & W_{2}
    %\end{bmatrix}
    \begin{bmatrix}
      \fB_{1} e \\ \fB_{2} e
    \end{bmatrix}
    = 0
    }
  \end{align}
\end{subequations}
  is maximally dissipative.
%  \Nathanael{The assumptions of $W_{B}$ are commented. Is there are reason for that? Without assumptions this statement is not true.}
  %, if
  %and generates a contraction semigroup, if
  %\begin{enumerate}[label=\textup{(\roman{*})}]
  %  \item $\begin{bmatrix} W_{1} & W_{2} \end{bmatrix}$ has full rank and
   % \item $W_{1}W_{2}\adjun + W_{2}W_{1}\adjun \geq 0$.% or equivalently $\Re W_{1}W_{2}\adjun \geq 0$.
 % \end{enumerate}
\end{theorem}

\begin{proof}
  Partitioning $W_B = \begin{bmatrix} W_1 & W_2 \end{bmatrix}$ with $W_1,W_2\in\C^{2k\times 2k}$, we can equivalently write the boundary condition in $\dom \fA$ as
  \begin{equation*}
    \begin{bmatrix}
      \fB_{1} x \\ \fB_{2} x
    \end{bmatrix}
    \in
    \ker \begin{bmatrix} W_{1} & W_{2}\end{bmatrix}
  \end{equation*}
As $W_{1},W_{2}$ satisfy the conditions of \Cref{th:W1-W2-condition-implies-max-dissipativ} we have that $\ker  \begin{bmatrix} W_{1} & W_{2}\end{bmatrix}=\ker W_B$ is a~maximally dissipative linear relation.
%  By Lumer--Phillips theorem showing $\fA$ is maximally dissipative implies that $\fA$ generates a contraction semigroup.
Now using that, by \Cref{cor:btriple}, $(\C^{2k},\fB_{1},\fB_{2})$ is a boundary triple for $\cl{\mathfrak{J}}$, the result follows from \Cref{th:boundary-triple-maximal-dissipative}.
\end{proof}

The requirement for boundedness and positivity on $\bm{C}$, $\bm{L}$, $\bm{\epsilon}$ and $\bm{\mu}$, as stated in \Cref{ass:tl} and~\ref{ass:Maxwell}, implies that the operator $\hamiltonian\colon \mathcal{X}\to\mathcal{X}$, as defined in \eqref{eq:stateeffortH}, is strictly positive and bounded.
%
%\begin{align*}
%&c_1\id_k\leq C(\zeta)\leq c_2\id_k,\;c_1\id_k\leq L(\zeta)\leq c_2\id_k\quad\text{for almost all }\zeta\in[0,1],\\
%c_1\id_3\leq \bm{\epsilon}(\zeta)\leq c_2\id_3,\;c_1\id_3\leq \bm{\mu}(\zeta)\leq c_2\id_3\quad\text{for almost all }\\xi\in\Omega.
%\end{align*}
%In particular,
%\[\left(\begin{smallmatrix}\bf{B}\\\bf{D}\end{smallmatrix}\right)\mapsto \left( \right)^{1/2}\]
Consequently, we define
\begin{equation}\label{eq:Hnorm}
  \scprod*{\begin{psmallmatrix}\bm{\psi}_{1}\\ \bm{B}_{1}\\ \bm{q}_{1}\\ \bm{D}_{1}\end{psmallmatrix}}{\begin{psmallmatrix}\bm{\psi}_{2} \\ \bm{B}_{2} \\ \bm{q}_{2} \\ \bm{D}_{2} \end{psmallmatrix}}_{\hamiltonian}
  \coloneq
  \scprod*{\begin{psmallmatrix}\bm{\psi}_{1} \\ \bm{B}_{1} \\ \bm{q}_{1} \\ \bm{D}_{1} \end{psmallmatrix}}{\hamiltonian\begin{psmallmatrix}\bm{\psi}_{2}\\ \bm{B}_{2} \\ \bm{q}_{2} \\ \bm{D}_{2} \end{psmallmatrix}}_\X
  % \;\text{and}\;
  % \norm*{\begin{psmallmatrix}\bm{\psi}\\ \bm{B}\\ \bm{q}\\ \bm{D} \end{psmallmatrix}}_{\hamiltonian} \coloneq
  % \scprod*{\begin{psmallmatrix}\bm{\psi}\\ \bm{B}\\ \bm{q}\\ \bm{D}\end{psmallmatrix}}{\hamiltonian\begin{psmallmatrix}\bm{\psi}\\ \bm{B}\\ \bm{q}\\ \bm{D}\end{psmallmatrix}}^{1/2}
\end{equation}
which is an equivalent inner product on $\mathcal{X}$. We denote $\mathcal{X}$ endowed with $\scprod{\argdot}{\argdot}_{\hamiltonian}$ as $\XH$, i.e., the only difference between $\mathcal{X}$ and $\XH$ is the choice of the inner product, but the topology is the same. The corresponding norm $\norm{\argdot}_{\hamiltonian} = \sqrt{\scprod{\argdot}{\argdot}_{\hamiltonian}}$ is also called the energy norm.

%\begin{corollary}
%  the operator $(\cl{\mathfrak{J}} - \mathfrak{R})\hamiltonian$ with domain
%  \begin{equation*}
%    \dset*{x \in \hamiltonian^{-1} (\dom \cl{\mathfrak{J}})}{W_{1}\fB_{1}\hamiltonian x + W_{2}\fB_{2}\hamiltonian x = 0}
 % \end{equation*}
 % generates a contraction semigroup on $\XH$ (equipped with $\norm{\argdot}_{\hamiltonian}$).
%\end{corollary}

%\begin{proof}
 % Note that the operator equals $(\fA - \mathfrak{R})\hamiltonian$ on $\hamiltonian^{-1}(\dom \fA)$, where $\fA$ is from \Cref{th:restrictions-of-J-max-dissipative}. Since $\fA$ generates a contraction semigroup on $\mathcal{X}$, also $\fA\hamiltonian$ generates a contraction semigroup on $\XH$, see \cite[Lem.~7.2.3]{JaZw12}. By assumption $\mathfrak{R}\hamiltonian$ is a positive operator, therefore $\fA\hamiltonian - \mathfrak{R}\hamiltonian$ is still maximally dissipative and generates a contraction semigroup.
%\end{proof}

\begin{theorem}\label{thm:semigroup}
  Suppose that \Cref{ass:tl},~\ref{ass:bndcont} and~\ref{ass:Maxwell} are satisfied, and the spatial domains are as in \Cref{sec:Omega}.
  Further, let the operators and spaces $\mathfrak{J}$, $\mathfrak{R}$, $\hamiltonian$ and $\X$ be defined as in \eqref{eq:gensetup} and \eqref{eq:Jdef}, and let $\fB_1$, $\fB_2$ be as in \eqref{eq:Bdef}.
  % Assume that the matrix $W_{B}\in \mathbb{C}^{2k \times 4k}$ satisfies the conditions outlined in \Cref{ass:bndcont}.
  Then the operator $F$ defined by the restriction of $\cl{\mathfrak{J}}-\mathfrak{R}$ to
  % \begin{align*}
  %   A e &= \cl{\mathfrak{J}}e =
  %   \begin{bmatrix}
  %     0 & \cl{A_{2}} \\
  %     A_{1} & 0
  %   \end{bmatrix}
  %   e\\
  \begin{equation*}
    \dom F = \dset*{e \in \dom \cl{\mathfrak{J}}}
    {%
      W_B
      \begin{pmatrix}
        \fB_{1} e \\ \fB_{2} e
      \end{pmatrix}
      = 0
    }
  \end{equation*}
  is maximally dissipative.
  Further, the following holds:
  \begin{enumerate}[label=\textup{(\alph{*})}]
    \item\label{thm:semigroupa} $F\hamiltonian$ generates a~strongly continuous semigroup on $\mathcal{X}$.
          This semigroup is contractive with respect to the norm $\norm{\argdot}_{\hamiltonian}$ as in \eqref{eq:Hnorm}.

    \item\label{thm:semigroupb} If $\mathfrak{R}=0$ and
          \begin{equation}\label{eq:WB0}
            W_{B}
            \begin{bsmallmatrix}
              0 & \id_{{2k}} \\
              \id_{{2k}} & 0
            \end{bsmallmatrix}
            W_{B}\adjun = 0,
          \end{equation}
          then $F$ is skew-adjoint. That is, $F\adjun=-F$.

    \item\label{thm:semigroupc} If \eqref{eq:WB0} holds, then $F\hamiltonian$ generates a~strongly continuous group on $\mathcal{X}$. If, further, $\mathfrak{R}=0$, then this group is unitary with respect to the norm $\norm{\argdot}_{\hamiltonian}$ as in \eqref{eq:Hnorm}.
  \end{enumerate}
  %   if
  %   \begin{enumerate}[label=\textup{(\roman{*})}]
  %     \item $\begin{bmatrix} W_{1} & W_{2} \end{bmatrix}$ has full rank and
  %     \item $W_{1}W_{2}\adjun + W_{2}W_{1}\adjun \geq 0$ or equivalently $\Re W_{1}W_{2}\adjun \geq 0$.
  %   \end{enumerate}
\end{theorem}

\begin{proof}
  The operator $\fA$ as in \eqref{eq:fAdef}  maximally dissipative by \Cref{th:restrictions-of-J-max-dissipative}. By using the $\mathfrak{R}$ is bounded and dissipative, we can conclude that $F=\fA-\mathfrak{R}$ is maximally dissipative as well.
  \begin{enumerate}[label=\textup{(\alph{*})}]
    \item Maximal dissipativity of $F$ directly implies that $F\hamiltonian$ is maximally dissipative when $\mathcal{X}$ is equipped with the norm $\norm{\argdot}_{\hamiltonian}$. Therefore, the Lumer-Phillips theorem \cite[Prop.~3.8.4]{TuWe09} yields that $F\hamiltonian$ generates a contractive semigroup on $\mathcal{X}$ equipped with $\norm{\argdot}_{\hamiltonian}$. With the assistance of norm equivalence, this semigroup is also strongly continuous on $\mathcal{X}$ when equipped with the standard norm.

    \item Note that the condition \eqref{eq:WB0} on $W_{B} = \begin{bsmallmatrix}W_{1} & W_{2}\end{bsmallmatrix}$ implies that $\ker W_{B} = \ker \begin{bsmallmatrix}W_{1} & W_{2}\end{bsmallmatrix}$ is skew-adjoint. Hence, this property passes on to $F$.

    \item If $\mathfrak{R} = 0$ and \eqref{eq:WB0} are satisfied, we can use \ref{thm:semigroupb} to deduce that $F\hamiltonian$ is skew-adjoint with respect to the norm $\norm{\argdot}_{\hamiltonian}$.
    Then, by \cite[Thm.~3.8.6]{TuWe09},  $F\hamiltonian$ generates a unitary group on $\mathcal{X}$ equipped with $\norm{\argdot}_{\hamiltonian}$. Norm equivalence yields that this group exhibits strong continuity on $\mathcal{X}$ equipped with the standard norm.

    \noindent If $\mathfrak{R}$ is not necessarily the zero operator, we observe that $F\mathcal{H}$ represents a bounded perturbation of the generator of a $\conC_{0}$-group on $\mathcal{X}$. Utilizing \cite[Thm.~2.11.2]{TuWe09}, we can assert that $F\mathcal{H}$ itself generates a $\conC_{0}$-group.
    \qedhere
  \end{enumerate}
\end{proof}

\section{System node representation and well-posedness}\label{sec:sysnodes}

Here, we further explore the systems-theoretic analysis of the coupled field-cable system by examining its input-output characteristics. To accomplish this, we unitlize the system node framework established by {\sc Staffans} in \cite{St05}. Assume that $\mathcal{X}$, $\mathcal{U}$, and $\mathcal{Y}$ are Hilbert spaces. Our system is represented in the form
\begin{equation}
    \dot{x}(t)= A\&B\,\begin{psmallmatrix}x(t)\\u(t)\end{psmallmatrix},\quad
    y(t)= C\&D\,\begin{psmallmatrix}x(t)\\u(t)\end{psmallmatrix},
\label{eq:ODEnode}
\end{equation}
where $A\&B\colon \dom(A\&B)\subset \mathcal{X}\times \mathcal{U}\to \mathcal{X}$, $C\&D\colon \dom(C\&D)\subset \mathcal{X}\times \mathcal{U}\to\mathcal{\mathcal{Y}}$
are linear operators with specific properties detailed subsequently. Unlike the finite-dimensional case, the operators $A\&B$ and $C\&D$ do not separate into distinct components corresponding to the state and input. This is primarily motivated by the application of boundary control for partial differential equations.

\noindent The autonomous dynamics (i.e., those with trivial input $u=0$) are determined by the \emph{main operator}
$A\colon \dom(A)\subset X\to X$
with $\dom(A) \coloneq \dset{x\in X}{\spvek x0\in\dom(A\&B)}$ and $Ax \coloneq A\&B\spvek x0$ for all $x\in\dom(A)$.
%We first provide the modeling framework for the control systems that we are going to consider in this paper. For this, let $X$, $U$, and $\mathcal{\mathcal{Y}}$ be Hilbert spaces and consider a linear operator $S:X\times U\supset\dom S\to X \times \mathcal{\mathcal{Y}}$ describing the (abstract) dynamics
Next, we state the essential conditions on the operators $A\&B$ and $C\&D$ under which they constitute a system node.

\begin{definition}[System node]\label{def:sysnode}
  A \emph{system node} on the triple $(\mathcal{X},\mathcal{U},\mathcal{Y})$ of Hilbert spaces is a~linear operator $S =\sbvek{A\&B}{C\&D}$ with $A\&B\colon \dom(A\&B)\subset \mathcal{X}\times \mathcal{U}\to \mathcal{X}$, $C\&D\colon \dom(C\&D)\subset \mathcal{X}\times \mathcal{U}\to \mathcal{Y}$ satisfying the following conditions:
  %   $\left[\begin{smallmatrix}\mathcal{X}\\ U\end{smallmatrix}\right] \to \left[\begin{smallmatrix} \mathcal{X}\\\mathcal{Y}\end{smallmatrix}\right]$ that is decomposed into $S=\spvek{A\&B}{C\&D}$ with $A\&B = P_\mathcal{X} S:\dom S \to \mathcal{X}$ and $C\&D = P_\mathcal{Y} S:\dom S \to \mathcal{Y}$. Further, we set $Ax=A\&B\spvek{x}{0}$ and $\dom A=\{x\in \mathcal{X} : \spvek{x}{0}\in \dom S\}$ and demand the following conditions:
  \begin{enumerate}[label=(\roman{*})]
    \item $A\&B$ is closed.
    \item $C\&D\in \Lb(\dom (A\&B),\mathcal{Y})$.
    \item For all $u\in \mathcal{U}$, there exists some $x\in \mathcal{X}$ with $\spvek{x}{u}\in \dom(S)$.
    \item The main operator $A$ is the generator of a~strongly continuous semigroup $\fA(\argdot)\colon
          \R_{\ge 0}\to \Lb(\mathcal{X})$ on $\mathcal{X}$.
  \end{enumerate}
\end{definition}
Next, we define our solution concepts for \eqref{eq:ODEnode}.

\begin{definition}[Classical/generalized trajectories]\label{def:traj}
  Let $S = \sbvek{A\& B}{C\& D}$ be a~system node  on $(\mathcal{X},\mathcal{U},\mathcal{Y})$, and let $T\in\R_{>0}$.\\
  A \emph{classical trajectory} for \eqref{eq:ODEnode} on $[0,T]$ is a triple
  \[
    (x,u,y)\,\in\,\conC^{1}([0,T];\mathcal{X})\times \conC^{}([0,T];\mathcal{U})\times \conC^{}([0,T];\mathcal{Y})
  \]
  which for all $t\in[0,T]$ satisfies \eqref{eq:ODEnode}.\\
  A \emph{generalized trajectory} for \eqref{eq:ODEnode} on $[0,T]$ is a~limit
  % triple
  % \[
  %   (x,u,y)\,\in\,\conC^{}([0,T];\mathcal{X})\times \Lp{2}([0,T];\mathcal{U})\times \Lp{2}([0,T];\mathcal{Y}),
  % \]
  of classical trajectories for \eqref{eq:ODEnode} on $[0,T]$ in the topology of $\conC^{}([0,T];\mathcal{X})\times \Lp{2}([0,T];\mathcal{U})\times \Lp{2}([0,T];\mathcal{Y})$.\\
\end{definition}

It is shown in {\cite[Thm.~4.3.9]{St05}} that, if $x_0\in \mathcal{X}$ and $u\in W^{1,2}([0,T];\mathcal{U})$ with $\sbvek{x_0}{u(0)}\in \dom (A\&B))$, then there exist unique $x\in \conC^{}([0,T];\mathcal{X})$ with $x(0)=x_0$ and $y\in \Lp{2}([0,T];\mathcal{Y})$, such that $(x,u,y)$ is
a~classical trajectory for \eqref{eq:ODEnode}.
%existence of unique classical trajectories under appropriate assumptions on the control function and initial value.

%\begin{proposition}[Existence of classical trajectories ]\label{prop:solex}
%Let $S$ be a system node on $(\mathcal{X},\mathcal{U},\mathcal{Y})$, let $T\in\R_{>0}$, $x_0\in \mathcal{X}$ and $u\in W^{1,2}([0,T];\mathcal{U})$ with $\sbvek{x_0}{u(0)}\in \dom S$. Then  there exist unique
%classical trajectory $(x,u,y)$ for \eqref{eq:ODEnode} with %$x(0)=x_0$.
%\end{proposition}
\noindent Well-posed systems are those with the property that the output and state depend continuously on the initial state and input.
\begin{definition}[Well-posed systems]\label{def:wp}
  Let $S = \sbvek{A\& B}{C\& D}$ be a~system node  on $(\mathcal{X},\mathcal{U},\mathcal{Y})$. The system \eqref{eq:ODEnode} is called \emph{well-posed}, if for some (and hence all) $t>0$, there exists some $c_t>0$, such that the classical (and thus also the generalized) trajectories for \eqref{eq:ODEnode} on $[0,t]$ fulfill
  \begin{equation}\label{eq:wp}
    \norm{x(t)}_{\mathcal{X}} + \norm{y}_{\Lp{2}([0,t];\mathcal{Y})} \leq
    c_t\big(\norm{x(0)}_{\mathcal{X}}
    + \norm{u}_{\Lp{2}([0,t];\mathcal{U})}\big).
  \end{equation}
\end{definition}
%Note that well-posedness also comprises the properties of admissibility of the input and output operators %\cite[Secs.~4.2, 4.3]{TuWe09}.

%\begin{remark}[Well-posed systems]\label{rem:wp}
%Let $S = \sbvek{A\& B}{C\& D}$ be a~system node on $(\mathcal{X},\mathcal{U},\mathcal{Y})$ and $T>0$. Well-posedness of \eqref{eq:ODEnode} is equivalent to boundedness of the following mappings for all $T>0$
%\begin{alignat*}{4}
%\fB_T&\colon L^2([0,T];\mathcal{U}) &&\to \mathcal{X}, &&\quad \fC_T&\colon \mathcal{X}&\to L^2([0,T];\mathcal{Y}),\\
%\fD_T&\colon L^2([0,T];\mathcal{U}) &&\to L^2([0,T];\mathcal{Y}),
%\end{alignat*}
%where
%\begin{itemize}
%\item $\fB_T u=x(T)$, where $(x,u,y)$ is the generalized trajectory for \eqref{eq:ODEnode} on $[0,T]$ with $x(0)=0$,
%\item $\fC_T x_0=y$, where $(x,u,y)$ is the generalized trajectory for \eqref{eq:ODEnode} on $[0,T]$ with $u=0$ and $x(0)=x_0$,
%\item $\fD_T u=y$, where $(x,u,y)$ is the generalized trajectory for \eqref{eq:ODEnode} on $[0,T]$ with $x(0)=0$.
%\end{itemize}
%In particular, well-posedness implies that the integral in \eqref{eq:mildsol} is an element of $\mathcal{X}$.
%\end{remark}

Now we define the system node corresponding to the coupled field-cable system. To this end, under \Cref{ass:tl},~\ref{ass:bndcont} and~\ref{ass:Maxwell},  we consider the spatial domains as specified in \Cref{sec:Omega}. Further, let  $\mathfrak{J}$, $\mathfrak{R}$, $\hamiltonian$ and $\X$ as in~\eqref{eq:gensetup} and~\eqref{eq:Jdef}, and, additionally $\mathcal{U}=\C^m$, $\mathcal{Y}=\C^p$. Let $W_{C,{\rm out}}\in\C^{p\times 4k}$, and assume that
$W_{B,{\rm inp}} \in \C^{m\times 4k}$, $W_{B,0} \in \C^{(2k-m)\times 4k}$, 
fulfill \Cref{ass:bndcont}.
%\[\mathcal{X}=    \Lp{2}((0,1);\C^{m}) \times
%\Lp{2}(\Omega;\C^{3}) \times
%\Lp{2}((0,1);\C^{k}) \times
%\Lp{2}(\Omega;\C^{3}),\quad\mathcal{U}=\C^m,\quad\mathcal{Y}=\C^p.
%\]
The operators defining the system node are given by
\begin{subequations}\label{eq:FGKL1}
  \begin{equation}
    F\&G\begin{bsmallmatrix}\hamiltonian&0\\0&\id_{m}\end{bsmallmatrix},\quad K\&L\begin{bsmallmatrix}\hamiltonian&0\\0&\id_{m}\end{bsmallmatrix}
  \end{equation}
  with
  \begin{align}
    \dom(F\&G)
    &=
      \dset*{
      \begin{psmallmatrix}\bm{I}\\\bm{H}\\\bm{V}\\\bm{E}\\u\end{psmallmatrix}           \in\dom\cl{\mathfrak{J}}\times\C^m
      % \begin{bmatrix}
      %   \Hspace^{1}(0,1)^{k} \\
      %   \Hspace(\rot,\Omega) \\
      %   \Hspace^{1}(0,1)^{k} \\
      %   \Cc_{\Gamma_{\textup{r}}}(\Omega;\C^3)
      % \end{bmatrix}_{\times}
      }{\begin{pmatrix}u\\0\end{pmatrix} = \begin{bmatrix}W_{B,{\rm inp}}\\W_{B,0}\end{bmatrix}
      \begin{psmallmatrix}
        \phantom{-}\bm{V}(0) \\
        \phantom{-}\bm{I}(0) \\
        \phantom{-}\bm{V}(1) \\
        -\bm{I}(1)
      \end{psmallmatrix}
      },
    \\
    F\&G
    \begin{psmallmatrix}\bm{I}\\\bm{H}\\\bm{V}\\\bm{E}\\u\end{psmallmatrix}
    &= (\mathfrak{J}-\mathfrak{R})
      \begin{psmallmatrix}\bm{I}\\\bm{H}\\\bm{V}\\\bm{E}\end{psmallmatrix},
    \\
    K\&L\begin{psmallmatrix}\bm{I}\\\bm{H}\\\bm{V}\\\bm{E}\\u\end{psmallmatrix}
    &=W_{C,{\rm out}}
      \begin{psmallmatrix}
        \phantom{-}\bm{V}(0) \\
        \phantom{-}\bm{I}(0) \\
        \phantom{-}\bm{V}(1) \\
        -\bm{I}(1)
      \end{psmallmatrix},%\\
      % \hamiltonian \begin{psmallmatrix}\bm{\psi}\\\bm{B}\\\bm{q}\\\bm{D}\end{psmallmatrix}&=\begin{psmallmatrix}\bm{L}^{-1}\bm{\psi}\\\bm{\mu}^{-1}\bm{B}\\\bm{C}^{-1}\bm{q}\\\bm{\epsilon}^{-1}\bm{D}\end{psmallmatrix}
  \end{align}
\end{subequations}

\begin{theorem}\label{thm:sysnodecoup}
  Suppose that \Cref{ass:tl},~\ref{ass:bndcont} and~\ref{ass:Maxwell} are satisfied, and the spatial domain is as depicted in \Cref{sec:Omega}. Then for $\mathcal{X}$ as in \eqref{eq:Xspace}, and $F\&G$, $K\&L$ as in \eqref{eq:FGKL1}, and $\hamiltonian$ as in \eqref{eq:stateeffortH}, the operator
  \begin{equation}
    S \coloneq
    \begin{bsmallmatrix}
      F\&G\\[-1mm]\\K\&L\end{bsmallmatrix}
    \begin{bsmallmatrix}\hamiltonian &0\\0&\id_{m}\end{bsmallmatrix}
    \label{eq:cablesysnode}
  \end{equation}
  is a~system node on $(\mathcal{X},\C^m,\C^p)$.
  Further, if the output is co-located, then 
  \begin{equation}
    M=\begin{bsmallmatrix}\phantom{-}F\&G\\[-1mm]\\-K\&L\end{bsmallmatrix}\label{eq:cabledissnode}
  \end{equation}
  with domain $\dom(M)=\dom(F\&G)$ is a~maximally dissipative operator. In this case,
  all generalized trajectories of
  the system corresponding to the system node $S$ on $[0,T]$ fulfill, for $\fB_1$, $\fB_{2}$ as in \eqref{eq:Bdef}, for all $t \in [0,T]$
  \begingroup
  \thinmuskip=2mu%
  \medmuskip=2mu%
  \thickmuskip=2mu plus 2mu%
  \begin{align}\label{eq:energybalance}
    \begin{aligned}
      \MoveEqLeft[0.5]
      \tfrac{1}{2} \scprod[\big]{x(t)}{\hamiltonian x(t)}_{\mathcal{X}} - \tfrac{1}{2} \scprod[\big]{x(0)}{\hamiltonian x(0)}_{\mathcal{X}}
      \\
      &= \int_0^t \Re\scprod{u(\tau)}{y(\tau)}_{\C^m} \dx[\tau]
        + \int_0^t \Re \scprod{\hamiltonian x(\tau)}{\mathfrak{R}\hamiltonian x(\tau)}_{\mathcal{X}} \dx[\tau]
      \\
      &\phantom{=}\mathclose{} +
        \begin{aligned}[t]
          \MoveEqLeft[3.5]
          \tfrac{1}{2}
          \int_0^t \biggl\langle
          \begin{bsmallmatrix}\fB_{1}\\\fB_{2}\end{bsmallmatrix}\hamiltonian x(\tau),
          %\\ &
               \left(
               \begin{bsmallmatrix} 0 & \id_{2k} \\ \id_{2k} & 0\end{bsmallmatrix} -\begin{bsmallmatrix}W_{B}\\W_C\end{bsmallmatrix}\adjun
               \begin{bsmallmatrix}0&\id_{2k}\\\id_{2k}&0\end{bsmallmatrix}
               \begin{bsmallmatrix}W_{B}\\W_C\end{bsmallmatrix}
               \right)
               \begin{bsmallmatrix}\fB_{1} \\ \fB_{2}\end{bsmallmatrix} \hamiltonian x(\tau) \biggr\rangle_{\mspace{-5mu}\C^{4k}}
               \mspace{-5mu}\dx[\tau].
        \end{aligned}
    \end{aligned}
  \end{align}
  \endgroup
\end{theorem}

\begin{proof}
To show that $S$ is a~system node on
$(\mathcal{X},\C^m,\C^p)$, we successively prove the four properties in \Cref{def:sysnode} in the order (iv), (iii), (i), (ii).\\
The semigroup property of $F\hamiltonian$
(i.e., property (iv)) is established by
\Cref{thm:semigroup}, as the main
operator of $S$ is precisely the one
addressed in that theorem.\\
Property (iii) is a~direct consequence of
\Cref{cor:btriple}\,\ref{cor:btriplec}.\\
To prove (i), let us assume that $(x_n,u_n)_{n\in\N}$ is a sequence in $\dom((F\hamiltonian)\&G)$ that converges to $(x,u)\in\mathcal{X}\times\C^m$ (in the $\mathcal{X}\times\C^m$ topology). Additionally, suppose that $(F\&G(\hamiltonian x_n,u_n))_{n\in\N}$ converges in $\mathcal{X}$ to $z$.

With the definition of $F\&G$ and the boundedness of $\hamiltonian$ and $\mathfrak{R}$, we can conclude that $(\hamiltonian x_n)$ converges to
$\hamiltonian x$, and $(\cl{\mathfrak{J}}x_n)$ converges in $\mathcal{X}$ to $z+\mathfrak{R}\hamiltonian x$. Now, since $\cl{\mathfrak{J}}$ is
closed, we can assert that $\hamiltonian x$ belongs to $\dom\cl{\mathfrak{J}}$, with $\mathfrak{J}\hamiltonian x= z+\mathfrak{R}\hamiltonian x$. This implies that $z=(\mathfrak{J}-\mathfrak{R})\hamiltonian x$. In particular, $(\hamiltonian x_n)$
converges in the topology of $\dom \mathfrak{J}$ to $x$. Partitioning
\begin{equation}
  \hamiltonian x_n=\begin{psmallmatrix}\bm{I}_n\\\bm{H}_n\\\bm{V}_n\\\bm{E}_n\end{psmallmatrix},\;\;\hamiltonian x=\begin{psmallmatrix}\bm{I}\\\bm{H}\\\bm{V}\\\bm{E}\end{psmallmatrix},\label{eq:Hx}
\end{equation}
we can conclude from \Cref{cor:btriple} that
\[
  \begin{psmallmatrix}u_n\\0\end{psmallmatrix}=\begin{bsmallmatrix}W_{B,{\rm inp}}\\W_{B,0}\end{bsmallmatrix}
  \begin{psmallmatrix}
    \phantom{-}\bm{V}_n(0) \\
    \phantom{-}\bm{I}_n(0) \\
    \phantom{-}\bm{V}_n(1) \\
    -\bm{I}_n(1)
  \end{psmallmatrix}\underset{n\to\infty}{\to} \begin{bsmallmatrix}W_{B,{\rm inp}}\\W_{B,0}\end{bsmallmatrix}
  \begin{psmallmatrix}
    \phantom{-}\bm{V}(0) \\
    \phantom{-}\bm{I}(0) \\
    \phantom{-}\bm{V}(1) \\
    -\bm{I}(1)
  \end{psmallmatrix}=\begin{psmallmatrix}u\\0\end{psmallmatrix}.
\]
The definition of $F\&G$ now yields that $(\hamiltonian x,u)\in\dom(F\&G)$ with
\[
  F\&G\begin{psmallmatrix}\hamiltonian x\\u\end{psmallmatrix}=(\mathfrak{J}-\mathfrak{R})\hamiltonian x=z.
\]
Hence we have shown that $F\&G$ is closed.\\
Property (ii) holds as, according to \Cref{cor:btriple},\ref{cor:btriplec}, the evaluation at the boundary of the transmission line represents a bounded operator on $\dom(\cl{\mathfrak{J}})$.\\
To complete the proof of the result, we now assume that the output is co-located, i.e.,
\eqref{eq:col1} holds for some $W_C\in\C^{2k\times 4k}$ with \eqref{eq:col2}.
Let $\begin{psmallmatrix}e\\u\end{psmallmatrix}\in\dom(F\&G)$, and partition $e$ as $\hamiltonian x$ in \eqref{eq:Hx}. Then
\[
\Re\left\langle \fB_{1}e,\fB_{2}e \right\rangle_{\C^{2k}}\\
=\tfrac12\left\langle
\begin{bsmallmatrix}\fB_{1}\\\fB_{2}\end{bsmallmatrix}e,\begin{bsmallmatrix}0&\id_{2k}\\\id_{2k}&0\end{bsmallmatrix}
\begin{bsmallmatrix}\fB_{1}\\\fB_{2}\end{bsmallmatrix}e \right\rangle_{\C^{4k}}
\]
and
\begin{align*}
\Re\left\langle u,W_{C,{\rm out}}
\begin{bsmallmatrix}\fB_{1}\\\fB_{2}\end{bsmallmatrix}e \right\rangle_{\C^m}
&=\Re\left\langle \begin{psmallmatrix}u\\0\end{psmallmatrix},W_{C}
\begin{bsmallmatrix}\fB_{1}\\\fB_{2}\end{bsmallmatrix} e\right\rangle_{\C^{2k}}\\
&=\Re\left\langle W_{B}
\begin{bsmallmatrix}\fB_{1}\\\fB_{2}\end{bsmallmatrix}e,W_{C}
\begin{bsmallmatrix}\fB_{1}\\\fB_{2}\end{bsmallmatrix}e \right\rangle_{\C^{2k}}\\
&=\tfrac12\left\langle \begin{bsmallmatrix}W_{B}\\W_C\end{bsmallmatrix}
\begin{bsmallmatrix}\fB_{1}\\\fB_{2}\end{bsmallmatrix}e,\begin{bsmallmatrix}0&\id_{2k}\\\id_{2k}&0\end{bsmallmatrix}\begin{bsmallmatrix}W_{B}\\W_C\end{bsmallmatrix}
\begin{bsmallmatrix}\fB_{1}\\\fB_{2}\end{bsmallmatrix} e\right\rangle_{\C^{4k}}\\
&=\tfrac12\left\langle
\begin{bsmallmatrix}\fB_{1}\\\fB_{2}\end{bsmallmatrix}e,\begin{bsmallmatrix}W_{B}\\W_C\end{bsmallmatrix}\adjun\begin{bsmallmatrix}0&\id_{2k}\\\id_{2k}&0\end{bsmallmatrix}\begin{bsmallmatrix}W_{B}\\W_C\end{bsmallmatrix}
\begin{bsmallmatrix}\fB_{1}\\\fB_{2}\end{bsmallmatrix}e \right\rangle_{\C^{4k}}.
\end{align*}
Using these two identities, we obtain that
\begin{align*}
\MoveEqLeft
\Re \scprod*{\begin{psmallmatrix} e \\ u \end{psmallmatrix}}
{\begin{bsmallmatrix}\phantom{-} F\&G \\[-1mm] \\ -K\&L \end{bsmallmatrix} \begin{psmallmatrix} e \\ u \end{psmallmatrix}}_{\mathcal{X}\times\C^{m}}
\\
&=\Re\left\langle e,F\&G\begin{psmallmatrix}e\\u\end{psmallmatrix} \right\rangle_{\mathcal{X}}-
\Re\left\langle u,K\&L\begin{psmallmatrix}e\\u\end{psmallmatrix} \right\rangle_{\C^m}
\\
&=\Re\left\langle e,\mathfrak{J}e \right\rangle_{\mathcal{X}}-\Re\left\langle e,\mathfrak{R}e \right\rangle_{\mathcal{X}}-
\Re\left\langle u,W_{C,{\rm out}}
\begin{bsmallmatrix}\fB_{1}\\\fB_{2}\end{bsmallmatrix} \right\rangle_{\C^m}\\
&\stackrel{\mathclap{\eqref{eq:abstract-green-identity}}}{=}
\Re\scprod{\fB_{1} e}{\fB_{2}e}_{\C^{2m}} - \Re\left\langle e,\mathfrak{R}e \right\rangle_{\mathcal{X}}-
\Re\left\langle u,W_{C,{\rm out}}
\begin{bsmallmatrix}\fB_{1}\\ \fB_{2}\end{bsmallmatrix} \right\rangle_{\C^m}\\
&=-\Re\left\langle e,\mathfrak{R}e \right\rangle_{\mathcal{X}}+\Re\scprod{\fB_{1} e}{\fB_{2}e}_{\C^{2m}}-
\Re\left\langle u,W_{C,{\rm out}}
\begin{bsmallmatrix}\fB_{1}\\\fB_{2}\end{bsmallmatrix}e \right\rangle_{\C^m}
\\
&=-\Re \left\langle e,\mathfrak{R}e \right\rangle_{\mathcal{X}}\\&\qquad+\tfrac12
\scprod*{\begin{bsmallmatrix} \fB_{1} \\ \fB_{2} \end{bsmallmatrix} e} {\smash[b]{\underbrace{\left(\begin{bsmallmatrix} 0 & \id_{2k} \\ \id_{2k} & 0\end{bsmallmatrix} - \begin{bsmallmatrix} W_{B} \\ W_C \end{bsmallmatrix}\adjun \begin{bsmallmatrix} 0 & \id_{2k} \\ \id_{2k} & 0 \end{bsmallmatrix} \begin{bsmallmatrix} W_{B} \\ W_C \end{bsmallmatrix} \right)}_{\leq \mathrlap{0}}
\begin{bsmallmatrix}\fB_{1} \\ \fB_{2} \end{bsmallmatrix}e}}_{\C^{4k}}
%\vphantom{\underbrace{\begin{bsmallmatrix} B_{1} \\ B_{2} \end{bsmallmatrix}}_{\leq 0}}
\leq 0.\\[-2mm]
\end{align*}
This shows that $M$ is dissipative. Maximal dissipativity of $M$ now follows, by incorporating that $S$ is a~system node, from \cite[Prop.~3.8]{PhReSc23}.
Furthermore, by combining the previously derived expression for
$\scprod*{\begin{psmallmatrix}e\\u\end{psmallmatrix}}{\begin{bsmallmatrix}\phantom{-}F\&G\\[-1mm]\\-K\&L\end{bsmallmatrix}\begin{psmallmatrix}e\\u\end{psmallmatrix}}_{\X\times\C^{m}}$ with \cite[Thm.~3.11]{PhReSc23}, we can establish that \eqref{eq:energybalance} holds.
\end{proof}

\begin{remark}
%\
%  \begin{enumerate}[label=(\alph{*})]
 %   \item 
 Systems defined by nodes of the form \eqref{eq:cablesysnode}, where $\hamiltonian$ is self-adjoint and positive and $M$ in \eqref{eq:cabledissnode} is dissipative, are called \emph{port-Hamiltonian} in \cite{PhReSc23}.
The framework there is more general, as $\hamiltonian$ need not be bounded or boundedly invertible.
\end{remark}

As a last result, we present a~sufficient criterion for well-posedness.

\begin{theorem}\label{thm:sysnodecoupwp}
Suppose that \Cref{ass:tl}, \ref{ass:bndcont}, and \ref{ass:Maxwell} hold, and that the spatial domains are as in \Cref{sec:Omega}. Further assume the strict inequality  \[
    W_B
    \begin{bmatrix}
      0 & \id_{{2k}} \\
      \id_{{2k}} & 0
    \end{bmatrix}
    W_B\adjun > 0,
  \]
and let $W_{C,{\rm out}}\in\C^{p\times 4k}$,  $p\in\N$.
  Then for $\X$ as in \eqref{eq:Xspace}, and $F\&G$, $K\&L$ as in \eqref{eq:FGKL1}, and $\hamiltonian$ as in \eqref{eq:stateeffortH},
  the system corresponding to the node \eqref{eq:cablesysnode} 
  % \begin{equation}
  %   \label{eq:origsys}
  %   \begin{psmallmatrix}\dot{x}(t) \\ y(t)\end{psmallmatrix}
  %   = \begin{bsmallmatrix}F\&G\\[-1mm] \\ K\&L\end{bsmallmatrix}
  %   \begin{bsmallmatrix}\hamiltonian&0 \\ 0&\id_m\end{bsmallmatrix}
  %   \begin{psmallmatrix}{x}(t) \\ u(t)\end{psmallmatrix}
  % \end{equation}
  is well-posed.
\end{theorem}

\begin{proof}
  It is not a loss of generality to assume that $m=2k$, i.e., $W_{B,{\rm out}}=W_B$. In other words, all boundary conditions are actually ones in which the input acts. The case with the presence of homogeneous boundary conditions can be addressed by setting the corresponding inputs to zero.\\
  We partition $W_B=[W_{B1},\,W_{B2}]$ with $W_{B1},W_{B2}\in\C^{2k\times 2k}$.
  Then \eqref{eq:WBdef} means that
  \begin{equation}W_{B1}W_{B2}\adjun+W_{B2}W_{B1}\adjun>0.\label{eq:WBpos}\end{equation}
  Since for any $x\in\ker W_{B2}^*$, it holds that
  \[x^*(W_{B1}W_{B2}^*+W_{B2}W_{B1}^*)x=0,\]
  \eqref{eq:WBpos} implies that $x=0$. This shows that $W_{B2}$ is invertible. Now we define
  \[\widetilde{W_C}=[W_{B2}^{-*},\, 0_{2k\times 2k}],\]
  and consider the system node
  \[
    \widetilde{S} \coloneq
    \begin{bsmallmatrix}F\&G\\\widetilde{K\&L}\end{bsmallmatrix}
    \begin{bsmallmatrix}\hamiltonian &0\\0&\id_{m}\end{bsmallmatrix}
  \]
  with
  \[
    \widetilde{K\&L}\begin{psmallmatrix}\bm{I}\\\bm{H}\\\bm{V}\\\bm{E}\\u\end{psmallmatrix}=\widetilde{W_{C}}  \begin{psmallmatrix}
    \phantom{-}\bm{V}(0) \\
    \phantom{-}\bm{I}(0) \\
    \phantom{-}\bm{V}(1) \\
    -\bm{I}(1)
    \end{psmallmatrix}.
  \]
  Let $\fB_1$ be as in \eqref{eq:Bdef}. Define the system
  \begin{equation}
    \label{eq:auxsys}
    \begin{psmallmatrix}\dot{x}(t)\\\widetilde{y}(t)\end{psmallmatrix}
    =\begin{bsmallmatrix}F\&G\\\widetilde{K\&L}\end{bsmallmatrix}
    \begin{bsmallmatrix}\hamiltonian&0\\0&\id_m\end{bsmallmatrix}
    \begin{psmallmatrix}{x}(t)\\u(t)\end{psmallmatrix},
  \end{equation}
  with output equation
  \[
    \widetilde{y}(t) \coloneq W_{B2}^{-*}\fB_{1}\hamiltonian x(t).
  \]
  We have
  \begingroup
  \thinmuskip=1mu%
  \medmuskip=1mu%
  \thickmuskip=2mu plus 0mu%
  \begin{align*}
    \MoveEqLeft[0]
    \begin{bmatrix} W_B \\ \widetilde{W_C} \end{bmatrix}
    \begin{bmatrix} 0 & \id_{2k} \\ \id_{2k} & 0\end{bmatrix}
    \begin{bmatrix} W_B\adjun & \widetilde{W_C}\adjun\end{bmatrix}
    = \begin{bmatrix} W_{B1} & W_{B2} \\ W_{B2}^{-\adjunsymb} & 0\end{bmatrix}
    \begin{bmatrix}0 & \id_{2k} \\ \id_{2k} & 0\end{bmatrix}
    \begin{bmatrix} W_{B1}\adjun & W_{B2}^{-1} \\ {W}_{B2}\adjun & 0\end{bmatrix}
    \\
    &= \begin{bmatrix}W_{B1} & W_{B2} \\ W_{B2}^{-\adjunsymb} & 0\end{bmatrix}
      \begin{bmatrix} 0 & \id_{2k} \\ \id_{2k} & 0\end{bmatrix}
      \begin{bmatrix}{W}_{B2}\adjun & 0 \\ W_{B1}\adjun & W_{B2}^{-1}\end{bmatrix}
      =
      \begin{bmatrix}
        W_{B1}W_{B2}\adjun + W_{B2}W_{B1}\adjun & \id_{2k} \\
        \id_{2k} & 0
      \end{bmatrix}.
  \end{align*}
  \endgroup
  Let
  \[
    W \coloneq W_{B2}^{-*}W_{B1}\adjun + W_{B1}\adjun W_{B2}^{-1}
    = W_{B2}^{-1}(W_{B1}W_{B2}\adjun + W_{B2}W_{B1}\adjun)W_{B2}^{-*} > 0,
  \]
  and define $\delta>0$ as be the minimal eigenvalue of $W$.
  Then \Cref{thm:sysnodecoup} gives rise to the energy balance
  \begin{align*}
  &  \tfrac{1}{2} \scprod{x(t)}{\hamiltonian x(t)}_{\X}
    - \tfrac{1}{2} \scprod{x(0)}{\hamiltonian x(0)}_{\X}
    \\
    &=
      \int_{0}^{t} \!\! \Re\scprod{u(\tau)}{\widetilde{y}(\tau)}_{\C^m} \dx[\tau]
      + \int_{0}^{t} \!\!\Re
      \scprod*{\hamiltonian x(\tau)}{\mathfrak{R}\hamiltonian x(\tau)}_{\X}\dx[\tau]
 - \frac{1}{2} \int_{0}^{t}\!\!
      \scprod{\widetilde{y}(\tau)}{W\widetilde{y}(\tau)}_{\C^{2k}} \dx[\tau]
    \\
    &\leq \int_{0}^{t} \Re\scprod{u(\tau)}{\widetilde{y}(\tau)}_{\C^m}
      - \delta \norm{\widetilde{y}(\tau)}^2_{\C^m} \dx[\tau]
  \end{align*}
  Now using that
  \[
    \Re\scprod{u(\tau)}{\widetilde{y}(\tau)}_{\C^m}
    \leq \tfrac{1}{2\delta} \norm{u(\tau)}^2_{\C^m}
    + \tfrac{\delta}{2} \norm{\widetilde{y}(\tau)}^2_{\C^m},
  \]
  we obtain
  \[
    \tfrac{1}{2} \scprod{x(t)}{\hamiltonian x(t)}_{\X}
    - \tfrac{1}{2} \scprod{x(0)}{\hamiltonian x(0)}_{\X}
    \leq
    \tfrac{1}{2\delta} \norm{u}^{2}_{\Lp{2}((0,t);\C^k)}
    -\tfrac{\delta}{2} \norm{\widetilde{y}}^{2}_{\Lp{2}((0,t);\C^k)},
  \]
  and thus
  \[
    \scprod{x(t)}{\hamiltonian x(t)}_{\X}
    + \delta \norm{\widetilde{y}}^{2}_{\Lp{2}((0,t);\C^k)}
    \leq \scprod{x(0)}{\hamiltonian x(0)}_{\X}
    + \tfrac{1}{\delta} \norm{u}_{\Lp{2}((0,t);\C^k)}.
  \]
  Now using the equivalence of the standard norm in $\X$ and the one in \eqref{eq:Hnorm}, we can immediately conclude that \eqref{eq:auxsys} is a~well-posed system. That is, there exists some $c>0$, such that
  \[
    \norm{x(t)}_{\X} + \norm{\widetilde{y}}_{\Lp{2}((0,t);\C^{k})}
    \leq c\big(\norm{x(0)}_{\X} + \norm{u}_{\Lp{2}((0,t);\C^{k})}\big).
  \]
  Now we show well-posedness of the actual system. It follows directly from construction that $\begin{bsmallmatrix}W_B \\ \widetilde{W_C}\end{bsmallmatrix}$
  is invertible. Then we have, for almost all $t\in[0,T]$,
  \begin{multline*}
    y(t) = W_C
    \begin{bsmallmatrix}\fB_{1} \\ \fB_{2} \end{bsmallmatrix} \hamiltonian x(\tau)
    = W_C \begin{bsmallmatrix} W_B \\ \widetilde{W_C} \end{bsmallmatrix}^{-1}
    \begin{bsmallmatrix} W_B \\ \widetilde{W_C} \end{bsmallmatrix}\hamiltonian x(\tau)
    = W_C \begin{bsmallmatrix} W_B \\ \widetilde{W_C} \end{bsmallmatrix}^{-1}
    \begin{psmallmatrix} u(t) \\ \widetilde{y}(t) \end{psmallmatrix}.
  \end{multline*}
  Now, by choosing
  \[
    \gamma \coloneq
    \norm*{W_C\begin{bsmallmatrix} W_B \\ \widetilde{W_C} \end{bsmallmatrix}^{-1}},
  \]
  we obtain
  \begin{multline*}
    \norm{x(t)}_{\X} + \norm{y}_{\Lp{2}((0,t);\C^{k})}
    \leq \norm{x(t)}_{\X} + \gamma\norm{y}_{\Lp{2}((0,t);\C^{k})}
    + \gamma\norm{u}_{\Lp{2}((0,t);\C^k)}
    \\
    \leq c\max\{1,\gamma\}\big( \norm{x(0)}_{\X} + \norm{u}_{\Lp{2}((0,t);\C^{k})} \big)
    + \gamma \norm{u}_{\Lp{2}((0,t);\C^k)}.
  \end{multline*}
  Consequently, the generalized trajectories fulfill \eqref{eq:wp} for
  \[
    c_{t} =\max\{1,c\} \max\{1,\gamma\} (1+\gamma),
  \]
  which completes the proof.
\end{proof}

\section{Conclusion}\label{sec:conclusion}

We have presented an analysis for radiating cable harnesses, resulting in telegrapher's and Maxwell's equations that are coupled through boundary conditions. An operator and systems theoretic analysis of the entire coupled system has been conducted using the theories of boundary triplesand system nodes. It has been shown that the autonomous part is described by a strongly continuous semigroup, and a sufficient criterion for well-posedness has been provided.

%Surprisingly, the conditions for a boundary control systems are exactly the same as there are no Maxwell's equations involved at all. This is in particular astounding, because normally port-Hamiltonian system on a spatially multi-dimensional domain require stronger conditions. In this case Maxwell's equations do not have open ports, therefore they are ``hidden'' inside the dynamic of the system.

%Conclusions may be used to restate your hypothesis or research question, restate your major findings, explain the relevance and the added value of your work,
%highlight any limitations of your study, describe future directions for research and recommendations.

%In some disciplines use of Discussion or `Conclusion' is interchangeable. It is not mandatory to use both. Please refer to Journal-level guidance for any specific requirements.

%\bmhead{Supplementary information}

%If your article has accompanying supplementary file/s please state so here.

%Authors reporting data from electrophoretic gels and blots should supply the full unprocessed scans for key as part of their Supplementary information. This may be requested by the editorial team/s if it is missing.

%Please refer to Journal-level guidance for any specific requirements.

%\noindent\textbf{Acknowledgments.}

\par\smallskip\noindent\textbf{Acknowledgments.}%

This work was supported by the collaborative research center SFB 1701 ``Port-Hamiltonian Systems'', funded by the German Research Foundation (DFG), project number 531152215.

%\section*{Declarations}

%\begin{itemize}
%\item Funding: not applicable
%\item Conflict of interest/Competing interests: not applicable
%\item\bm{E}thics approval: still not applicable
%\item Consent to participate: yet again, not applicable.
%\item Consent for publication: n.a.
%\item Availability of data and materials: not applicable
%\item Code availability: There's no code.
%\item Authors' contributions: We follow the good tradition in mathematical research, emphasizing content and interaction, of arranging the participating authors alphabetically and assuming that all authors are equally involved and responsible.
%\end{itemize}

%\noindent

%%===================================================%%
%% For presentation purpose, we have included        %%
%% \bigskip command. please ignore this.             %%
%%===================================================%%

\appendix

\section{Boundary triples}

We review the most important properties of boundary triples for skew-symmetric operators for this work.
More details can be found in~\cite[Chap.~3]{GoGo91} and \cite{BeHaSn20}, where boundary triples are regarded for symmetric operators. For skew-symmetric operators we refer to \cite{Sk-Phd} (we will use the skew-symmetric version).

A \emph{linear relation} $T$ from a~space $\X$ to a~space $\Y$ is a subspace
of $\X\times \Y$. Clearly, by an identification with its graph, any linear operator is also a linear relation. In this context, linear relations can be regarded as multi-valued linear operators.
% We will use the following notation
% \begin{align*}
%   \ker T &\coloneq \dset{x \in X}{(x,0) \in T},  &\ran T &\coloneq \dset{y \in Y}{\exists x :(x,y)\in T}, \\
%   \mul T &\coloneq \dset{y \in Y}{(0,y) \in T},  &\dom T &\coloneq \dset{x \in X}{\exists y :(x,y) \in T}.
% \end{align*}
% Thus, $T$ is single-valued, if $\mul T = \sset{0}$.
% The closure $\cl{T}$ of a linear relation $T$ is the closure in $X\times Y$.
% Note that every linear relation is closable. Also every operator has a
% closure as a linear relation, but its closure can be multi-valued. Therefore,
% showing $\mul \cl{T} = \sset{0}$ is necessary, even if $\mul T = \sset{0}$.
% For an additional linear relation $S$ from $Y$ to another vector space $Z$ we define the composition $ST$ as
% \begin{align*}
%   ST \coloneq \dset{(x,z)\in X\times Z}{\exists y \in Y \ \text{such that}\ (x,y) \in T \ \text{and}\ (y,z)\in S}.
% \end{align*}
% For a linear relation $T$ from a Hilbert space $X$ to a Hilbert space $Y$ the adjoint relation is defined by
% \[
%   T\adjun \coloneq \dset{(u,v)\in Y \times X}{\scprod{u}{y}_{Y} = \scprod{v}{x}_{X} \;\text{for all}\; (x,y) \in T}
% \]
% and the following holds true
% \[
%   \ker T\adjun = (\ran T)^{\perp}
%   ,\quad
%   \mul T\adjun = (\dom T)^{\perp}
%   \quad\text{and}\quad
%   T\adjun =
%   \left[\begin{smallmatrix}
%       0 & \opid_{Y} \\
%       -\opid_{X} & 0
%     \end{smallmatrix}\right] T^{\perp}
%   ,
% \]
% where
% $\begin{bsmallmatrix} 0 & \opid_{Y} \\ -\opid_{X} & 0\end{bsmallmatrix} T \coloneq \dset{(y,-x)}{(x,y)\in T}$
% and $T^{\perp}$ is the orthogonal complement in $X\times Y$.

\begin{definition}
  A linear relation $T$ on a Hilbert space $\X$ (from $\X$ to $\X$)
  \begin{itemize}
    \item is \emph{dissipative}, if $\Re \scprod{x}{y}_{\X} \leq 0$ for every $\begin{psmallmatrix} x \\ y\end{psmallmatrix} \in T$ and
    \item \emph{maximally dissipative}, if additionally it has no proper dissipative extension.
  \end{itemize}
\end{definition}

%More details can be found in~\cite[Chapter~2]{Sk-Phd}.

\begin{definition}\label{def:boundary-triple}
  Let $\fA_{0}\colon \dom(\fA_{0}) \subseteq \X \to \X$ be a densely defined, skew-symmetric, and closed operator on a Hilbert space $\X$.
  By a \emph{boundary triple} for $\fA_{0}\adjun$ we mean a triple $(\U, \fB_{1},\fB_{2})$ consisting of a Hilbert space $\U$, and two linear operators $\fB_{1},\fB_{2} \colon \dom \fA_{0}\adjun \to \U$ such that
  \begin{enumerate}
    \item\label{def:boundary-triple-surjective}
      the mapping
      $\begin{psmallmatrix}\fB_{1} \\ \fB_{2} \end{psmallmatrix} \colon \dom \fA_{0}\adjun \to \U\times \U, x \mapsto \begin{psmallmatrix}\fB_{1} x \\ \fB_{2} x\end{psmallmatrix}$
      is surjective, and

    \item\label{def:boundary-triple-equation}
      for $x,y \in \dom \fA_{0}\adjun$ there holds
      \begin{equation}\label{eq:boundary-triple}
        \scprod{\fA_{0}\adjun x}{y}_{X} + \scprod{x}{\fA_{0}\adjun y}_{X} = \scprod{\fB_{1} x}{\fB_{2} y}_{\U} + \scprod{\fB_{2} x}{\fB_{1} y}_{\U}.
      \end{equation}
  \end{enumerate}
\end{definition}

\noindent An important result for boundary triples is that maximally dissipative restrictions of $\fA_{0}\adjun$ can be charcaterized by maximally dissipative linear relations $\Theta$ on $\U$. The following theorem from \cite[Prop.~2.4.10]{Sk-Phd} will clarify that.
%Moreover, we can characterize m-dissipative linear relations on $\mathcal{B}$ by contractive linear operators on $\mathcal{B}$

\begin{theorem}\label{th:boundary-triple-maximal-dissipative}
  Let $(\U,\fB_{1},\fB_{2})$ be a boundary triple for $\fA_{0}\adjun$ and $\Theta$ be a maximally dissipative linear relation on $\U$. Then $\fA_{\Theta}$, which is the restriction of $\fA_{0}\adjun$ to
  \begin{align*}
    \dom \fA_{\Theta} = \dset*{x \in \dom \fA_{0}\adjun}{%
    \begin{psmallmatrix} \fB_{1}x \\ \fB_{2}x \end{psmallmatrix} \in \Theta
    }
  \end{align*}
  is a~maximally dissipative operator.
\end{theorem}

Since the systems considered in this article has finite-dimensional input and output spaces, our attention is directed towards linear relations in the form of $\Theta = \ker \begin{bmatrix}W_{1} & W_{2}\end{bmatrix}$, where $W_1$ and $W_2$ are square matrices. We establish a criterion for maximal dissipativity of such relations.%In this motivation for the above result is particular we want $\Theta$ to represent usual boundary conditions. This can be achieved by regarding $\Theta = \ker \begin{bmatrix}W_{1} & W_{2}\end{bmatrix}$ for $W_{1},W_{2} \in \Lb(\mathcal{B})$, because then
%\begin{equation*}
%  \begin{pmatrix}
%    B_{1}x \\ B_{2}x
%  \end{pmatrix}
%  \in \Theta
%  \quad\Leftrightarrow\quad
%  W_{1}B_{1}x + W_{2}B_{2}x = 0.
%\end{equation*}
%Clearly,  is not always m-dissipative. %However, we can easily characterize those %$W_{1}, W_{2}$ that induce a m-dissipative $%\Theta$. In particular we are interested in %the case $\mathcal{B} = \C^{k}$.

\begin{lemma}\label{th:W1-W2-condition-implies-max-dissipativ}
  Let $W_{1}, W_{2} \in \C^{\ell\times\ell}$. Then, for $W \coloneq \begin{bmatrix}W_1 & W_2\end{bmatrix}$, the relation $\ker W$ is maximally dissipative, if $W$ has full row rank and $W_{1}W_{2}\adjun + W_{2}W_{1}\adjun \geq 0$.
\end{lemma}

\begin{proof}
The relation $\ran\sbvek{W_2\adjun}{-W_1\adjun}$ is maximally dissipative by \cite[Lem.~3.5]{GeReHa21} and, further, it
is, in the sense of \cite[Def.~1.3.1]{BeHaSn20}, the adjoint of $\ker W$. Then we can conclude from \cite[Prop.~1.6.7]{BeHaSn20} that $\ker W$ is maximally dissipative.
\end{proof}

\newcommand{\etalchar}[1]{$^{#1}$}

\end{document}